%% file: main.tex
\documentclass[preprint,3p,times,11pt]{elsarticle} 

\usepackage{amssymb}
\usepackage{lipsum}
\usepackage{amsmath}
\usepackage{cool}
\usepackage{mathtools}
\usepackage{dirtytalk}
\usepackage{cuted}
\usepackage{relsize}
\usepackage[font=small]{caption}
\usepackage{subcaption}
\usepackage{ragged2e}
\usepackage{ifthen}
\usepackage[normalem]{ulem}
\usepackage{xcolor}
\usepackage{url}
\usepackage[inkscapearea=page]{svg}
\usepackage{tabularx}
\usepackage{hyperref}
\usepackage{mathrsfs}
\usepackage{enumitem}
\usepackage{multirow}
\usepackage{stmaryrd}
\usepackage{algorithm, algorithmic}
\usepackage{diagbox}
\usepackage{microtype}

\newboolean{AnnotateChanges}
\setboolean{AnnotateChanges}{false} 

\ifthenelse{\boolean{AnnotateChanges}}{%
\definecolor{annotateColor}{RGB}{255,0,0}
}{%
\definecolor{annotateColor}{RGB}{0,0,0} 
}

\newcommand{\added}[2][]{%
    \ifthenelse{\boolean{AnnotateChanges}}{%
        \textcolor{annotateColor}{#2}\textbf{\textcolor{blue}{#1}}}{#2}}

\newcommand{\deleted}[2][]{%
    \ifthenelse{\boolean{AnnotateChanges}}{%
        \textcolor{annotateColor}{\sout{#2}}\textbf{\textcolor{blue}{#1}}}{}}

\newcommand{\replaced}[3][]{%
    \ifthenelse{\boolean{AnnotateChanges}}{%
        \textcolor{annotateColor}{#2\sout{#3}}\textbf{\textcolor{blue}{#1}}}{#2}}

\usepackage{stfloats}

\usepackage{amsthm}
\newtheorem{theorem}{Theorem}[section]
\newtheorem{proposition}[theorem]{Proposition}
\newtheorem{lemma}[theorem]{Lemma}
\theoremstyle{definition}
\newtheorem{assumption}[theorem]{Assumption}
\newtheorem{remark}[theorem]{Remark}

\usepackage{tikz}
\usepackage{tikzscale}
\usetikzlibrary{calc,intersections,arrows,shadings,patterns.meta,arrows.meta}

\usepackage{lineno}

\journal{}

\newcommand{\bTheta}{\boldsymbol{\Theta}}
\newcommand{\bnabla}{\boldsymbol{\nabla}}
\newcommand{\bn}{\boldsymbol{n}}
\newcommand{\bx}{\boldsymbol{x}}

\newcommand{\bq}{\boldsymbol{q}}

\newcommand{\Ab}{\boldsymbol{\mathrm{A}}}
\newcommand{\meant}[1]{\langle#1\rangle}
\newcommand{\mean}[1]{\{#1\}}
\newcommand{\jump}[1]{\llbracket#1\rrbracket}
\newcommand{\bN}{\boldsymbol{\mathrm{N}}}

\definecolor{mypurple}{RGB}{168, 121, 206}
\definecolor{myblue}{RGB}{110,170,190}
\definecolor{mygreen}{RGB}{197,206,142}
\definecolor{anotherblue}{RGB}{82,85,255}
\definecolor{myred}{RGB}{255,158,158}

\usepackage{float}
 \usepackage[newfloat,frozencache,cachedir={_minted}]{minted} 
\usepackage[utf8x]{inputenc}
\usepackage{textgreek}
\DeclareUnicodeCharacter{2218}{$\circ$}
\DeclareUnicodeCharacter{2299}{$\odot$}
\DeclareUnicodeCharacter{222B}{$\int$}
\DeclareUnicodeCharacter{22C5}{$\cdot$}
\DeclareUnicodeCharacter{2207}{$\nabla$}
\DeclareUnicodeCharacter{1D62}{$_i$}
\DeclareUnicodeCharacter{2081}{$_1$}
\DeclareUnicodeCharacter{2082}{$_2$}
\DeclareUnicodeCharacter{2083}{$_3$}
\DeclareUnicodeCharacter{2229}{$\cap$}

\makeatletter
\newcommand\fs@tbruled{%
  \def\@fs@cfont{\normalfont}%
  \let\@fs@capt\floatc@plain
  \def\@fs@pre{\hrule\kern2pt}%
  \def\@fs@post{\kern2pt\hrule\relax}%
  \def\@fs@mid{\kern2pt\hrule\kern2pt}
  \let\@fs@iftopcapt\iftrue
}
\makeatother

\floatstyle{tbruled}
\restylefloat{listing}

\begin{document}

\begin{frontmatter}

\title{Shape calculus and automatic differentiation for multi-phase level-set topology optimisation with unfitted finite elements}

\author[qut]{Zachary J. Wegert\corref{cor}}
\ead{wegert@qut.edu.au}
\author[umea]{Martin Berggren}
\ead{martin.berggren@cs.umu.se}
\author[qut]{Vivien J. Challis}
\ead{vivien.challis@qut.edu.au}

\cortext[cor]{Corresponding author}

\affiliation[qut]{organization={School of Mathematical Sciences, Queensland University of Technology},
            addressline={2 George St},
            city={Brisbane},
            postcode={4000},
            state={Queensland},
            country={Australia}}

\affiliation[umea]{organization={Department of Computing Science, Umeå University},
            city={SE-901 87 Umeå},
            country={Sweden}}

\begin{abstract}
We present shape calculus techniques and a scalable automatic shape differentiation framework for multi-phase topology optimisation on unfitted discretisations defined by several level-set functions. First we establish general, exact shape calculus expressions in the discrete case for multi-phase systems by leveraging concepts from convex geometry. To complement this theoretical foundation, we introduce an open-source multi-phase automatic shape differentiation framework based on polytopal cutting. This computational framework is validated against both finite differences and our established exact expressions, matching the latter to near machine precision. Furthermore, the proposed automatic shape differentiation is scalable across distributed computing environments, demonstrating near ideal weak scaling up to 1.65 billion finite elements across 13,824 computer cores. We demonstrate our implementation by solving unfitted multi-phase topology optimisation problems for anisotropic diffusion, linear elasticity, and fluid--structure interaction. Together, these theoretical and computational contributions provide a robust and accessible foundation for advancing multi-phase topology optimisation using unfitted finite element methods. In particular, the methods enable the solution of topology optimisation problems involving multi-phase and multi-physics systems with non-trivial boundary conditions. The open-source software is available at \url{https://github.com/zjwegert/GridapTopOpt.jl}.
\end{abstract}



\end{frontmatter}


\section{Introduction}\label{sec: Introduction}\noindent
New additive manufacturing technologies are enabling the fabrication of complex multi-material designs that were previously unattainable \citep{Yu_Griffis_Manogharan_Panesar_2025}. Across industries such as aviation and aerospace, these advances create significant opportunities for lead-time and cost reduction, for instance through weight reduction, without sacrificing design performance \citep{BlakeyMilner_Gradl_Snedden_Brooks_Pitot_Lopez_Leary_Berto_Plessis_2021}. Owing in part to the considerable growth in additive manufacturing, topology optimisation has become a key design tool in several engineering fields such as structural design, aerospace, and robotics \citep{Meng_Zhang_Quan_Shi_Tang_Hou_Breitkopf_Zhu_Gao_2020,Peng_Xiuli_Shaojing_Yufan_2025}. Topology optimisation enables the optimisation of material layouts to maximise performance, while satisfying physics constraints formulated as partial differential equations (PDEs). Broadly speaking, topology optimisation methods fall into two categories: density-based (also called material distribution) approaches \cite{Bendsoe89, Rozvanyetal1992} and level-set approaches \cite{10.1016/S0045-7825(02)00559-5_2003,10.1016/j.jcp.2003.09.032_2004}. In density methods, the design variables are typically material densities assigned to the elements or nodes of a mesh, whereas in level-set methods, the domain boundary is represented implicitly by a level-set function that is updated, for instance, according to an evolution equation.

Multi-phase topology optimisation facilitates the design of components with more than two distinct regions that have different material properties. In this context, both density and level-set approaches have been developed and studied. The predominant approach in the construction of these methods is to smoothly interpolate material properties between regions \citep{Peng_Xiuli_Shaojing_Yufan_2025}. In some cases, such as for functionally graded materials \citep{10.1007/s00158-017-1688-2_2017}, this interpolation of material properties is physically realistic. However, when modelling the situation of material properties that change sharply at the interface between the two phases, interpolating the material properties only approximates the interface conditions. Furthermore, the imposition of more complicated interface conditions in the case of multi-physics problems may be intractable. For this reason, several authors have recently developed level set-based body-fitted mesh approaches for multi-phase topology optimisation. \citet{10.1051/cocv/2013076_2014} considered the sensitivity analysis for two-phase problems involving a boundary-conforming mesh. \citet{Liu_Shi_Kang_2020} developed a body-fitted level-set method for two-dimensional bi-material elastic structures. \citet{NguyenTran_Phan_NguyenThoi_Nguyen_2025} combined a two-dimensional body-fitted `multi-trimmed' mesh technique with a reaction-diffusion-based level-set formulation for elastic structures. \citet{Kosta_Amir_2025} and \citet{Kosta_Shakur_Amir_2026} developed a multi-material isogeometric analysis method for two-dimensional elastic structures. These works are restricted to two-dimensional systems, and significant challenges persist in the tractable generation of more than one connected boundary-conforming mesh for multi-phase three-dimensional problems.
Furthermore, the shape calculus is problem specific, poses significant difficulty, and often requires approximation. A versatile and systematic approach for the topology optimisation of general sharp-interface multi-material problems is currently absent from the literature.

Unfitted finite element methods are powerful techniques for solving PDEs over complicated physical domains that overcome the need to generate a boundary-conforming mesh \citep{10.1002/nme.4823_2015,Badia_Verdugo_Martin_2018}. For domains defined using a level-set function, this typically works by computing a non-conforming triangulation of the boundary by intersecting the zero level curve with the background triangulation. Integration of the variational form of a PDE over the interface or cut cells can then be carried out using the resulting cut-cell triangulation. Ill-conditioned linear systems can arise due to arbitrarily small cuts \citep{dePrenter2023}, and therefore several schemes have been developed to address this issue. The approaches include ghost penalty stabilisation (CutFEM,  \citep{10.1002/nme.4823_2015}) and discrete extension operators (AgFEM, \citep{Badia_Verdugo_Martin_2018}). Unfitted finite element methods have been applied to solve a wide range of PDE problems, including modelling fluid--structure interaction \citep{Schott_Ager_Wall_2019}, acoustics \citep{10.1002/nme.5621_2018}, and electrophysiology \citep{Berre_Rognes_Massing_2024}.
The extension of unfitted finite element methods to multi-phase PDE problems typically relies on recursively cutting the background mesh and adding appropriate jump conditions on the distinct phase interfaces \citep{10.1002/nme.4823_2015,Claus_Kerfriden_2018,10.1016/j.cma.2019.01.009_2019,Gammanpila_Aulisa_Chierici_2025}.

Unfitted finite element methods have been used for solving single-phase level-set topology optimisation problems, including those involving linear elasticity \citep{10.1016/j.cma.2017.09.005_2018,WEGERT2025118203}, fluid flow \citep{10.1016/j.cma.2017.03.007_2017,WEGERT2025118203}, and acoustics \cite{10.1002/nme.5621_2018}. However, for multi-phase systems, a general framework for sensitivity analysis to enable gradient-based optimisation is lacking.
\citet{10.1051/cocv/2013076_2014} considered the sensitivity analysis of two-phase problems using Hadamard's domain transformation approach. This provides a clear framework for the case of two-phase problems.
However, generalising the approach to an arbitrary number of phases defined by several level-set functions remains a difficult issue. As a result, the use of multi-phase unfitted finite element methods for topology optimisation has not been considered in the literature.

In recent work by \citet{Berggren_2023} and the extension by \citet{WEGERT2025118203}, a `discretise-then-differentiate' shape calculus approach was proposed for single-phase problems based on the concept of Delfour dilations \citep{Delfour_2018}. This can be regarded as a `discretise-then-differentiate' approach because it directly relies on locally parametrising the dilated region under a perturbation of a level-set function defined on a fixed mesh. Here we generalise the results of \citet{Berggren_2023} and \citet{WEGERT2025118203}
to arbitrary convex polytopes that arise in the case of multiple intersecting domains defined by level-set functions.
The resulting shape calculus results make it possible to use gradient-based optimisation with multiple level-set functions.

It is also possible to compute shape derivatives using automatic differentiation. Automatic differentiation enables the exact recovery of derivatives to machine precision by applying the chain rule, in which derivatives are propagated from inputs to outputs (forward-mode) or from outputs to inputs (reverse-mode) \citep{Griewank_Walther_2008}. In the context of the single-phase shape calculus of \citet{Berggren_2023}, automatic shape differentiation was proposed by \citet{WEGERT2025118203}. In that work, it was shown that automatic shape differentiation techniques were able to recover the shape calculus results of \citet{Berggren_2023} to machine precision. This significantly reduces the burden associated with computing shape derivatives, particularly for complicated multi-physics problems. Here we consider the extension of the automatic shape differentiation techniques of \citet{WEGERT2025118203} to the case of several intersecting domains defined by level-set functions.

We implement multi-phase automatic shape differentiation for unfitted discretisations in the Julia package GridapTopOpt \cite{GridapTopOpt} and the wider Gridap-package ecosystem \cite{Badia2020,Verdugo2022}. Our implementation utilises GridapDistributed \cite{Badia2022} so that it can be leveraged in both serial and distributed computing frameworks on central processing units (CPUs). Like Gridap, our implementation aims for a clear and intuitive application programming interface (API). Furthermore, our automatic differentiation implementation is fully compatible with GridapTopOpt's automatic differentiation of arbitrary PDE-constrained maps. Providing this open-source software will facilitate the solution of complex multi-physics and multi-phase topology optimisation problems by the research community.

The remainder of the article is as follows. In Section \ref{sec: multi-phase unfitted fes}, we introduce the reader to the concept of multi-phase unfitted finite elements to provide context for our results. In Section \ref{sec: ls sd} we extend the shape calculus of \citet{Berggren_2023} and \citet{WEGERT2025118203} to the case of intersecting domains defined by multiple level-set functions. The key mathematical results are given in Theorems~\ref{theorem 1}, \ref{theorem 2}, and \ref{theorem 3}. In Section \ref{sec: asd}, we discuss automatic differentiation for multi-phase unfitted discretisations and provide scalability benchmarks for our automatic differentiation implementation. In Section \ref{sec: Level-set topology optimisation}, we briefly discuss our implementation of multi-phase unfitted level-set topology optimisation. In Section \ref{sec: Examples}, we consider exemplar topology optimisation problems that showcase the capabilities provided by our approach. Finally, in Section \ref{sec: Conclusions} we present our concluding remarks.

\section{Multi-phase unfitted finite elements}\label{sec: multi-phase unfitted fes}\noindent
To illustrate the concept of multi-phase unfitted finite element methods, we will consider the construction of a variational formulation of a multi-phase diffusion problem on which the discretisation will be based, and for which we will parametrise the subdomains using level-set functions.

\subsection{A multi-phase diffusion boundary-value problem}\noindent
Consider a hold-all \textit{background} domain $D\subset\mathbb{R}^d$, in which we consider an open and bounded domain $\Omega\subset D$ that can be subdivided into two non-overlapping subdomains $\Omega_1$ and $\Omega_2$, that is, $\overline\Omega=\overline\Omega_1\cup\overline\Omega_2$, and  denote the nonempty interface between these by $\Gamma_{12}=\partial\Omega_1\cap\partial\Omega_2$.
On $\Omega$, consider the scalar field $u$ and the vector flux function $\bq(u)$, defined locally on $\Omega_i$, $i=1,2$, and denote $\bq_i(u_i)\coloneqq\bq(u)\rvert_{\Omega_i}$, where $u_i\coloneqq u\rvert_{\Omega_i}$.
These fields satisfy the boundary-value problem
\begin{subequations}\label{eqn: diff}
\begin{align}
        -\bnabla\cdot\bq(u) &= 0 \quad \text{ in } \Omega, \label{eqn: diff omega}\\
        \bq(u)\cdot\bn &= g \quad \text{ on } \Gamma_N, \label{eqn: diff gammaN}\\
        \bq(u)\cdot\bn &= 0 \quad \text{ on } \partial\Omega\setminus(\Gamma_N\cup\Gamma_D),\label{eqn: diff noflux}\\
        u &= 0 \quad \text{ on } \Gamma_D\label{eqn: diff dirich},
\end{align}
\end{subequations}
where a homogeneous Dirichlet condition for $u$ is given on one part of the boundary and  homogeneous and inhomogeneous fluxes are defined on the rest of the boundary. In this formulation, we assume that $\Gamma_N\subset\partial D$ and $\Gamma_D\subset\partial D$.
On the phase interface,
we impose the transmission conditions
\begin{subequations}\label{eqn: jumps}
\begin{align}
    \jump{u} &= u_1\bn_1 + u_2\bn_2= \boldsymbol 0\quad\text{ on }\Gamma_{12},\label{eqn: diff jump}\\
    \jump{\bq(u)} &= \bq_1(u_1)\cdot\bn_1 + \bq_2(u_2)\cdot\bn_2 = 0 \quad\text{ on }\Gamma_{12},\label{eqn: diff flux jump}
\end{align}
\end{subequations}
where $\bn_i$, $i=1,2$, is the outward-directed normal field on $\partial\Omega_i\cap\Gamma_{12}$ satisfying $\bn_1=-\bn_2$.
Here, we have adopted the jump convention of Arnold et al.~\cite[\S\,3.1]{ArBrCoMa02}.
Note that under this convention, the jump of a scalar function is a vector and the jump of a vector function is a scalar.
The advantage with this notation is that no unique orientation needs to be assigned to $\Gamma_{12}$ to define the jumps.

In this general discussion, the relation between the scalar unknown $u$ and the flux function $\bq(u)$ is unspecified.
Later we will consider the case of anisotropic multi-phase diffusion with $\bq_i=\Ab_i\bnabla u_i$ where $\Ab_i$ are symmetric diffusion tensors.

To derive the variational formulation, we begin by defining the spaces $\mathcal{V}_i=\{v\in H^1(\Omega_i):v\rvert_{\Gamma_D}=0\}$. Next, separate \eqref{eqn: diff omega} over each partition $\Omega_i$, multiply by a test function $v_i\in \mathcal{V}_i$, sum the contributions, and integrate by parts. This gives
\begin{equation}
    0=\sum_{i}(-\bnabla\cdot\bq_i(u_i), v_i)_{\Omega_i}=\sum_{i}(\bq_i(u_i),\bnabla v_i)_{\Omega_i}-\sum_i(\bq_i(u_i)\cdot\bn_i,v_i)_{\partial\Omega_i},
\end{equation}
where we use the notation $(u,v)_\omega=\int_{\omega}u\cdot{v}~\mathrm{d}\omega$ for integrals over domains, boundaries, or interfaces $\omega$.
Note that the integrand includes a dot product for vector arguments $u$, $v$.
Substituting boundary conditions~\eqref{eqn: diff gammaN}--\eqref{eqn: diff dirich} yields that
\begin{equation}\label{eqn: inter 1}
    0=\sum_{i}(\bq_i(u_i),\bnabla v_i)_{\Omega_i}-\sum_i(g,v_i)_{\Gamma_N}-(\bq_1\cdot\bn_1,v_1)_{\Gamma_{12}}-(\bq_2\cdot\bn_2,v_2)_{\Gamma_{12}}.
\end{equation}
Defining the weighted mean operators $\mean{u}=\kappa_1u_1+\kappa_2u_2$ and $\meant{u}=\kappa_2u_1+\kappa_1u_2$ with $\kappa_1+\kappa_2=1$, a substitution shows that $\jump{\bq v}=\jump{v}\cdot\mean{\bq} + \meant{v}\jump{\bq}$.
Thus, \eqref{eqn: inter 1} can be rewritten as
\begin{equation}
    0=\sum_{i}(\bq_i(u_i),\bnabla v_i)_{\Omega_i}-\sum_i(g,v_i)_{\Gamma_N}-(\mean{\bq(u)},\jump{v})_{\Gamma_{12}}-(\jump{\bq(u)},\meant{v})_{\Gamma_{12}},
\end{equation}
and using the interface jump condition~\eqref{eqn: diff flux jump}, it follows that
\begin{equation}\label{eqn: varform1}
    0=\sum_{i}(\bq_i(u_i),\bnabla v_i)_{\Omega_i}-\sum_i(g,v_i)_{\Gamma_N}-(\mean{\bq(u)},\jump{v})_{\Gamma_{12}}.
\end{equation}

From here, we enforce the continuity condition in \eqref{eqn: diff jump} using a symmetric Nitsche method \citep{10.1002/nme.4823_2015}. 
This method relies on adding, to the variational form, a symmetry term $(\mean{\bq(v)},\jump{u})_{\Gamma_{12}}$ and a penalty term $(\mu\jump{u},\jump{v})_{\Gamma_{12}}$,
where $\mu$ is a penalty parameter.
These are consistent additions in the sense that $\jump{u}=0$ in the case of $u$ being a solution to boundary value problem~\eqref{eqn: diff} subject to transmission conditions~\eqref{eqn: jumps}. 

Adding the the symmetry and penalty term to variational expression~\eqref{eqn: varform1} yields the following variational formulation of the boundary value problem~\eqref{eqn: diff} under jump conditions~\eqref{eqn: jumps}:
\begin{equation}\label{eqn: cont var form}
\begin{aligned}
&\text{Find $(u_1,u_2)\in\mathcal{V}_1\times\mathcal{V}_2$ such that}
\\ &\qquad
    a([u_1,u_2],[v_1,v_2])=l([v_1,v_2]) &\forall (v_1,v_2)\in\mathcal{V}_1\times\mathcal{V}_2,
\end{aligned}
\end{equation}
where
\begin{equation}\label{eqn: cts bilinear}
    \begin{aligned}
        a([u_1,u_2],[v_1,v_2]) &= \sum_{i}(\bq_i(u_i),\bnabla v_i)_{\Omega_i}
        -(\mean{\bq(u)},\jump{v})_{\Gamma_{12}}
        -(\mean{\bq(v)},\jump{u})_{\Gamma_{12}}
        +(\mu\jump{u},\jump{v})_{\Gamma_{12}},
    \end{aligned}
\end{equation}
and
\begin{equation}\label{eqn: cts linear}
    l([v_1,v_2]) = \sum_i(g,v_i)_{\Gamma_N}.
\end{equation}

\subsection{Discrete formulation}\noindent
Suppose that we partition the background domain $D$ by cells $K\in\mathcal{T}_h$ where $\mathcal{T}_h$ is the set of cells. We denote the facets of $\mathcal{T}_h$ to be $\mathcal{F}_h$. Suppose that $D$ is partitioned by two level-set functions $(\phi_1,\phi_2)\in \mathcal{V}_{h}^2$ such that
\begin{equation}
        D_{\phi_i}=\{\bx\in D:\phi_i(\bx)<0\},\quad D\setminus \overline{D}_{\phi_i}=\{\bx\in D:\phi_i(\bx)>0\},
        \quad \partial D_{\phi_i}=\{\bx\in D:\phi_i(\bx)=0\},
\end{equation}
for $i=1,2$, where $\mathcal{V}_{h}$ is the space of continuous, piecewise linear functions on $\mathcal{T}_h$:
\begin{equation}
    \mathcal{V}_{h} = \{v\in C^0(\mathcal{T}_h):v\rvert_{K}\in\mathbb{P}_1(K)~\forall K\in\mathcal{T}_h\}.
\end{equation}
Using set operations on $D_{\phi_1}$ and $D_{\phi_2}$, we can define four non-overlapping subdomains. Figure~\ref{fig:fig2-new} shows an example of this.
\begin{figure}[!t]
    \centering
    \def\svgwidth{\textwidth}
    \input{Figure2_new_tex}
    \caption{An example of non-overlapping subdomains $\Omega_j$ defined by set operations on the domains $D_{\phi_i}$ that are generated by two level-set functions $\phi_i$. The interfaces between each subdomain $\Omega_i$ can be written as $\Gamma_{ij}=\partial\Omega_i\cap\partial\Omega_j$ for $i\neq j$.}
    \label{fig:fig2-new}
\end{figure}
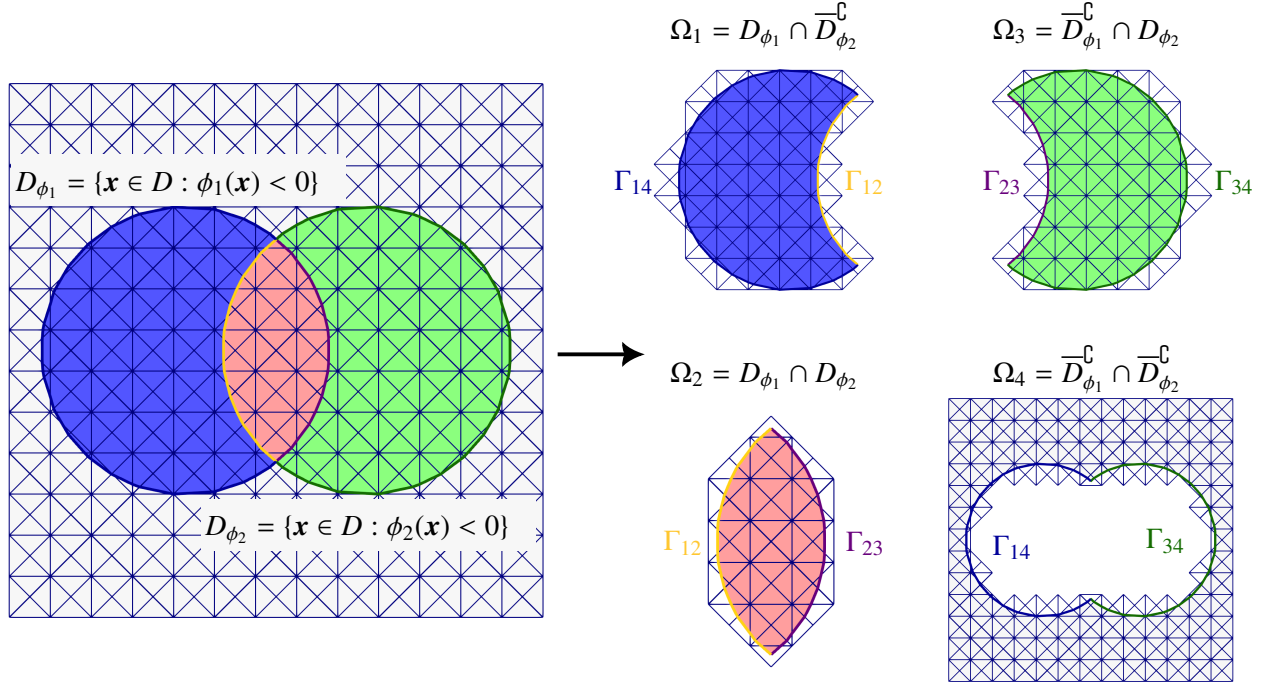

For the purpose of this demonstration, we consider parametrising $\Omega_1$ and $\Omega_2$ as $\Omega_1(\phi_1,\phi_2)=D_{\phi_1}\cap \overline{D}_{\phi_2}^\complement$ and $\Omega_2(\phi_1,\phi_2)=D_{\phi_1}\cap D_{\phi_2}$, respectively. As a result, $\Gamma_{12}$ can be written as $\Gamma_{12}(\phi_1,\phi_2)=D_{\phi_1}\cap\partial D_{\phi_2}$.

Suppose that we wish to discretise using CutFEM over the mesh shown in Figure~\ref{fig:fig2-new}.
The CutFEM is based on variational form~\eqref{eqn: cont var form} to which is added ghost penalty terms to control the condition number of the system matrix~\cite{Bu10}.
In this case, we add a ghost penalty term over the \textit{ghost skeleton} $\mathcal{F}_{G_i}$ of each of $\Omega_i$, where $\mathcal{F}_{G_i}$ is defined as follows: for distinct cells $K_1\in\mathcal{T}_h$ and $K_2\in\mathcal{T}_h$, a facet $F\in\mathcal{F}_{G_i}$ is given by $F=K_1\cap K_2$ where at least one of $K_1$ or $K_2$ intersect the interface $\partial\Omega_i$ \citep{10.1002/nme.4823_2015}. In particular, we add terms of the form
\begin{equation}
    j([u_1,u_2],[v_1,v_2]) = \sum_i\sum_{F\in\mathcal{F}_{G_i}}(\eta_{i}(h)\jump{\partial_{\boldsymbol{n}_F}u_i},\jump{\partial_{\boldsymbol{n}_F}v_i})_F,
\end{equation}
where $\eta_{i}(h)$ are penalty parameters.

The CutFEM approximation of boundary-value problem~\eqref{eqn: diff}--\eqref{eqn: jumps} is then:
\begin{equation}
\begin{aligned}
&\text{Find $(u_1,u_2)\in\mathcal{V}_{1,h}\times\mathcal{V}_{2,h}$ such that}
\\ &\qquad
A([u_1,u_2],[v_1,v_2])=l([v_1,v_2]) \qquad\forall (v_1,v_2)\in\mathcal{V}_{1,h}\times\mathcal{V}_{2,h},
\end{aligned}
\end{equation}
with
\begin{equation}
    A([u_1,u_2],[v_1,v_2]) = a([u_1,u_2],[v_1,v_2])+j([u_1,u_2],[v_1,v_2]),
\end{equation}
where $a$ and $l$ are defined in \ref{eqn: cts bilinear} and \eqref{eqn: cts linear}, respectively, and $\mathcal{V}_{i,h} = \{v\in C^0(\mathcal{T}_{i,h}):v\rvert_{K}\in\mathbb{P}_1(K)~\forall K\in\mathcal{T}_h,~v\rvert_{\Gamma_D}=0\}$, in which $\mathcal{T}_{i,h}$ are the active elements associated with phase $i$ given by
\begin{equation}
    \mathcal{T}_{i,h}=\{K\in\mathcal{T}_h:\Omega_i\cap K\neq\emptyset\}.
\end{equation}

Now that we have constructed a discrete forward problem,  we will in the next section consider how to compute derivative information under perturbations of $\phi_1$ and $\phi_2$ using level-set shape calculus.

\section{Shape differentiation via level-set dilation}\label{sec: ls sd}\noindent
In the following, we extend the shape calculus of \citet{Berggren_2023} to the case of intersecting domains defined by level-set functions. There are three types of directional derivatives under a perturbation of a level-set function that must be considered:
\begin{itemize}
    \item[(i)] \textit{Domain integrals} (Theorem~\ref{theorem 1}), that is, the directional derivative of a functional defined over the intersection of several domains, where the level-set function being perturbed defines one of these domains. An example is a functional defined on $\Omega_1 = D_{\phi_1}\cap  \overline{D}_{\phi_2}^\complement$ in Figure~\ref{fig:fig2-new} under a perturbation of $\phi_1$ or $\phi_2$.
    \item[(ii)] \textit{Boundary integrals truncated by a perturbed domain} (Theorem~\ref{theorem 2}), that is, the directional derivative of a functional defined over the boundary of a domain that is intersected with several other domains, and where the level-set function that is being perturbed defines one of the domains intersecting the boundary. An example is a functional defined on $\Gamma_{23}=\partial D_{\phi_1}\cap D_{\phi_2}$ in Figure~\ref{fig:fig2-new} under a perturbation of $\phi_2$.
    \item[(iii)] \textit{Boundary integrals} (Theorem~\ref{theorem 3}). Here, the functional is of the same type as for (ii); however the level-set function that is being perturbed to compute the directional derivative is the one that defines the boundary. An example is a functional defined on $\Gamma_{23}=\partial D_{\phi_1}\cap D_{\phi_2}$ in Figure~\ref{fig:fig2-new} under a perturbation of $\phi_1$.
\end{itemize}
Our results for Cases (i) and (iii) reduce to those of \citet{Berggren_2023} and \citet{WEGERT2025118203} for the case of a single level-set function. Case (ii) is new and appears only in the case of two or more intersecting level-set functions. 
We restrict the scope to simplex-type meshes in $\mathbb{R}^d$.

\subsection{Preliminaries}\noindent
For clarity, we restate some of the previously used notation. Let $D\subset\mathbb{R}^d$ be a hold-all domain, to be referred to as the background domain.
We partition $D$ by simplices $K\in\mathcal{T}_h$, where $\mathcal{T}_h$ is the set of cells, and denote the facets of $\mathcal{T}_h$ to be $\mathcal{F}_h$. Note that the facets are of codimension one. We consider the following space of level-set functions that partition $D$:
\begin{equation}
    \mathcal{V}_{h} = \{v\in C^0(\mathcal{T}_h):v\rvert_{K}\in\mathbb{P}_1(K)~\forall K\in\mathcal{T}_h\}.
\end{equation}
Assume now that $D$ is partitioned by a level-set function $\phi_i\in \mathcal{V}_{h}$ such that
\begin{equation}
        D_{\phi_i}=\{\bx\in D:\phi_i(\bx)<0\},\quad D\setminus \overline{D}_{\phi_i}=\{\bx\in D:\phi_i(\bx)>0\},
        \quad \partial D_{\phi_i}=\{\bx\in D:\phi_i(\bx)=0\},
\end{equation}
and consider perturbations of the form
\begin{equation}\label{eqn: perturbation}
    \phi_{i,t}(\bx)=\phi_i(\bx)+tw(\bx),
\end{equation}
where $t\geq0$ and $w$ is a linear Lagrangian basis function of $\mathcal{V}_{h}$. We define the associated perturbed domains
\begin{equation}
    D_{\phi_i,t}=\{\bx\in D:\phi_{i,t}(\bx)<0\}.
\end{equation}
Note that by construction $D_{\phi_i,t}\subseteq D_{\phi_i,t^\prime}$ for $t^\prime\leq t$, since $w\geq0$ by convention.
In other words, because of our choice of perturbation and how the level-set function defines $D_{\phi_i}$, the perturbed domain is always shrinking for increased $t$ \citep{Berggren_2023}.

Before proceeding, we recall the following assumptions from \citet{Berggren_2023}:
\begin{assumption}\label{assumption 1} There is a $t_\mathrm{max}>0$ such that for each $K\in\mathcal{T}_h$, if $\partial D_{\phi_{i,t}}\cap K$ is either empty or non-empty for some $t\in(0,t_\mathrm{max}]$, it has the same property for each $t\in(0,t_\mathrm{max}]$.
\end{assumption}
This is an assumption on the smallness of the perturbation, motivated by the fact that we will in the end be taking limits as $t\to0^+$.
\begin{assumption}\label{assumption 2} The boundary $\partial D_{\phi_i}$ does not intersect any mesh points of the triangulation $\mathcal{T}_h$.
\end{assumption}
Under this assumption, the semiderivatives will agree in the limits $t\rightarrow0^+$ and $t\rightarrow0^-$, respectively.

\subsection{The conceptual approach}\noindent
To illustrate the conceptual approach of level-set shape differentiation for multiple level-set functions, suppose we consider a functional of the form
\begin{equation}
    J_1(\phi_1)=\int_{D_{\phi_1}\cap D_{\phi_2}}f(\bx)~\mathrm{d}\bx
\end{equation}
where $\phi_2$ is fixed. For $t>0$ sufficiently small, it holds that
\begin{equation}\label{eqn: dom integral concept}
    \frac{1}{t}\bigl(J_1(\phi_{1,t})-J_1(\phi_1)\bigr)=\frac{1}{t}\int_{D_{\phi_{1,t}}\cap D_{\phi_2}}f(\bx)~\mathrm{d}\bx-\frac{1}{t}\int_{D_{\phi_1}\cap D_{\phi_2}}f(\bx)~\mathrm{d}\bx=-\frac{1}{t}\int_{E_t}f(\bx)~\mathrm{d}\bx,
\end{equation}
where, since $(D_{\phi_{1,t}}\cap D_{\phi_2})\subset (D_{\phi_1}\cap D_{\phi_2})$,  $E_t = (D_{\phi_1}\cap D_{\phi_2})\setminus (D_{\phi_{1,t}}\cap D_{\phi_2})$.
Moreover, in the algebra of subsets $A$, $B$, $C$ of a universe $D$, it holds that $(A\cap B)\setminus(C\cap B) = (A\setminus C)\cap B$.
Thus, $E_t=(D_{\phi_{1}}\cap D_{\phi_2})\setminus(D_{\phi_1,t}\cap D_{\phi_2}) = (D_{\phi_1}\setminus{D_{\phi_{1,t}}})\cap D_{\phi_2}$. Figure~\ref{fig:Et_fig} shows an example region $E_t$ in two dimensions.
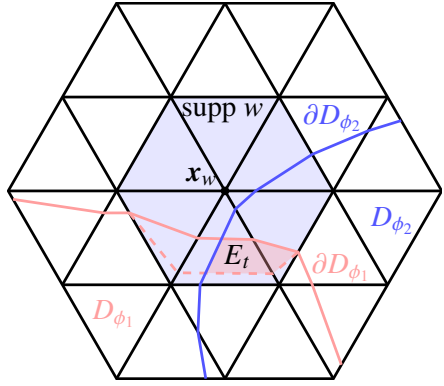
\begin{figure}[!t]
    \centering
    \def\svgwidth{0.35\textwidth}
    \input{Et_two_ls_tex}
    \caption{An example region $E_t=(D_{\phi_1}\setminus{D_{\phi_{1,t}}})\cap D_{\phi_2}$ under a perturbation $\phi_{1,t}=\phi_1+tw$. We denote the support of a linear Lagrangian basis $w$ located at $\bx_w$ by the shaded blue region. Here we only depict two domains defined by level-set functions. In the rest of the article, we consider the case of an arbitrary number of domains.}
    \label{fig:Et_fig}
    \end{figure}
Broadly speaking, an approach to compute the directional derivative $\mathrm{d}J_1(\phi_1)(w)$ is to: (i) parametrise $E_t$; (ii) express the final term in \eqref{eqn: dom integral concept} as a kind of coarea formula (that is, an integral of surface integrals); and (iii) take the limit as $t\rightarrow0$.



Without additional mathematical machinery, the number of cases needed to derive the parametrisation of $E_t$ and the corresponding coarea formula is combinatorial in the number of level-set functions. In the below we therefore borrow some concepts from convex geometry to generalise the approach of \citet{Berggren_2023} to arbitrary convex polytopes.

\subsection{Domain integrals}\label{subsec:dom int}\noindent
In order to prove a general result for several intersecting domains, we will consider the generalisation of the results of \citet{Berggren_2023} to arbitrary convex polytopes $\mathcal{P}\subseteq K\in\mathcal{T}_h$.
Let $V(\mathcal P)$ denote the extreme points of convex polytope $\mathcal P$.
We begin with the following results based on concepts in \citet{Gruber_2007}.
\begin{proposition}\label{prop: 1}
    An intersection plane $\Pi$ cuts a convex polytope $\mathcal{P}$ into two convex polytopes.
\end{proposition}
\begin{proof}
    Suppose we have a convex polytope $\mathcal{P}$ generated by a finite set of vertices $V(\mathcal{P})$. By the Minkowski--Weyl theorem \citep[Theorem 14.2,][]{Gruber_2007}, $\mathcal{P}$ has a halfspace representation
    \begin{equation}
        \mathcal{P}=\bigcap_{i=1}^mH_i,
    \end{equation}
    where $\{H_i\}_{i=1}^m$ is a finite set of halfspaces. Note that halfspaces are defined as one of the two convex sets generated by partitioning $\mathbb{R}^d$ with a hyperplane, that is,
    \begin{equation}
        H_i=\{\bx\in\mathbb{R}^d:\bx\cdot  \boldsymbol{\alpha}_i<c_i\}.
    \end{equation}
    Now, suppose that $\Pi$ divides $\mathbb{R}^d$ into halfspaces $H^+_\Pi$ and $H^-_\Pi$. Then, the resulting polytopes are
    \begin{equation}
        \mathcal{P}_\pm=\mathcal{P}\cap H^\pm_\Pi.
    \end{equation}
    Since $\mathcal{P}$ and $H^\pm_\Pi$ are convex, their intersection is convex.
\end{proof}

\begin{lemma}\label{lemma: lemma 2}
    Given a convex polytope $\mathcal{P}$ and halfspaces $H^\pm$ generated by the plane $\Pi$ cutting $\mathcal{P}$, there exists $\hat{\bx}_w\in V(\mathcal{P}\cap H^\pm)$  and $S\subset\partial(\mathcal{P}\cap H^\mp)$ such that all points $\bx\in\mathcal{P}\cap H^\mp$ can be parametrised by $\bx=\hat{\bx}_w+\sigma(\bx_S-\hat{\bx}_w)$ for $\sigma\in(0,1)$ and $\bx_S\in S$.
\end{lemma}
\begin{proof}
    The proof follows by convexity of $\mathcal{P}\cap H^\pm$ from Proposition~\ref{prop: 1}. In particular, by convexity, all points $\bx\in \mathcal{P}\cap H^\mp$ can be connected by a ray to a point $\hat{\bx}_w\in V(\mathcal{P}\cap H^\pm)$. Extending this ray to the boundary $\partial(\mathcal{P}\cap H^\mp)$ recovers the set $S$.
\end{proof}

A conceptual picture of the statement in Lemma~\ref{lemma: lemma 2} is shown in Figure~\ref{fig: concept a} for the case of $\hat{\bx}_w\in V(\mathcal{P}\cap H^+)$ and $S\subset\partial(\mathcal{P}\cap H^{-})$.

\begin{figure}[t]
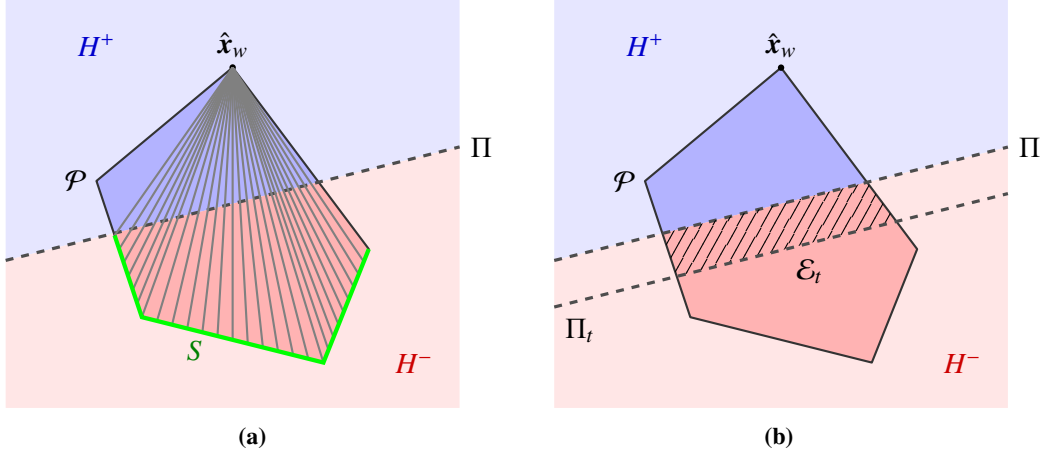

    \centering
    \begin{subfigure}{0.4\textwidth}
        \centering
        \includegraphics[width=\textwidth]{half_no_anim_end.tikz}
        \caption{}
        \label{fig: concept a}
    \end{subfigure}
    \begin{subfigure}{0.435\textwidth}
        \centering
        \includegraphics[width=\textwidth]{half_param_end.tikz}
        \caption{}
        \label{fig: concept b}
    \end{subfigure}
    \caption{(a) A conceptual picture of the statement in Lemma~\ref{lemma: lemma 2}. (b) A visualisation of the quantities $\Pi_t$ and $\mathcal{E}_t=(H^-\setminus\overline{H^-_t})\cap\mathcal{P}$ where $\Pi_t$ satisfies $\Pi_t\rvert_{K}=\{\bx\in K:\phi_t(\bx)=0\}$.} 
    \label{fig: concept}
\end{figure}

Suppose that we now specify $\Pi$ \textit{locally} in a polytope $\mathcal{P}\subseteq K\in\mathcal{T}_h$ using a level-set function $\phi(\bx)=0$. Without lack of generality, we assume that $\Pi_t$, defined such that $\Pi_t\rvert_{K}=\{\bx\in K:\phi_t(\bx)=0\}$, is \textit{pushed} into the halfspace $H^-$ generated by $\Pi$ for $t\geq0$. We can then define the perturbed region to be $\mathcal{E}_t=(H^-\setminus\overline{H^-_t})\cap\mathcal{P}$, which is itself convex. Figure~\ref{fig: concept b} shows a conceptual picture of this setup.

We require extensions of Assumption~\ref{assumption 1} and \ref{assumption 2} because we are working on a general polytope $\mathcal{P}\subseteq K$. In the first, we restrict the perturbation to be sufficiently small such that the following condition holds.
\begin{assumption}\label{assumption 3}
    There is a $t_\mathrm{max}>0$ such that, for the convex polytope $\mathcal{P}\subseteq K\in\mathcal{T}_h$, the geometric elements (that is, facets, edges, and vertices for $d=3$) of $H^\pm\cap \mathcal{P}$ are topologically equivalent to those of $H^\pm_t\cap \mathcal{P}$ for all $t\in(0,t_\mathrm{max}]$.
\end{assumption}
Next, we introduce an assumption on the intersection of multiple domain boundaries. 
\begin{assumption}\label{assumption 4}
    For a subfacet of dimension $d-2$ (that is, a point for $d=2$ or an edge for $d=3$) and where level-set function $\phi_k$ is perturbed, $\partial D_{\phi_k}$ intersects at most one other $\partial D_{\phi_i}$ along the subfacet, and any intersection does not occur on the boundary of a cell $\partial K$.
\end{assumption}

We can now establish the existence of a parametrisation of $\mathcal{E}_t$.
\begin{lemma}\label{lemma: lemma 3}
    Consider a moving plane $\Pi_t$ that is locally parametrised by $\phi_{t}(\bx)=\phi(\bx)+tw(\bx)=0$ for some $t \in (0, t_{max}]$, according to Assumption~\ref{assumption 3}, and where $w$ is a linear Lagrangian basis function.
    Suppose $K\in\mathcal{T}_h$ is a cell in the support of the basis function $w$ and that $\Pi_t$ cuts a convex polytope $\mathcal{P}\subseteq K$ into two sub-polytopes $\mathcal{P}\cap H_{t}^\pm$, where $H_{t}^\pm$ are the halfspaces defined by $\Pi_t$. Then, there exists a set $S$ and a diffeomorphism $\hat{\boldsymbol{X}}:S\times(0,t)\rightarrow \mathcal{E}_t$ such that $\hat{\boldsymbol{X}}(S,\tau)=\Pi_\tau\cap\mathcal{P}$ for each $\tau\in(0,t)$ where $\mathcal{E}_t=(H^-\setminus\overline{H^-_t})\cap\mathcal{P}$.
\end{lemma}
\begin{proof}
    By Lemma~\ref{lemma: lemma 2}, there exists $\hat{\bx}_w\in V(\mathcal{P}\cap H_{0}^+)$ and $S\subset\partial(\mathcal{P}\cap H_{\tau}^-)$ such that all points $\bx_\tau\in\Pi_\tau\cap\mathcal{P}$ can be parametrised by
    \begin{equation}
        \bx_\tau=\hat{\bx}_w+\sigma(\bx_S-\hat{\bx}_w)
    \end{equation}
    for $\sigma\in(0,1)$ and $\bx_S\in S$. Note that $\hat{\bx}_w$ is fixed for all $t\in(0,t_{\rm max})$ by Assumption~\ref{assumption 3}.

    We may then proceed in a similar way to \citet[Lemma~6.5]{Berggren_2023}.
    Since $\phi_\tau$ vanishes on $\Pi_\tau$ we have
    \begin{equation}\label{eqn: 3nef25fs}
        0=\phi_\tau(\bx_\tau)=\phi_\tau(\hat{\bx}_w+\sigma(\bx_S-\hat{\bx}_w))=\phi_\tau(\hat{\bx}_w)+\sigma \bnabla\phi_\tau\rvert_{K}\cdot(\bx_S-\hat{\bx}_w)
    \end{equation}
    where we have used the fact that $\phi_\tau\rvert_K$ is linear. Since $\hat{\bx}_w$ and each $\bx_S\in S$ are on opposite sides of $\Pi_\tau$, the last term in \eqref{eqn: 3nef25fs} cannot vanish. Thus, for each $\bx_S\in S$, we can solve \eqref{eqn: 3nef25fs} for $\sigma$, and define $\hat{\boldsymbol{X}}:S\times(0,t)\rightarrow \mathcal{E}_t$ by
    \begin{equation}
        \hat{\boldsymbol{X}}(\bx_S,\tau)=\hat{\bx}_w-\frac{\phi_\tau(\hat{\bx}_w)}{\bnabla\phi_\tau\rvert_{K}\cdot(\bx_S-\hat{\bx}_w)}(\bx_S-\hat{\bx}_w),
    \end{equation}
    which satisfies $\hat{\boldsymbol{X}}(S,\tau)=\Pi_\tau\cap\mathcal{P}$. Thus, $\hat{\boldsymbol{X}}$ is surjective. Furthermore, $\hat{\boldsymbol{X}}$ is injective because $\Pi_\tau\cap\mathcal{P}$ and $\Pi_{\tau^\prime}\cap\mathcal{P}$ nowhere intersect for $\tau\neq\tau^\prime$, and because of uniqueness of \eqref{eqn: 3nef25fs} with respect to $\sigma$ and $\bx_S$. Finally, $\hat{\boldsymbol{X}}$ is smooth because it is rational and non-singular.
\end{proof}
Using Lemma~\ref{lemma: lemma 3} we can now parametrise the integral over $\mathcal{E}_t$ where $\mathcal{P}$ is a convex polytope equal to, or a subset of, $K\in\mathcal{T}_h$. 
\begin{lemma}\label{lemma: 4}
    Consider a perturbation $\phi_{t}(\bx)=\phi(\bx)+tw(\bx)$ for some $t \in (0, t_{max}]$, according to Assumption~\ref{assumption 3}, and where $w$ is a linear Lagrangian basis function.
    Let $\mathcal{P}\subseteq K\in\mathcal{T}_h$ be a convex polytope 
    where $K$ is in the support of $w$ and $\Pi_t$ be a moving plane parametrised locally by $\Pi_t\rvert_{K}=\{\bx\in K:\phi_t(\bx)=0\}$. For $f\in C^0(\overline{\mathcal{E}_t})$  where $\mathcal{E}_t=(H^-\setminus\overline{H^-_t})\cap\mathcal{P}$, it holds that
    \begin{equation}      \int_{\mathcal{E}_t}f~\mathrm{d}\bx=\int_0^t\int_{\Pi_\tau\cap\mathcal{P}}f\frac{w}{\lvert\partial_{\bn_\tau}\phi_{\tau}\rvert}~\mathrm{d}S~\mathrm{d}\tau.
    \end{equation}
\end{lemma}
\begin{proof}
    Suppose $\hat{\boldsymbol{X}}$ is the mapping from Lemma~\ref{lemma: lemma 3} and $\mathcal{U}\subset\mathbb{R}^{d-1}$ is the domain of a piecewise smooth parametrisation of $S$, again from Lemma~\ref{lemma: lemma 3}. We define a parametrisation of $\mathcal{E}_t$ by $\boldsymbol{X}:\mathcal{U}\times(0,t)\rightarrow\mathcal{E}_t$ as
    \begin{equation}
        \boldsymbol{X}(\boldsymbol{u},\tau)=\hat{\boldsymbol{X}}(\bx_S(\boldsymbol{u}),\tau),
    \end{equation}
    where $\tau\in(0,t)$ and the mapping $\boldsymbol u\mapsto \boldsymbol x_S(\boldsymbol u)$ is piecewise affine such that $\boldsymbol{X}(\mathcal{U},\tau)=\Pi_\tau\cap\mathcal{P}$.
    The Jacobian of $\boldsymbol{X}$ is given by
    \begin{equation}\label{eqn: lem 4 inter 1}
        \det D\boldsymbol{X}=\pm\pderiv{\boldsymbol{X}}{\tau}\cdot \boldsymbol{N}_\tau\sigma,
    \end{equation}
    where, in the general case for $d>2$, using the notation of Spivak~\cite[Ch.~4]{Spivak_2018},
    \begin{equation}
    \begin{aligned}
    \boldsymbol{N}_\tau\sigma=\pderiv{\boldsymbol{X}}{u_1}\times\dots\times\pderiv{\boldsymbol{X}}{u_{d-1}},
    \qquad
    \sigma=\left\lVert\pderiv{\boldsymbol{X}}{u_1}\times\dots\times\pderiv{\boldsymbol{X}}{u_{d-1}}\right\rVert,
    \end{aligned}
    \end{equation}
in which $\times$ is the generalised cross product defined such that $\boldsymbol{N}_\tau\sigma$ is the vector perpendicular to each of the tangent vectors $\partial\boldsymbol X/\partial u_i$, 
and the plus/minus depends on the orientation of the parametrisation. For the case of $d=2$, we can define
\begin{equation}
\boldsymbol{N}_\tau\sigma = \left(-\pderiv{\boldsymbol{X}_2}{u_1},\pderiv{\boldsymbol{X}_1}{u_1}\right)^\intercal
\end{equation}
By these definitions, $\boldsymbol{N}_\tau$ will be the unit normal vector to the hyperplane $\Pi_\tau\cap\mathcal{P}$ outward with respect to $H^-_\tau$.

    We may now follow similar steps to those in \citet{Berggren_2023}. Namely, since $\boldsymbol{X}(\mathcal{U},\tau)=\Pi_t\cap\mathcal{P}$ and $\phi_\tau\rvert_{\Pi_\tau\cap\mathcal{P}}=0$ we have
    \begin{equation}
        \phi_\tau\circ\boldsymbol{X}=\phi\circ\boldsymbol{X}+\tau w\circ\boldsymbol{X}=0.
    \end{equation}
    Differentiating with respect to $\tau$, it follows that
    \begin{equation}\label{eqn: lem 4 inter 2}
    \begin{aligned}
            0&=\pderiv{}\tau\bigl(\phi\circ\boldsymbol{X}+\tau w\circ\boldsymbol{X}\bigr)=
            (\bnabla\phi\circ\boldsymbol X)\cdot\pderiv{\boldsymbol X}\tau
            + w\circ\boldsymbol X + \tau\bigl(\bnabla w\circ\boldsymbol X\bigr)\cdot\pderiv{\boldsymbol X}\tau
            \\
            &= w\circ\boldsymbol{X} + \bigl(\bnabla\phi_\tau\circ\boldsymbol X\bigr)\cdot\pderiv{\boldsymbol X}\tau
            =
            w\circ\boldsymbol{X}+\lVert\bnabla \phi_\tau\circ\boldsymbol{X}\rVert\boldsymbol{N}_\tau\cdot\pderiv{\boldsymbol{X}}{\tau},
    \end{aligned}
    \end{equation}
     where, in the fourth equality, we have used  the property of the level-set function that
    \begin{equation}
        \boldsymbol{N}_\tau=\frac{\bnabla \phi_\tau\circ\boldsymbol{X}}{\lVert\bnabla \phi_\tau\circ\boldsymbol{X}\rVert}.
    \end{equation}
    Combining \eqref{eqn: lem 4 inter 1} and \eqref{eqn: lem 4 inter 2} and taking the absolute value then gives
    \begin{equation}
        \lvert\det D\boldsymbol{X}\rvert=\frac{w\circ\boldsymbol{X}}{\lVert\bnabla \phi_\tau\circ\boldsymbol{X}\rVert}\sigma,
    \end{equation}
    where we have used that $w\geq0$. We can now parametrise the integral in terms of $\boldsymbol{X}$ as follows,
    \begin{equation}
        \int_{\mathcal{E}_t}f~\mathrm{d}\bx=\int_0^t\int_\mathcal{U}f\circ\boldsymbol{X}\lvert\det D\boldsymbol{X}\rvert~\mathrm{d}\boldsymbol{u}~\mathrm{d}\tau=\int_0^t\int_\mathcal{U}f\circ\boldsymbol{X}\frac{w\circ\boldsymbol{X}}{\lVert\bnabla \phi_\tau\circ\boldsymbol{X}\rVert}\sigma~\mathrm{d}\boldsymbol{u}~\mathrm{d}\tau.
    \end{equation}
    Thus a change of variables and using the fact that the tangential derivative of $\phi_\tau$ vanishes on $\Pi_\tau\cap\mathcal{P}$ yields
    \begin{equation}
        \int_0^t\int_\mathcal{U}f\circ\boldsymbol{X}\frac{w\circ\boldsymbol{X}}{\lVert\bnabla \phi_\tau\circ\boldsymbol{X}\rVert}\sigma~\mathrm{d}\boldsymbol{u}~\mathrm{d}\tau=\int_0^t\int_{\Pi_\tau\cap\mathcal{P}}f\frac{w}{\lvert\partial_{\bn_\tau} \phi_\tau\rvert}~\mathrm{d}S~\mathrm{d}\tau,
    \end{equation}
    which completes the proof.
\end{proof}

Now having a result for a generic polytope, we need to relate this back to the case of several intersecting domains. To do so, we introduce some additional notation as follows.

For $n$ level-set functions, there are $2^n$ different unique subdomains that combinatorically can be defined by the signs of each level-set function. For example, for the case of two level-set functions as in Figure~\ref{fig:fig2-new}, there are four subdomains $\Omega_1,\dots,\Omega_4$ defined as
\begin{equation}
    D_{\phi_1}\cap  \overline{D}_{\phi_2}^\complement,~D_{\phi_1}\cap  D_{\phi_2},~\overline{D}_{\phi_1}^\complement\cap  D_{\phi_2},~\text{and}~\overline{D}_{\phi_1}^\complement\cap  \overline{D}_{\phi_2}^\complement.
\end{equation}
More generally, the $2^n$ subdomain can be enumerated as \citep{10.1051/cocv/2013076_2014}
\begin{equation}\label{eqn: all subdomains}
\begin{cases}
    \Omega_1&=D_{\phi_1}\cap D_{\phi_2}\cap\dots\cap D_{\phi_n}\\
    \Omega_2 &= \overline{D}_{\phi_1}^\complement\cap D_{\phi_2}\cap\dots\cap D_{\phi_n}\\
    &~\vdots\\
    \Omega_{2^n} &= \overline{D}_{\phi_1}^\complement\cap \overline{D}_{\phi_2}^\complement\cap\dots\cap \overline{D}_{\phi_n}^\complement
\end{cases},
\end{equation}
where the order can be arbitrarily chosen. Taking inspiration from this representation, we define $F$ to be any one of the subdomains $\Omega_1$, $\Omega_2$, $\dots$, $\Omega_{2^n}$. We then define $F_k$ to be the region with the domain defined by $\phi_k$, which will be either $D_{\phi_k}$ or $\overline{D}_{\phi_k}^\complement$, removed from the construction of $F$. Analogously,  also define $F_{kl}$ to be the region with both of the domains defined by $\phi_k$ and $\phi_l$ removed from the construction of $F$. For example, if $F=\Omega_2=\overline{D}_{\phi_1}^\complement\cap D_{\phi_2}\cap\dots\cap D_{\phi_n}$, then $F_1$ and $F_{1n}$ are given by $F_1=D_{\phi_2}\cap\dots\cap D_{\phi_n}$ and $F_{1n} = D_{\phi_2}\cap\dots\cap D_{\phi_{n-1}}$, respectively. 

Finally, we need a way to keep track of which domains are and are not complemented in a particular $F$. To do so, we denote the set of all domains $\mathcal{D}=\{D_{\phi_1},D_{\phi_2},\dots,D_{\phi_n}\}$ and define $\mathcal{D}_c\subset\mathcal{D}$ and $\mathcal{D}_p\subset\mathcal{D}$ to be the sets of domains that appear with and without complements in $F$, respectively. By definition, these sets of domains should satisfy $\mathcal{D}_c\cap\mathcal{D}_p=\emptyset$ and $\mathcal{D}_c\cup\mathcal{D}_p=\mathcal{D}$. In terms of our previous example for $F=\Omega_2=\overline{D}_{\phi_1}^\complement\cap D_{\phi_2}\cap\dots\cap D_{\phi_n}$, we have $\mathcal{D}_c=\{D_{\phi_1}\}$ and $\mathcal{D}_p=\{D_{\phi_2},\dots,D_{\phi_n}\}$.

\begin{remark}
Using the sets of domains $\mathcal{D}_c$ and $\mathcal{D}_p$, we can alternatively write $F$ as
\begin{equation}
    F=\left[\bigcap_{A\in\mathcal{D}_p\cup\{D\}}A\right]\cap\left[\bigcap_{B\in\mathcal{D}_c}B^\complement\right],
\end{equation}
where $D^\complement=\emptyset$ because $D$ is the hold-all domain. Similar expressions can be written for $F_k$ and $F_{kl}$, however for brevity we do not write them here.
\end{remark}

The above notation is convenient, because when we perturb the $k^{\rm th}$ level-set function, we are actually cutting a fixed domain $F_k$ with $\partial D_{\phi_{k},\tau}$. As such, the perturbed domain $E_t$ can be redefined as $E_t=(D_{\phi_k}\setminus\overline{D_{\phi_{k},t}})\cap F_k$, assuming $D_{\phi_k}\in\mathcal{D}_p$. We can then establish equivalence to the polytopal representation as follows: given $K\in\mathcal{T}_h$, we choose
\begin{equation}
    \mathcal{P}\equiv F_k\cap K.
\end{equation}
We then use that $\partial D_{\phi_{k},t}\cap K$ is planar, and we can therefore choose $\Pi_t\equiv\partial D_{\phi_{k},t}\cap K$. Finally, by construction \begin{equation}
    E_t\cap K\equiv\mathcal{E}_t
\end{equation}
inside $K$. Figure~\ref{fig:representations of the boundary} shows a visualisation of this.
\begin{figure}[!t]
    \centering
    \begin{tikzpicture}
        \coordinate (A) at (0,0);
        \coordinate (B) at (5,0);
        \coordinate (C) at (2.5,4);

        \path[name path=tri] (A) -- (B) -- (C) -- cycle;

        \path[name path=line1] ($(A)!0.3!(C)$) -- ($(A)!0.4!(B)$);
        \path[name path=line3] ($(A)!0.2!(B)$) -- ($(B)!0.3!(C)$);
        \path[name path=line4] ($(A)!0.7!(C)$) -- ($(B)!0.6!(C)$);
        \path[name path=cut_line] ($(A)!0.65!(B)$) -- ($(B)!0.75!(C)$);

        \fill [name intersections={of=line1 and line3, by=int}];
        \fill [name intersections={of=tri and line4, by=int4tri}];
        \fill[name intersections={of=tri and line4, name=i, total=\t}];

        \fill [name intersections={of=line3 and cut_line, by=int_line3_cl}];
        \fill [name intersections={of=line4 and cut_line, by=int_line4_cl}];
        \shade[top color=myred!50] (i-2) -- ($(A)!0.3!(C)$) -- (int) -- (int_line3_cl) -- (int_line4_cl) -- cycle;
        \draw[very thick] (A) -- (B) -- (C) -- cycle;

        \draw[very thick, dashed] ($(A)!0.3!(C)$) -- (int);
        \draw[very thick, dashed] (int) -- ($(A)!0.4!(B)$);
        \draw[very thick, dashed] ($(A)!0.2!(B)$) -- (int);
        \draw[very thick, dashed] (int) -- ($(B)!0.3!(C)$);
        \draw[very thick, dashed] (i-1) -- (i-2);
        \draw[very thick, dashed, name path=cut_line] ($(A)!0.65!(B)$) -- ($(B)!0.75!(C)$);

        \path[name path=blue_top] ($(A)!0.5!(C)$) -- ($(B)!0.45!(C)$);
        \path[name path=blue_bot] ($(A)!0.41666!(C)$) -- ($(B)!0.36!(C)$);

        \path [name intersections={of=blue_top and cut_line, by=cut_top}];
        \path [name intersections={of=blue_bot and cut_line, by=cut_bot}];
        \draw[anotherblue, very thick] ($(A)!0.5!(C)$) -- (cut_top);
        \draw[anotherblue, dashed, very thick, dash phase=4pt] (cut_top) -- ($(B)!0.45!(C)$);
        \draw[anotherblue,dashed, very thick, dash phase=4.5pt] ($(A)!0.41666!(C)$) -- ($(B)!0.36!(C)$) node[midway,below left] {$E_t$};

        \coordinate (blueL_boundary) at ($(A)!0.5!(C)$);
        \coordinate (blueR_boundary) at ($(B)!0.45!(C)$);
        \coordinate (blueL_extended) at ($(blueR_boundary)!1.20!(blueL_boundary)$);
        \coordinate (blueR_extended) at ($(blueL_boundary)!1.20!(blueR_boundary)$);
        \coordinate (blueL_extended2) at ($ (blueL_extended) + (-0.5,0.5) $);
        \coordinate (blueR_extended2) at ($ (blueR_extended) + (0.5,0.5) $);
        \draw[anotherblue, dashed, very thick] (blueL_extended2) -- (blueL_boundary);
        \draw[anotherblue, dashed, very thick] (blueR_boundary) -- (blueR_extended2);
        \coordinate (blueL_boundary2) at ($(A)!0.41666!(C)$);
        \coordinate (blueR_boundary2) at ($(B)!0.36!(C)$);
        \coordinate (blueL_extended3) at ($ (blueL_boundary2) + (-1,0.65) $);
        \coordinate (blueR_extended3) at ($ (blueR_boundary2) + (1,0.6) $);
        \draw[anotherblue, dashed, very thick] (blueL_extended3) -- (blueL_boundary2);
        \draw[anotherblue, dashed, very thick] (blueR_boundary2) -- (blueR_extended3);

        \draw[pattern={Lines[angle=30,distance=4pt]},pattern color=anotherblue,draw=none]
            ($(A)!0.5!(C)$) --
            (cut_top) --
            (cut_bot) --
            ($(A)!0.41666!(C)$) --
            cycle;

        \node[above] at (blueL_extended2) {\textcolor{anotherblue}{$\partial D_{\phi_{k}}$}};
        \node[below] at (blueL_extended3) {\textcolor{anotherblue}{$\partial D_{\phi_{k},t}$}};
        \node[above right] at ($(i-2) + (0.05,-0.1)$) {\textcolor{myred}{$F_k\cap K$}};
        \node[right] at (B) {$K$};

        \coordinate (K1) at (5.5,2.8);
        \coordinate (K2) at (5.5+8/4,2.8);
        \draw [<->,very thick] (K1) to [out=30,in=150] (K2);

        \coordinate (A) at (0+7.5,0);
        \coordinate (B) at (5+7.5,0);
        \coordinate (C) at (2.5+7.5,4);

        \path[name path=tri] (A) -- (B) -- (C) -- cycle;

        \path[name path=line1] ($(A)!0.3!(C)$) -- ($(A)!0.4!(B)$);
        \path[name path=line3] ($(A)!0.2!(B)$) -- ($(B)!0.3!(C)$);
        \path[name path=line4] ($(A)!0.7!(C)$) -- ($(B)!0.6!(C)$);
        \path[name path=cut_line] ($(A)!0.65!(B)$) -- ($(B)!0.75!(C)$);

        \fill [name intersections={of=line1 and line3, by=int}];
        \fill [name intersections={of=tri and line4, by=int4tri}];
        \fill[name intersections={of=tri and line4, name=i, total=\t}];

        \fill [name intersections={of=line3 and cut_line, by=int_line3_cl}];
        \fill [name intersections={of=line4 and cut_line, by=int_line4_cl}];
        \shade[top color=myred!50] (i-2) -- ($(A)!0.3!(C)$) -- (int) -- (int_line3_cl) -- (int_line4_cl) -- cycle;
        \draw[very thick] (A) -- (B) -- (C) -- cycle;

        \draw[very thick, dashed] ($(A)!0.3!(C)$) -- (int);
        \draw[very thick, dashed] (int) -- ($(A)!0.4!(B)$);
        \draw[very thick, dashed] ($(A)!0.2!(B)$) -- (int);
        \draw[very thick, dashed] (int) -- ($(B)!0.3!(C)$);
        \draw[very thick, dashed] (i-1) -- (i-2);
        \draw[very thick, dashed, name path=cut_line] ($(A)!0.65!(B)$) -- ($(B)!0.75!(C)$);
        \path[name path=blue_top] ($(A)!0.5!(C)$) -- ($(B)!0.45!(C)$);
        \path[name path=blue_bot] ($(A)!0.41666!(C)$) -- ($(B)!0.36!(C)$);
        \path [name intersections={of=blue_top and cut_line, by=cut_top}];
        \path [name intersections={of=blue_bot and cut_line, by=cut_bot}];
        \draw[anotherblue, very thick] ($(A)!0.5!(C)$) -- (cut_top);

        \coordinate (blueL_boundary) at ($(A)!0.5!(C)$);
        \coordinate (blueR_boundary) at ($(B)!0.45!(C)$);
        \coordinate (blueL_extended) at ($(blueR_boundary)!1.3!(blueL_boundary)$);
        \coordinate (blueR_extended) at ($(blueL_boundary)!1.3!(blueR_boundary)$);
        \draw[anotherblue, very thick,dashed] (blueL_extended) -- (blueL_boundary);
        \draw[anotherblue, very thick,dashed] (cut_top) -- (blueR_extended);

        \coordinate (blueL_boundary2) at ($(A)!0.41666!(C)$);
        \coordinate (blueR_boundary2) at ($(B)!0.36!(C)$);
        \coordinate (blueL_extended2) at ($(blueR_boundary2)!1.3!(blueL_boundary2)$);
        \coordinate (blueR_extended2) at ($(blueL_boundary2)!1.3!(blueR_boundary2)$);
        \draw[anotherblue, very thick,dashed] (blueL_extended2) -- (blueR_extended2) node[midway,below left] {$\mathcal{E}_t$};
        \draw[pattern={Lines[angle=30,distance=4pt]},pattern color=anotherblue,draw=none]
            ($(A)!0.5!(C)$) --
            (cut_top) --
            (cut_bot) --
            ($(A)!0.41666!(C)$) --
            cycle;

        \node[above] at (blueL_extended) {\textcolor{anotherblue}{$\Pi$}};
        \node[below] at (blueL_extended2) {\textcolor{anotherblue}{$\Pi_t$}};
        \node[above right] at ($(i-2) + (0.5,-0.1)$) {\textcolor{myred}{$\mathcal{P}$}};
        \node[right] at (B) {$K$};
    \end{tikzpicture}
    \caption{An illustration of the two equivalent representations of a domain being cut by several level-set functions. The left figure shows the domain $F_k\cap K$ cut by a boundary $\partial D_{\phi_{k},t}$ for some $t\in(0,t_{\rm max}]$ satisfying Assumption~\ref{assumption 3} and the resulting perturbed domain $E_t=(D_{\phi_k}\setminus\overline{D_{\phi_{k},t}})\cap F_k$. Note that the dashed black lines represent fixed boundaries of other domains $D_{\phi_i}$ for $i\neq k$. The right figure shows the equivalent representation in terms of a polytope $\mathcal{P}$ cut by a plane $\Pi_t$, again for some $t\in(0,t_{\rm max}]$. In this case, the perturbed region is defined using the halfspaces for $\Pi$ and $\Pi_t$ as $\mathcal{E}_t=(H^-\setminus\overline{H^-_t})\cap\mathcal{P}$. Since the $\partial D_{\phi_{k},t}\cap K$ is planar, it holds that inside $K$ we have $\partial D_{\phi_{k},t}\cap K\equiv\Pi_t$. An immediate corollary of this is that we also have $E_t\cap K\equiv\mathcal{E}_t$.}
    \label{fig:representations of the boundary}
\end{figure}
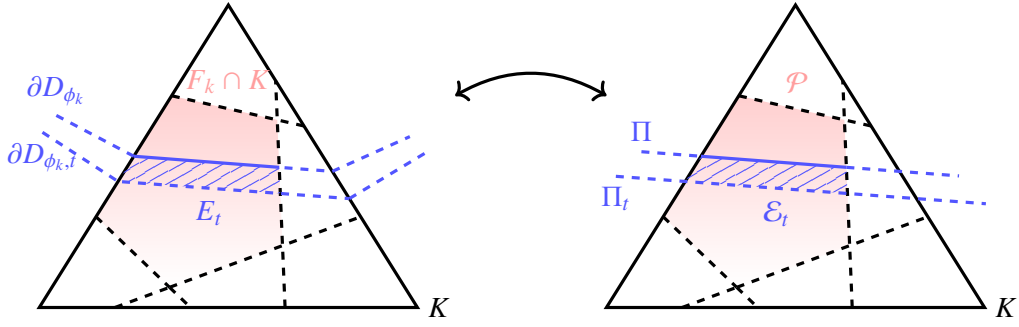

In Theorem~\ref{theorem 1} below, as well as in Theorems \ref{theorem 2}, and \ref{theorem 3}, we use the notation $C^k(\overline{\mathcal{T}}_{\!h})$, where $k\in\mathbb N$, for the space of functions $f$ such that $f|_K\in C^k(\bar K)$ for each $K\in\mathcal T_h$.
Note that such functions may possess jump discontinuities between elements, a generality that is needed for some functions emanating from a finite-element discretisation.

\begin{theorem}\label{theorem 1}
    Let $F$ be any one of the subdomains~\eqref{eqn: all subdomains} defined by $n$ level set functions.
    Given a perturbation $\phi_{k,t}(\bx)=\phi_k(\bx)+tw(\bx)$, where $w$ is a linear Lagrangian basis function, the directional semideriative of the integral
    \begin{equation}
        J(\phi_k)=\int_{F}f~\mathrm{d}\bx,
    \end{equation}
    with $D_{\phi_k}\in\mathcal{D}_p$ and $f\in C^0(\overline{\mathcal{T}}_{\!h})$, satisfies
    \begin{equation}\label{eqn: theorem 1 main}
        \mathrm{d}J(\phi_k)(w)=\lim_{t\rightarrow0^+}\frac{1}{t}\bigl(J(\phi_{k,t})-J(\phi_k)\bigr)=-\int_{\partial D_{\phi_k}\cap F_k}f\frac{w}{\lvert\partial_{\bn}\phi_{k}\rvert}~\mathrm{d}S,
    \end{equation}
    where $\bn$ is the unit normal to $\partial D_{\phi_k}$. If $D_{\phi_k}\in\mathcal{D}_c$, the right-hand side is of opposite sign.
    \begin{enumerate}
        \item[(i)] If Assumption~\ref{assumption 2} is violated, then $f$ and $\partial_{\bn}\phi$ are the limits of these functions from the interior of $F$, whenever these quantities possess jump discontinuities on $\partial D_{\phi_k}\cap F_k$.
        \item[(ii)] If Assumption~\ref{assumption 2} is satisfied, the semiderivatives $t \rightarrow 0^-$ and $t \rightarrow 0^+$ agree.
    \end{enumerate}
\end{theorem}
\begin{proof}
    Suppose $D_{\phi_k}\in\mathcal{D}_p$ and let $t\in(0,t_{\rm max}]$ according to Assumption~\ref{assumption 3}.
    Then we have
    \begin{equation}
        \frac{1}{t}\bigl(J(\phi_{k,t})-J(\phi_k)\bigr)=\frac{1}{t}\left(\int_{D_{\phi_{k},t}\cap F_k}f~\mathrm{d}\bx-\int_{D_{\phi_{k}}\cap F_k}f~\mathrm{d}\bx\right)=-\frac{1}{t}\sum_{K\in\mathcal{T}_h}\int_{E_t\cap K}f~\mathrm{d}\bx,
    \end{equation}
    where, using the same arguments as used for \eqref{eqn: dom integral concept}, $E_t=(D_{\phi_k}\setminus\overline{D_{\phi_{k},t}})\cap F_k$.
    By the equivalence of the domain representation and the polytopal representation inside a cell $K$, we can apply Lemma~\ref{lemma: 4} to find
    \begin{equation}
        \frac{1}{t}\bigl(J(\phi_{k,t})-J(\phi_k)\bigr)=-\frac{1}{t}\sum_{K\in\mathcal{T}_h}\int_0^t\int_{\partial{D_{\phi_{k},\tau}}\cap F_k\cap K}f\frac{w}{\lvert\partial_{\bn}\phi_{\tau}\rvert}~\mathrm{d}S~\mathrm{d}\tau,
    \end{equation}
    where we have identified $\Pi_t\equiv\partial{D_{\phi_{k},t}}\cap K$. Taking the limit as then yields \eqref{eqn: theorem 1 main}. The case for $D_{\phi_k}\in\mathcal{D}_c$ follows similarly with the addition of a chain-rule computation because $D_{\phi_k}^\complement = D_{-\phi_k}$.

    Results (i) and (ii) follow by similar arguments to those of \citet{Berggren_2023}. Namely, (i) follows from the fact that $\partial D_{\phi_{k},\tau}\cap F_k$ is interior to $F$ for each $\tau\in(0, t)$. Furthermore, an analogous analysis can be carried out for $t \leq 0$ that would also lead to \eqref{eqn: theorem 1 main} for $t \rightarrow 0^-$, but with $f$ and $\partial_{\bn}\phi$ being the limits from the exterior of $F$. Finally, (ii) follows from the fact that $f$ and $\partial_{\bn}\phi$ are continuous almost everywhere on $\partial D_{\phi_{k},\tau}\cap F_k$ when Assumption~\ref{assumption 2} is satisfied. Note that in this case, the violation of Assumption~\ref{assumption 4} has no effect by the same continuity on $\partial D_{\phi_{k},\tau}\cap F_k$.
\end{proof}
\begin{remark}\label{remark: likelihood}
    It is worth noting that in reality, provided that the domains are not initialised with such an intersection, the likelihood that Assumption~\ref{assumption 2} or Assumption~\ref{assumption 4} are violated when the level-set function is updated by an optimisation algorithm is very close to zero.
    Put in measure-theoretic terms, Assumptions~\ref{assumption 2} and~\ref{assumption 4} are satisfied for almost all level-set functions.
\end{remark}


\subsection{Boundary truncated by perturbed domains}\label{subsec:boundary int trunc}\noindent
It turns out that in the case of several intersecting domains, there is a new case to consider --- a case not occurring with a single level-set function --- in which a boundary is truncated by a perturbed domain. 
This situation naturally arises when integrating on interfaces between two regions.
One example is when integrating on $\Gamma_{23}$ in Figure~\ref{fig:fig2-new} and computing a derivative of
\begin{equation}
    F(\phi_2)=\int_{\partial D_{\phi_1}\cap D_{\phi_2}}f~\mathrm{d}\bx
\end{equation}
under a perturbation of the level-set function $\phi_2$.
Generally, we will consider the case shown in Figure~\ref{fig: truncated bdry}, integrating over a part of $\Pi_{\int}$. Note that we allow $\Pi_{\int}$ to align with an edge $e\in E(\mathcal{P})$, which turns out to be useful later.

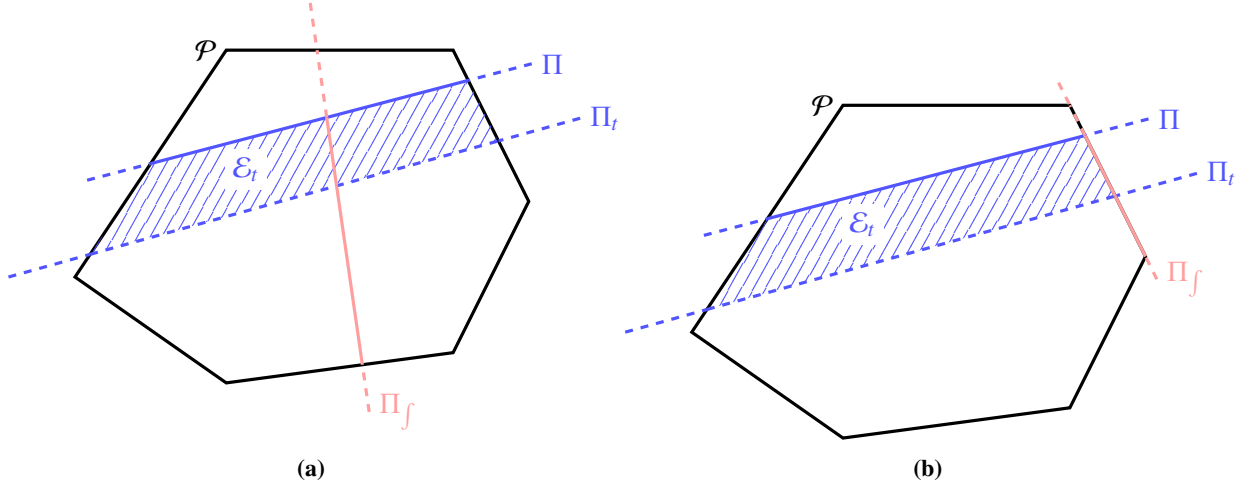
\begin{figure}[H]
    \centering
    \begin{subfigure}{0.49\textwidth}
    \centering
    \begin{tikzpicture}
        \coordinate (A) at (2,0.6);
        \coordinate (B) at (5,1);
        \coordinate (C) at (6,3);
        \coordinate (D) at (5,5);
        \coordinate (E) at (2,5);
        \coordinate (F) at (0,2);
        \draw[very thick] (A) -- (B) -- (C) -- (D) -- (E) -- (F) -- cycle;
        \node[left] at (E) {$\mathcal{P}$};
        \draw[anotherblue, very thick, name path=blue0] ($(F)!0.5!(E)$) -- ($(D)!0.2!(C)$);
        \draw[anotherblue, dashed, very thick, name path=blue1] ($(F)!0.1!(E)$) -- ($(D)!0.6!(C)$);
        \coordinate (blueL_boundary) at ($(F)!0.5!(E)$);
        \coordinate (blueR_boundary) at ($(D)!0.2!(C)$);
        \coordinate (blueL_extended) at ($(blueR_boundary)!1.2!(blueL_boundary)$);
        \coordinate (blueR_extended) at ($(blueL_boundary)!1.2!(blueR_boundary)$);
        \draw[anotherblue, very thick,dashed] (blueL_extended) -- (blueL_boundary);
        \draw[anotherblue, very thick,dashed] (blueR_extended) -- (blueR_boundary);
        \coordinate (blueL_boundary2) at ($(F)!0.1!(E)$);
        \coordinate (blueR_boundary2) at ($(D)!0.6!(C)$);
        \coordinate (blueL_extended2) at ($(blueR_boundary2)!1.2!(blueL_boundary2)$);
        \coordinate (blueR_extended2) at ($(blueL_boundary2)!1.2!(blueR_boundary2)$);
        \draw[anotherblue, very thick,dashed] (blueL_extended2) -- (blueL_boundary2);
        \draw[anotherblue, very thick,dashed] (blueR_extended2) -- (blueR_boundary2);
        \node[right] at (blueR_extended) {\textcolor{anotherblue}{$\Pi$}};
        \node[right] at (blueR_extended2) {\textcolor{anotherblue}{$\Pi_{t}$}};
        \draw[pattern={Lines[angle=61,distance=4pt]},pattern color=anotherblue,draw=none]
            (blueL_boundary) -- (blueR_boundary) -- (blueR_boundary2) -- (blueL_boundary2) -- cycle;
        \node[below] at ($(blueL_boundary)!0.3!(blueR_boundary)$) {\colorbox{white}{\textcolor{anotherblue}{$\mathcal{E}_t$}}};
        \path[name path=red] ($(A)!0.6!(B)$) -- ($(D)!0.6!(E)$);
        \fill [name intersections={of=blue0 and red, by=int0}];
        \fill [name intersections={of=blue1 and red, by=int1}];
        \draw[myred,very thick] ($(A)!0.6!(B)$) -- (int1);
        \draw[myred,very thick] (int0) -- (int1);
        \draw[myred,very thick,dashed] (int0) -- ($(D)!0.6!(E)$);
        \coordinate (redL_boundary) at ($(A)!0.6!(B)$);
        \coordinate (redR_boundary) at ($(D)!0.6!(E)$);
        \coordinate (redL_extended) at ($(redR_boundary)!1.15!(redL_boundary)$);
        \coordinate (redR_extended) at ($(redL_boundary)!1.15!(redR_boundary)$);
        \draw[myred, very thick,dashed] (redL_extended) -- (redL_boundary);
        \draw[myred, very thick,dashed] (redR_extended) -- (redR_boundary);
        \node[right] at (redL_extended) {\textcolor{myred}{$\Pi_{\int}$}};
    \end{tikzpicture}
    \caption{}
    \end{subfigure}
        \begin{subfigure}{0.49\textwidth}
    \centering
    \begin{tikzpicture}
        \coordinate (A) at (2,0.6);
        \coordinate (B) at (5,1);
        \coordinate (C) at (6,3);
        \coordinate (D) at (5,5);
        \coordinate (E) at (2,5);
        \coordinate (F) at (0,2);
        \draw[very thick] (A) -- (B) -- (C) -- (D) -- (E) -- (F) -- cycle;
        \node[left] at (E) {$\mathcal{P}$};
        \draw[anotherblue, very thick, name path=blue0] ($(F)!0.5!(E)$) -- ($(D)!0.2!(C)$);
        \draw[anotherblue, dashed, very thick, name path=blue1] ($(F)!0.1!(E)$) -- ($(D)!0.6!(C)$);
        \coordinate (blueL_boundary) at ($(F)!0.5!(E)$);
        \coordinate (blueR_boundary) at ($(D)!0.2!(C)$);
        \coordinate (blueL_extended) at ($(blueR_boundary)!1.2!(blueL_boundary)$);
        \coordinate (blueR_extended) at ($(blueL_boundary)!1.2!(blueR_boundary)$);
        \draw[anotherblue, very thick,dashed] (blueL_extended) -- (blueL_boundary);
        \draw[anotherblue, very thick,dashed] (blueR_extended) -- (blueR_boundary);
        \coordinate (blueL_boundary2) at ($(F)!0.1!(E)$);
        \coordinate (blueR_boundary2) at ($(D)!0.6!(C)$);
        \coordinate (blueL_extended2) at ($(blueR_boundary2)!1.2!(blueL_boundary2)$);
        \coordinate (blueR_extended2) at ($(blueL_boundary2)!1.2!(blueR_boundary2)$);
        \draw[anotherblue, very thick,dashed] (blueL_extended2) -- (blueL_boundary2);
        \draw[anotherblue, very thick,dashed] (blueR_extended2) -- (blueR_boundary2);
        \node[right] at (blueR_extended) {\textcolor{anotherblue}{$\Pi$}};
        \node[right] at (blueR_extended2) {\textcolor{anotherblue}{$\Pi_{t}$}};
        \draw[pattern={Lines[angle=61,distance=4pt]},pattern color=anotherblue,draw=none]
            (blueL_boundary) -- (blueR_boundary) -- (blueR_boundary2) -- (blueL_boundary2) -- cycle;
        \node[below] at ($(blueL_boundary)!0.3!(blueR_boundary)$) {\colorbox{white}{\textcolor{anotherblue}{$\mathcal{E}_t$}}};
        \path[name path=red] (C) -- (D);
        \fill [name intersections={of=blue0 and red, by=int0}];
        \fill [name intersections={of=blue1 and red, by=int1}];
        \draw[myred,very thick] (C) -- (int1);
        \draw[myred,very thick] (int0) -- (int1);
        \draw[myred,very thick,dashed] (int0) -- (D);
        \coordinate (redL_boundary) at (C);
        \coordinate (redR_boundary) at (D);
        \coordinate (redL_extended) at ($(redR_boundary)!1.15!(redL_boundary)$);
        \coordinate (redR_extended) at ($(redL_boundary)!1.15!(redR_boundary)$);
        \draw[myred, very thick,dashed] (redL_extended) -- (redL_boundary);
        \draw[myred, very thick,dashed] (redR_extended) -- (redR_boundary);
        \node[right] at (redL_extended) {\textcolor{myred}{$\Pi_{\int}$}};
    \end{tikzpicture}
    \caption{}
    \end{subfigure}
    \caption{An illustration of the case of a plane $\Pi_{\int}$ truncated by a moving plane $\Pi_t$ with: (a) $\Pi_{\int}$ inside a convex polytope $\mathcal{P}$; or (b) $\Pi_{\int}$ on an edge of $\mathcal{P}$. Conceptually, we can think of $\Pi_{\int}\cap\mathcal{P}$ as a $(d-1)$--dimensional convex polytope that is being cut by the plane $\Pi_t$.}
    \label{fig: truncated bdry}
\end{figure}

\begin{lemma}\label{lemma: lemma 5}
    Consider a moving plane $\Pi_t$ that is locally parametrised by $\phi_{t}(\bx)=\phi(\bx)+tw(\bx)=0$ for some $t \in (0, t_{max}]$, according to Assumption~\ref{assumption 3}, and where $w$ is a linear Lagrangian basis function. Suppose $K\in\mathcal{T}_h$ is a cell in the support of the basis function $w$, $\mathcal{P}\subseteq K$ is a $d$--dimensional convex polytope, $\Pi_{\int}$ is a plane that has a non-empty intersection with $\overline{\mathcal{P}}$, and $\Pi_t$ cuts the $(d-1)$--dimensional convex polytope $\Pi_{\int}\cap\overline{\mathcal{P}}$ into two sub-polytopes $\Pi_{\int}\cap\overline{\mathcal{P}}\cap H_{t}^\pm$, where $H_{t}^\pm$ are the halfspaces defined by $\Pi_t$. Then, there exists a set $S$ and a diffeomorphism $\hat{\boldsymbol{X}}:S\times(0,t)\rightarrow \mathcal{E}_t\cap\Pi_{\int}$ such that $\hat{\boldsymbol{X}}(S,\tau)=\Pi_\tau\cap\Pi_{\int}$ for each $\tau\in(0,t)$.
\end{lemma}
\begin{proof}
    We can apply Lemma~\ref{lemma: lemma 2} to the case of a $(d-1)$--dimensional convex polytope $\Pi_{\int}\cap\overline{\mathcal{P}}$ that is cut by a plane $\Pi_t$. We thus establish the existence of $\hat{\bx}_w\in V(\Pi_{\int}\cap\overline{\mathcal{P}}\cap H^+)$ and $S\subset\partial(\Pi_{\int}\cap\overline{\mathcal{P}}\cap H_{\tau}^-)$ such that all points $\bx_\tau\in\Pi_\tau\cap\Pi_{\int}$ can be parametrised by
    \begin{equation}
        \bx_\tau=\hat{\bx}_w+\sigma(\bx_S-\hat{\bx}_w)
    \end{equation}
    for $\sigma\in(0,1)$ and $\bx_S\in S$. Note that $\hat{\bx}_w$ is fixed for all $t\in(0,t_{\rm max})$ by Assumption~\ref{assumption 3}. We may then restrict ourselves to working on the plane $\Pi_{\int}$, i.e., the gradient operator $\bnabla$ is the intrinsic derivative on $\Pi_{\int}$, which is given extrinsically as the surface gradient $\bnabla_{\Pi_{\int}}$. From here, the proof follows similarly to Lemma~\ref{lemma: lemma 3}.
\end{proof}
Figure~\ref{fig:surface param} shows a conceptual picture of the parametrisation $\hat{\boldsymbol{X}}$ in two and three dimensions.
\begin{figure}[t]
    \centering
    \begin{subfigure}{\textwidth}
        \centering
        \begin{tikzpicture}
            \coordinate (A) at (1.41176,0.35294);
            \coordinate (B) at (6.68824,1.64706);
            \coordinate (C) at (6.11765,3.52941);
            \coordinate (D) at (5.17647,7.29412);
            \coordinate (E) at (0,6);
            \coordinate (F) at (0.94118,2.23529);
            \draw[anotherblue, very thick, dashed] (0.72,3.12) -- (6,4);
            \draw[anotherblue, dashed, very thick] (0.48,4.08) -- (5.76,4.96);
            \node[right] at (6,4) {\textcolor{anotherblue}{$\Pi_\tau$}};
            \node[right] at (5.76,4.96) {\textcolor{anotherblue}{$\Pi$}};
            \draw[myred,very thick] ($(A)!0.6!(B)$) -- ($(D)!0.6!(E)$);
            \coordinate (redb) at ($(A)!0.6!(B)$);
            \coordinate (redt) at ($(D)!0.6!(E)$);
            \node[right] at ($(redb)!0.35!(redt)$) {\textcolor{myred}{$\Pi_{\int}\cap\overline{\mathcal{P}}$}};
            \coordinate (K1) at (7.2,4);
            \coordinate (K2) at (8.2,4);
            \draw [<->,very thick] (K1) to (K2);
            \coordinate (redt) at (9.5,4);
            \coordinate (redb) at (15.0,4);
            \draw[myred,very thick] (redb) -- (redt);
            \coordinate (redb) at (redb);
            \coordinate (redt) at (redt);
            \node[below] at ($(redb)!0.5!(redt) + (0,2.0)$) {\textcolor{myred}{on $\Pi_{\int}\cap\overline{\mathcal{P}}$}};
            \filldraw[anotherblue] ($(redb)!0.5!(redt)$) circle (2pt) node[anchor=north west] {$\Pi_\tau\cap\Pi_{\int}$};
            \draw [gray,-{Latex[width=10pt, length=10pt]},dashed,line width=3pt] (redt) to ($(redb)!0.5!(redt)$);
            \node[gray,below] at ($(redb)!0.75!(redt)$) {$\hat{\boldsymbol{X}}(\bx_S,\tau)$};
            \filldraw (redb) circle (2pt) node[anchor=west] {\textcolor{black}{$\bx_S\in S$}};
            \filldraw (redt) circle (2pt) node[anchor=east] {$\hat{\bx}_w$};
        \end{tikzpicture}
        \caption{$d=2$}
    \end{subfigure}
    \begin{subfigure}{\textwidth}
        \centering
        \begin{tikzpicture}[
            x={(-0.866cm,-0.5cm)},
            y={(0.866cm,-0.5cm)},
            z={(0cm,1cm)},
            scale=0.8,
            plane1/.style={fill=anotherblue!40, opacity=0.7, thick, line join=round},
            plane2/.style={fill=myred!50, opacity=0.7, thick, line join=round},
            plane3/.style={fill=anotherblue!20, opacity=0.7, thick, line join=round}
        ]
            \def\xa{-2.5} \def\xb{2.5}
            \def\ya{-2.5} \def\yb{2.5}
            \def\za{-2.8} \def\zb{2}
            \def\zc{-1.4}
            \def\slant{0.5}
            \draw[myred, very thick] ({\slant*\za}, \ya, \za) -- ({\slant*\za}, \yb, \za) --
                    ({\slant*\zb}, \yb, \zb) -- ({\slant*\zb}, \ya, \zb) -- cycle;
            \fill[plane2]
                ({\slant*\za}, \ya, \za) -- ({\slant*\za}, \yb, \za) --
                ({\slant*\zc}, \yb, \zc) -- ({\slant*\zc}, \ya, \zc) -- cycle;
            \fill[plane3]
                (\xa, \ya, \zc) -- (\xb, \ya, \zc) -- (\xb, \yb, \zc) -- (\xa, \yb, \zc) -- cycle;
            \fill[plane2]
                ({\slant*\zc}, \ya, \zc) -- ({\slant*\zc}, \yb, \zc) --
                (0, \yb, 0) -- (0, \ya, 0) -- cycle;
            \draw[ultra thick, anotherblue, dashed] ({\slant*\zc}, \ya, \zc) -- ({\slant*\zc}, \yb, \zc);
            \fill[plane1]
                (\xa, \ya, 0) -- (\xb, \ya, 0) -- (\xb, \yb, 0) -- (\xa, \yb, 0) -- cycle;
            \fill[plane2]
                (0, \ya, 0) -- (0, \yb, 0) -- ({\slant*\zb}, \yb, \zb) -- ({\slant*\zb}, \ya, \zb) -- cycle;
            \draw[myred, very thick]
                (0, \ya, 0) --
                ({\slant*\zb}, \ya, \zb) --
                ({\slant*\zb}, \yb, \zb) --
                (0, \yb, 0);
            \draw[myred, very thick]
                ({\slant*\zc}, \yb, \zc) --
                (0, \yb, 0);
            \draw[ultra thick, dashed, anotherblue] (0, \ya, 0) -- (0, \yb, 0);
            \node[anotherblue,left] at (\xb, \ya, 0) {$\Pi$};
            \node[anotherblue,left] at (\xb, \ya, \zc) {$\Pi_\tau$};
            \node[myred,left] at ({\slant*\zb}, \ya, \zb) {$\Pi_{\int}\cap\overline{\mathcal{P}}$};
        \end{tikzpicture}
        \begin{tikzpicture}[scale=0.9]
            \def\sx{1}
            \coordinate (K1) at (0,2.5);
            \coordinate (K2) at (1,2.5);
            \draw [<->,very thick] (K1) to (K2);
            \coordinate (xw) at (\sx+1.5,5);
            \draw[fill=myred!50, opacity=0.7,draw=myred, very thick]
                (\sx+1.5,0) --
                (\sx+6.5,0) --
                (\sx+6.5,5) --
                (xw) -- cycle;
            \draw[ultra thick, anotherblue, dashed, name path=pi]
                (\sx+0.5,1.5) -- (\sx+7.5,1.5);
            \node[above,anotherblue] at (\sx+7,1.5) {$\Pi_\tau$};
            \node[above,myred] at ({(\sx+7.5)/2+0.75},5) {on $\Pi_{\int}\cap\overline{\mathcal{P}}$};
            \draw[line width=2.5pt, green]
                (\sx+1.5,0) --
                (\sx+6.5,0) --
                (\sx+6.5,1.5);
            \node[green!50!black,below] at (\sx+2,0) {$S$};
            \path[name path=x] (xw) -- (\sx+4.5,0);
            \fill [name intersections={of=pi and x, by=int}];
            \draw [gray,-{Latex[width=10pt, length=10pt]},dashed,line width=3pt] (xw) to (int);
            \draw [gray,dashed,line width=1pt] (int) to (\sx+4.5,0);
            \node[gray,right] at ($(xw)!0.75!(int)$) {$\hat{\boldsymbol{X}}(\bx_S,\tau)$};
            \filldraw (xw) circle (2pt) node[anchor=east] {$\hat{\bx}_w$};
            \filldraw (\sx+4.5,0) circle (2pt) node[anchor=north] {\textcolor{black}{$\bx_S$}};
        \end{tikzpicture}
        \caption{$d=3$}
    \end{subfigure}
    \caption{A conceptual picture of the parametrisation $\hat{\boldsymbol{X}}$ from Lemma~\ref{lemma: lemma 5} in two and three dimensions. Note that $\hat{\boldsymbol{X}}$ is defined only on $\Pi_{\int}\cap\overline{\mathcal{P}}$.}
    \label{fig:surface param}
\end{figure}
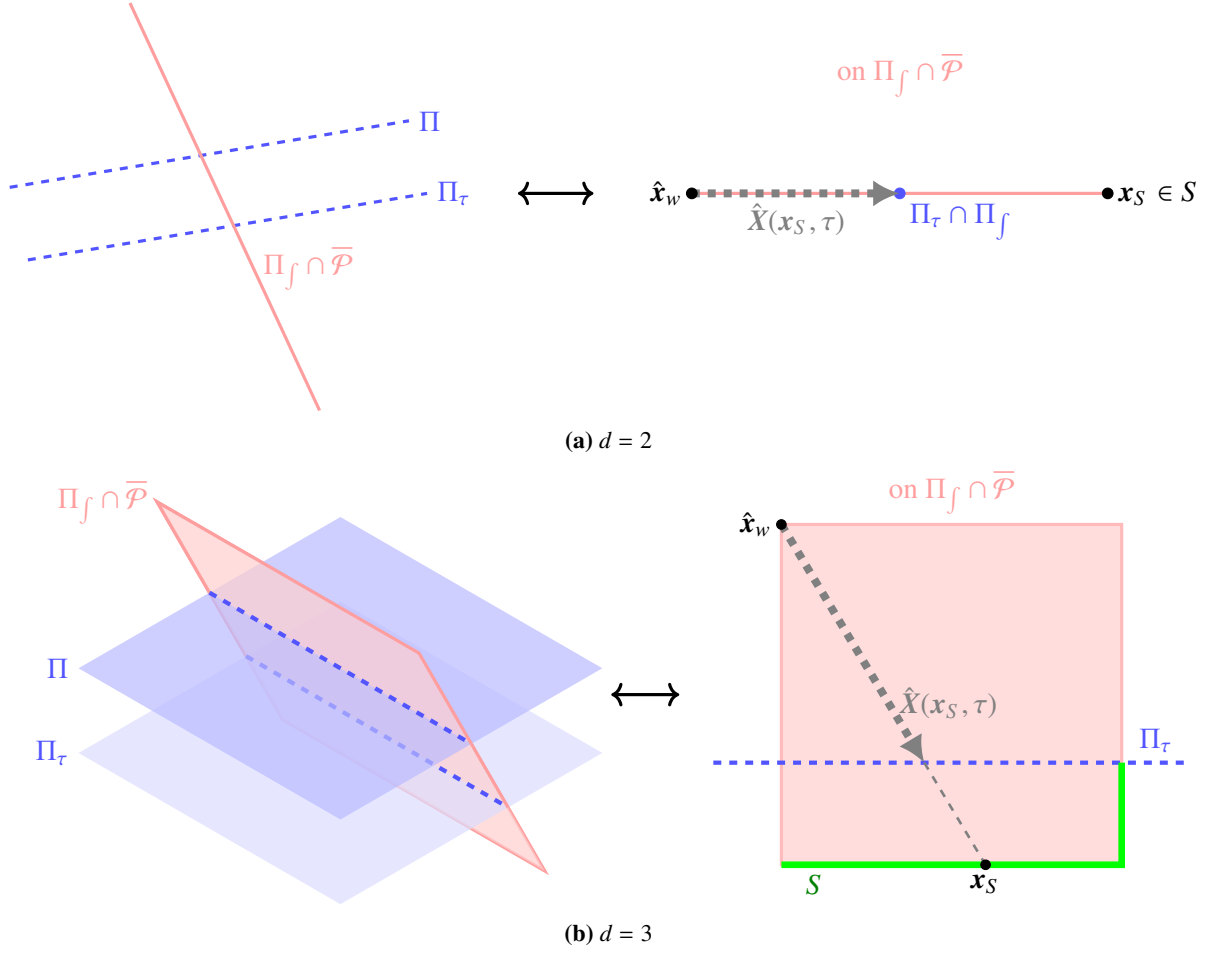
Using Lemma~\ref{lemma: lemma 5} we can now parametrise the integral over $\mathcal{E}_t\cap\Pi_{\int}$.
\begin{lemma}\label{lemma: 6}
    Consider a perturbation $\phi_{t}(\bx)=\phi(\bx)+tw(\bx)$ for some $t \in (0, t_{max}]$, according to Assumption~\ref{assumption 3}, and where $w$ is a linear Lagrangian basis function. Let $\mathcal{P}\subseteq K\in\mathcal{T}_h$ be a convex polytope where $K$ is in the support of $w$, $\Pi_t$ be a moving plane parametrised locally by $\Pi_t\rvert_{K}=\{\bx\in K:\phi_t(\bx)=0\}$, and $\Pi_{\int}$ be a plane that has a non-empty intersection with $\Pi_t\cap \overline{\mathcal{P}}$ for all $t$. For $g\in C^0(\overline{\mathcal{E}_t\cap\Pi_{\int}})$ where $\mathcal{E}_t=(H^-\setminus\overline{H^-_t})\cap\mathcal{P}$, it holds that
    \begin{equation}\label{eqn: lemma 6 int}
        \int_{\mathcal{E}_t\cap\Pi_{\int}}g~\mathrm{d}S=\int_0^t\int_{\Pi_\tau\cap\Pi_{\int}\cap\overline{\mathcal{P}}}g\frac{w}{\lvert\partial_{\boldsymbol{N}_\tau}\phi_{\tau}\rvert}~\mathrm{d}\gamma~\mathrm{d}\tau.
    \end{equation}
    where $\boldsymbol{N}_\tau$ is the unit normal to $\Pi_\tau$ inside $\Pi_{\int}$, outward with respect to $H^-_\tau$.
\end{lemma}
\begin{proof}
    Suppose $\hat{\boldsymbol{X}}$ is the mapping from Lemma~\ref{lemma: lemma 5} and $\mathcal{U}\subset\mathbb{R}^{d-2}$ is the domain of a piecewise smooth parametrisation of $S$, again from Lemma~\ref{lemma: lemma 5}. We define a parametrisation of $\mathcal{E}_t\cap\Pi_{\int}$ by $\boldsymbol{X}:\mathcal{U}\times(0,t)\rightarrow\mathcal{E}_t\cap\Pi_{\int}$ as
    \begin{equation}\label{eqn: parametr}
        \boldsymbol{X}(\boldsymbol{u},\tau)=\hat{\boldsymbol{X}}(\bx_S(\boldsymbol{u}),\tau),
    \end{equation}
    where $\tau\in(0,t)$, and the mapping $\boldsymbol u\mapsto\boldsymbol x_S(\boldsymbol u)$ is piecewise affine such that $\boldsymbol{X}(\mathcal{U},\tau)=\Pi_\tau\cap\Pi_{\int}$. Note that the parametrisation is defined only on $\Pi_{\int}\cap\overline{\mathcal{P}}$. Furthermore, in the case $d=2$, $S$ consists of a single point $\bx_{S}$, so the above parametrisation reduces to vectors $\boldsymbol{X}(\tau)=\hat{\boldsymbol{X}}(\bx_S,\tau)$ of one spatial dimension, that is, effectively a real number.
    However, for simplicity, we use notation~\eqref{eqn: parametr} for all dimensions $d$.

    Similarly to Lemma~\ref{lemma: 4}, it can be shown that the Jacobian is given by
    \begin{equation}\label{eqn: lem 6 inter 1}
        \det D\boldsymbol{X}=\pm\pderiv{\boldsymbol{X}}{\tau}\cdot \boldsymbol{N}_\tau\sigma,
    \end{equation}
    where, in this case,
    \begin{equation}
    \boldsymbol{N}_\tau\sigma=\pderiv{\boldsymbol{X}}{u_1}\times\dots\times\pderiv{\boldsymbol{X}}{u_{d-2}}, \qquad\sigma=\left\lVert\pderiv{\boldsymbol{X}}{u_1}\times\dots\times\pderiv{\boldsymbol{X}}{u_{d-2}}\right\rVert.
    \end{equation}
 For $d=3$, we again have
 \begin{equation}
 \boldsymbol{N}_\tau\sigma = \left(-\pderiv{\boldsymbol{X}_2}{u_1},\pderiv{\boldsymbol{X}_1}{u_1}\right)^\intercal.
 \end{equation}
 In the case of $d=2$, we can define $\boldsymbol{N}_\tau\sigma = \left(1\right)^\intercal$, which can be thought of as a unit tangent along $\Pi_\tau$.
 By these definitions, $\boldsymbol{N}_\tau$ will be the unit normal vector to $\Pi_\tau$ inside $\Pi_{\int}$ and outward with respect to $H^-_\tau$.

    Following very similar steps to the proof of Lemma~\ref{lemma: 4} leads to the expression
    \begin{equation}
        \lvert\det D\boldsymbol{X}\rvert=\frac{w\circ\boldsymbol{X}}{\lVert\bnabla \phi_\tau\circ\boldsymbol{X}\rVert}\sigma,
    \end{equation}
    where $\bnabla$ is the gradient on $\Pi_{\int}\cap\overline{\mathcal{P}}$ and $w$ and $\phi_\tau$ are restricted to $\Pi_{\int}\cap\overline{\mathcal{P}}$. \eqref{eqn: lemma 6 int} then follows by similar arguments to those given at the end of the proof of Lemma~\ref{lemma: 4}.
\end{proof}
\begin{remark}\label{remark: extrinsic N}
    The $\boldsymbol{N}_\tau$ can be written extrinsically in terms of the normal $\bn_\tau$ of $\Pi_\tau$ and the normal $\bn_{\int}$ of $\Pi_{\int}$. In particular, $\boldsymbol{N}_\tau=\pm\tilde{\boldsymbol{N}}_\tau/\lVert\tilde{\boldsymbol{N}}_\tau\rVert$ where $\tilde{\boldsymbol{N}}_\tau = \bn_\tau-(\bn_\tau\cdot\bn_{\int})\bn_{\int}$ and the sign is chosen such that $\boldsymbol{N}_\tau$ points outwards from $H^-_\tau$.
\end{remark}
Using Lemma~\ref{lemma: 6}, we establish the following result.
\begin{theorem}\label{theorem 2}
    Given a perturbation $\phi_{k,t}(\bx)=\phi_k(\bx)+tw(\bx)$, where $w$ is a linear Lagrangian basis function, the directional semideriative of boundary integral
    \begin{equation}\label{eqn: thrm 2 boundary int}
        J(\phi_k)=\int_{\partial D_{\phi_i}\cap F_i}g~\mathrm{d}S,
    \end{equation}
    for $i\neq k$ and $D_{\phi_k}\in\mathcal{D}_p$ satisfying Assumption~\ref{assumption 4} and $g\in C^0(\overline{\mathcal{T}}_{\!h})$, satisfies
    \begin{equation}\label{eqn: theorem 2 main}
        \mathrm{d}J(\phi_k)(w)=\lim_{t\rightarrow0^+}\frac{1}{t}\bigl(J(\phi_{k,t})-J(\phi_k)\bigr)=\added{-}\int_{\partial D_{\phi_i}\cap\partial D_{\phi_k}\cap F_{ik}}g\frac{w}{\lvert\partial_{\boldsymbol{N}^i}\phi_{k}\rvert}~\mathrm{d}\gamma,
    \end{equation}
    where $\boldsymbol{N}^i$ is the unit normal to $\partial D_{\phi_k}$ inside $\partial D_{\phi_i}$, outward with respect to $ D_{\phi_k}$ (see Remark~\ref{remark: extrinsic N}). If $D_{\phi_k}\in\mathcal{D}_c$, the right-hand side is of opposite sign.
    \begin{enumerate}
        \item[(i)] If Assumption~\ref{assumption 2} is violated, then $g$ and $\partial_{\boldsymbol{N}^i}\phi$ are the limits of these functions from the interior of $\partial D_{\phi_i}\cap F_i$ whenever these quantities possess jump discontinuities on $\partial D_{\phi_i}\cap\partial D_{\phi_k}\cap F_{ik}$.
        \item[(ii)] If Assumption~\ref{assumption 2} is satisfied, the semiderivatives $t \rightarrow 0^-$ and $t \rightarrow 0^+$ agree.
    \end{enumerate}
\end{theorem}
\begin{proof}
    Suppose $D_{\phi_k}\in\mathcal{D}_p$ and let $t\in(0,t_{\rm max}]$ according to Assumption~\ref{assumption 3}. Then for $E_t=(D_{\phi_k}\setminus\overline{D_{\phi_{k},t}})\cap F_{ik}$, we have
    \begin{equation}\label{eqn: FDBE1}
        \frac{1}{t}\bigl(J(\phi_{k,t})-J(\phi_k)\bigr)=\frac{1}{t}\left(\int_{\partial D_{\phi_i}\cap D_{\phi_{k,t}}\cap F_{ik}}g~\mathrm{d}\bx-\int_{\partial D_{\phi_i}\cap D_{\phi_{k}}\cap F_{ik}}g~\mathrm{d}\bx\right)=-\frac{1}{t}\sum_{K\in\mathcal{T}_h}\int_{\partial D_{\phi_i}\cap E_t\cap K}g~\mathrm{d}S.
    \end{equation}


    To apply Lemma~\eqref{lemma: 6} to the right side of expression~\eqref{eqn: FDBE1}, we will first need to identify its integration domain $\partial D_{\phi_i}\cap E_t\cap K$ with the $\mathcal E_t\cap\Pi_{\int}$ of formula~\eqref{eqn: lemma 6 int}.
    For this, select $\mathcal E_t = E_t\cap K$ and $\Pi_{\int}\equiv\partial D_{\phi_{i}}\cap K$, which means that, indeed, $\mathcal E_t\cap\Pi_{\int} =\partial D_{\phi_i}\cap E_t\cap K$.
    Note that the identification $\mathcal E_t = E_t\cap K$ implies, by the definitions of $\mathcal E_t$ and $E_t$, that
    \begin{equation}
       \mathcal E_t \equiv (H^-\setminus \overline{H^-_t})\cap\mathcal P = E_t\cap K \equiv \left(D_{\phi_k}\setminus\overline{D_{\phi_{k},t}}\right)\cap F_{ik}\cap K,
    \end{equation}
    which means that we can identify $H^-\setminus \overline{H^-_t} = \left(D_{\phi_k}\setminus\overline{D_{\phi_{k},t}}\right)\cap K$ and $\mathcal P = F_{ik}\cap K$.

    Next, for the integration domain on the right side of formula~\eqref{eqn: lemma 6 int}, that is $\Pi_\tau\cap\Pi_{\int}\cap\overline{\mathcal P}$,  note that since $H^-\setminus \overline{H^-_t} = \left(D_{\phi_k}\setminus\overline{D_{\phi_{k},t}}\right)\cap K$, it follows that $\Pi_\tau = \partial D_{\phi_k,\tau}\cap K$.
    All together, these definitions yield that $\Pi_\tau\cap\Pi_{\int}\cap\overline{\mathcal P} = \partial D_{\phi_k,\tau}\cap\partial D_{\phi_i}\cap F_{ik}\cap K$.

    Then, applying Lemma~\ref{lemma: 6}, we find
    \begin{equation}
        \frac{1}{t}\left(J(\phi_{k,t})-J(\phi_k)\right)=-\frac{1}{t}\sum_{K\in\mathcal{T}_h}\int_0^t\int_{\partial D_{\phi_i}\cap\partial{D_{\phi_{k},\tau}}\cap F_{ik}\cap K}g\frac{w}{\lvert\partial_{\boldsymbol{N}^i}\phi_{\tau}\rvert}~\mathrm{d}\gamma~\mathrm{d}\tau,
    \end{equation}
    where we use Assumption~\ref{assumption 4} to ensure that the intersection $\partial D_{\phi_i} \cap \partial D_{\phi_{k},\tau}$ does not change to become empty or non-empty for some $t$ depending on whether we take $t\rightarrow0^+$ or $t\rightarrow0^-$. Taking the limit gives \eqref{eqn: theorem 2 main}. The case for $D_{\phi_k}\in\mathcal{D}_c$ follows similarly with the addition of a chain-rule computation because $D_{\phi_k}^\complement = D_{-\phi_k}$.

    The cases of (i) and (ii) follow by similar arguments to those in the proof of Theorem~\ref{theorem 1}.
\end{proof}
\begin{remark}\label{remark: point wise 1}
    For $d=2$, \ref{eqn: theorem 2 main} is evaluated pointwise. Namely, for $D_{\phi_k}\in\mathcal{D}_p$,
    \begin{equation}
        \mathrm{d}J(\phi_k)(w)=\added{-}\sum_{K\in \mathcal{T}_h} \left.\left(g\frac{w}{\lvert\partial_{\boldsymbol{N}^i}\phi_{k}\rvert}\right)\right\rvert_{\partial D_{\phi_i}\cap\partial D_{\phi_k}\cap F_{ik}\cap K},
    \end{equation}
    where $\partial D_{\phi_i}\cap\partial D_{\phi_k}\cap F_{ik}$ is the set of intersection points of $\partial D_{\phi_i}\cap\partial D_{\phi_k}$ in $F_{ik}$.
\end{remark}

\subsection{Boundary integrals}\label{subsec:boundary int}\noindent
The final case considered is the directional derivative of functionals of the form
\begin{equation}
    J(\phi_k)=\int_{\partial{D_{\phi_k}}\cap F_k}g~\mathrm{d}S.
\end{equation}
Note that this case is different from \eqref{eqn: thrm 2 boundary int}, because now $\phi_k$ appears in $\partial{D_{\phi_k}}$ rather than implicitly inside $F_i$.
An example of this case is when integrating on $\Gamma_{23}$ in Figure~\ref{fig:fig2-new} and computing the derivative of
\begin{equation}
    F(\phi_1)=\int_{\partial D_{\phi_1}\cap D_{\phi_2}}f~\mathrm{d}\bx
\end{equation}
under a perturbation of the level-set function $\phi_1$.

We begin by establishing the following integration by parts result.
\begin{lemma}\label{lemma: int by parts}
    Consider a perturbation $\phi_{k,t}(\bx)=\phi_k(\bx)+tw(\bx)$ for some $t \in (0, t_{max}]$, according to Assumption~\ref{assumption 3}, and where $w$ is a linear Lagrangian basis function. For $g\in C^1(\overline{\mathcal{T}}_{\!h})$
    and $\bTheta\in L^\infty(E_t)^d$ such that $\bTheta\rvert_{E_t\cap K}\in C^1(\overline{E_t\cap K})^d$ for $K\in\mathcal{T}_h$ where $E_t=(D_{\phi_k}\setminus\overline{D_{\phi_{k},t}})\cap F_k$. Then, it holds that
    \begin{equation}\label{eqn: int by parts}
        \begin{aligned}
            \sum_{K\in\mathcal{T}_h}\int_{E_t\cap K}[g\bnabla\cdot\bTheta+\bTheta\cdot\bnabla g]&~\mathrm{d}\bx=\int_{\partial E_t\setminus (F_k^\partial \cup \partial D) }\bn\cdot\bTheta g~\mathrm{d}S
            \\&+\int_0^t\int_{\partial D \cap \partial D_{\phi_{k},\tau} \cap F_{k}}\bn_{\partial D}\cdot\bTheta g\frac{w}{\lvert\partial_{\boldsymbol{N}_\tau^D}\phi_{k,\tau}\rvert}~\mathrm{d}\gamma~\mathrm{d}\tau
            \\&+\sum_{D_{\phi_i}\in\mathcal{D}\setminus \{D_{\phi_k}\}}\int_0^t\int_{\partial D_{\phi_i} \cap \partial D_{\phi_{k},\tau} \cap F_{ik}}\pm\bn_{\phi_i}\cdot\bTheta g\frac{w}{\lvert\partial_{\boldsymbol{N}_\tau^i}\phi_{k,\tau}\rvert}~\mathrm{d}\gamma~\mathrm{d}\tau
            \\&+\sum_{S\in\mathcal{F}_h\setminus\partial D}\int_0^t\int_{S \cap F_{k} \cap \partial D_{\phi_{k},\tau}}\jump{
            \bTheta g}\frac{w}{\lvert\partial_{\boldsymbol{N}_\tau^S}\phi_{k,\tau}\rvert}~\mathrm{d}\gamma~\mathrm{d}\tau,
        \end{aligned}
    \end{equation}
    where
    \begin{equation}\label{eqn: facet jump}
        \jump{\bTheta g} = \bn_{S,1}\cdot\bTheta_1 g_1 + \bn_{S,2}\cdot\bTheta_2 g_2,
    \end{equation}
    in which $\bn_{S,1}$ and $\bn_{S,2}=-\bn_{S,1}$ are the unit normals to the plane $S\in\mathcal{F}_h\setminus\partial D$; $\bn$, $\bn_{\partial D}$, and $\bn_{\phi_i}$ are the unit normals to $\partial E_t\setminus(F_k^\partial\cup\partial D)$, $\partial D$, and $\partial D_{\phi_i}$ outward with respect to $E_t$, $D$ and $\partial D_{\phi_i}$, respectively; $\boldsymbol{N}_\tau^D$, $\boldsymbol{N}_\tau^i$, and $\boldsymbol{N}_\tau^S$ are the unit normals to $\partial D_{\phi_{k},\tau}$ inside $\partial D$, $\partial D_{\phi_i}$, and $S$, respectively, outward with respect to $D_{\phi_{k},\tau}$; and $F_k^\partial\coloneqq \bigcup_{A\in\mathcal{D}\setminus D_{\phi_k}}\partial A$ is the union of all boundaries except $\partial D_{\phi_k}$. The sign in the third term is positive if $D_{\phi_i}\in\mathcal{D}_p$ and negative if $D_{\phi_i}\in\mathcal{D}_c$.
\end{lemma}
\begin{proof}
Let $E_t=(D_{\phi_k}\setminus\overline{D_{\phi_{k},t}})\cap F_{k}$ and $\hat{E_t}=(D_{\phi_k}\setminus\overline{D_{\phi_{k},t}})\cap F_{ik}$. 
Then, by integration by parts applied to $g\bnabla\cdot\bTheta$ on each $E_t\cap K$ for $K\in\mathcal{T}_h$, we have
\begin{equation}\label{eqn: int by parts inter 1}
    \begin{aligned}
        \sum_{K\in\mathcal{T}_h}\int_{E_t\cap K}[g\bnabla\cdot\bTheta&+\bTheta\cdot\bnabla g]~\mathrm{d}\bx=\int_{\partial E_t\setminus (F_k^\partial \cup \partial D) }\bn\cdot\bTheta g~\mathrm{d}S+\sum_{K\in\mathcal{T}_h}\int_{\partial D\cap E_t \cap K}\bn_{\partial D}\cdot\bTheta g~\mathrm{d}S\\&+\sum_{D_{\phi_i}\in\mathcal{D}\setminus \{D_{\phi_k}\}}\sum_{K\in\mathcal{T}_h}\int_{\partial D_{\phi_i} \cap \hat{E_t} \cap K}\pm\bn_{\phi_i}\cdot\bTheta g~\mathrm{d}S+\sum_{S\in\mathcal{F}_h\setminus\partial D}\int_{S \cap E_t}\jump{\bTheta g}~\mathrm{d}S,
    \end{aligned}
\end{equation}
where the first term on the right-hand side corresponds to the part of $\partial E_t$ that aligns with $\partial D_{\phi_k}$ and $\partial D_{\phi_{k},t}$; the second is the part that aligns with the boundary of the hold-all domain; the third is the part that aligns with any of the boundaries $\partial D_{\phi_i}$, $i\neq k$), where the normal is outward with respect to $D_{\phi_i}$ and has negative sign if $D_{\phi_i}\in\mathcal{D}_c$; and the final term comes from the jump, defined in expression~\eqref{eqn: facet jump}, across a mesh facet in $\mathcal{F}_h$. Note that some of the elements of the sum in the third term will vanish because $\partial D_{\phi_i}$ will have an empty intersection with $\hat{E_t}$.

Identifying $E_t\cap K\equiv \mathcal{E}_t$ and $\partial D\cap K\equiv\Pi_{\int}$ and $\mathcal{P}\equiv F_k\cap K$ and using the fact that $\partial D_{\phi_i}$ is planar in $K$, we can apply Lemma~\ref{lemma: 6} to the second term on the right-hand side of \eqref{eqn: int by parts inter 1} to obtain
\begin{equation}\label{eqn: int by parts a}
    \int_{\partial D\cap E_t \cap K}\bn_{\partial D}\cdot\bTheta g~\mathrm{d}S=\int_0^t\int_{\partial D \cap \partial D_{\phi_{k},\tau} \cap F_k \cap K}\bn_{\partial D}\cdot\bTheta g\frac{w}{\lvert\partial_{\boldsymbol{N}_\tau^D}\phi_{k,\tau}\rvert}~\mathrm{d}\gamma~\mathrm{d}\tau,
\end{equation}
where $\boldsymbol{N}_\tau^D$ is the normal to $\partial D_{\phi_{k},\tau}$ inside $\partial D$, outward with respect to $D_{\phi_{k},\tau}$.

Next, identifying $\partial D_{\phi_i}\cap K\equiv\Pi_{\int}$ and $\mathcal{P}\equiv F_{ik}\cap K$, we can apply Lemma~\ref{lemma: 6} to the third term on the right-hand side of \eqref{eqn: int by parts inter 1} to obtain
\begin{equation}\label{eqn: int by parts b}
    \int_{\partial D_{\phi_i} \cap \hat{E_t} \cap K}\bn_{\phi_i}\cdot\bTheta g~\mathrm{d}S=\int_0^t\int_{\partial D_{\phi_i} \cap \partial D_{\phi_{k},\tau} \cap F_{ik} \cap K}\bn_{\phi_i}\cdot\bTheta g\frac{w}{\lvert\partial_{\boldsymbol{N}_\tau^i}\phi_{k,\tau}\rvert}~\mathrm{d}\gamma~\mathrm{d}\tau,
\end{equation}
where $\boldsymbol{N}_\tau^i$ is the normal to $\partial D_{\phi_{k},\tau}$ inside $\partial D_{\phi_{i}}$, outward with respect to $D_{\phi_{k},\tau}$.

Finally, identifying $S\equiv\Pi_{\int}$ and $\mathcal{P}\equiv F_{k}\cap K$, we may also apply Lemma~\ref{lemma: 6} to the final term on the right-hand side of \eqref{eqn: int by parts inter 1} to obtain \begin{equation}\label{eqn: int by parts c}
\begin{aligned}
    &\int_{S \cap E_t}\jump{\bTheta g}~\mathrm{d}S=\int_{S \cap E_t \cap\bar K_1}{\bn_1\cdot\bTheta_1 g_1}~\mathrm{d}S+\int_{S \cap E_t \cap \bar K_2}{\bn_2\cdot\bTheta_2 g_2}~\mathrm{d}S\\
    &~=\int_0^t\int_{S \cap \partial D_{\phi_{k},\tau} \cap F_k \cap\bar K_1}{\bn_1\cdot\bTheta_1 g_1}\frac{w}{\lvert\partial_{\boldsymbol{N}_\tau^S}\phi_{k,\tau}\rvert}~\mathrm{d}\gamma~\mathrm{d}\tau+\int_0^t\int_{S \cap \partial D_{\phi_{k},\tau} \cap F_k \cap\bar K_2}{\bn_2\cdot\bTheta_2 g_2}\frac{w}{\lvert\partial_{\boldsymbol{N}_\tau^S}\phi_{k,\tau}\rvert}~\mathrm{d}\gamma~\mathrm{d}\tau\\
    &~=\int_0^t\int_{S \cap \partial D_{\phi_{k},\tau}\cap F_k}\jump{\bTheta g}\frac{w}{\lvert\partial_{\boldsymbol{N}_\tau^S}\phi_{k,\tau}\rvert}~\mathrm{d}\gamma~\mathrm{d}\tau,
\end{aligned}
\end{equation}
where $\boldsymbol{N}_\tau^S$ is the normal to $\partial D_{\phi_{k},\tau}$ inside $S$, outward with respect to $D_{\phi_{k},\tau}$. 
In the first equality, we split $S \cap E_t$ into $S \cap E_t\cap\bar K_1$ and $S \cap E_t\cap\bar K_2$ where $K_1$ and $K_2$ are the two cells that share the facet $S$. In the second equality, we apply Lemma~\ref{lemma: 6} to each of these integrals, identifying $S$ as $\Pi_{\int}$ and $\partial D_{\phi_{k},\tau}$ as $\Pi_\tau$. In the final equality, we combine the two integrals into a single integral on $S \cap \partial D_{\phi_{k},\tau}$ using the fact that $\boldsymbol{N}_\tau^S$ is the samx   e on $S\cap\bar K_1$ or $S\cap\bar K_2$ by continuity of $\partial D_{\phi_{k},\tau}$.

Substituting \eqref{eqn: int by parts a}, \eqref{eqn: int by parts b}, and \eqref{eqn: int by parts c} into \eqref{eqn: int by parts inter 1} and rewriting the sums over $K\in\mathcal{T}_h$ completes the proof.
\end{proof}
Thanks to the machinery developed above and in the previous sections, we may state the following result, whose proof follows steps similar to the proof of Theorem~6.11 in \citet{Berggren_2023}.
\begin{theorem}\label{theorem 3}
    Given a perturbation $\phi_{k,t}(\bx)=\phi_k(\bx)+tw(\bx)$,where $w$ is a linear Lagrangian basis function, the directional semideriative of a boundary integral
    \begin{equation}
        J(\phi_k)=\int_{\partial{D_{\phi_k}}\cap F_k}g~\mathrm{d}S,
    \end{equation}
    for $D_{\phi_k}$ 
    satisfying Assumption~\ref{assumption 2} and Assumption~\ref{assumption 4} and $g\in C^1(\overline{\mathcal{T}}_{\!h})$, satisfies
    \begin{equation}\label{thrm 3 main}
        \begin{aligned}
            \mathrm{d}J&(\phi_k)(w)=\lim_{t\rightarrow0^+}\frac{1}{t}\left(J(\phi_{k,t})-J(\phi_k)\right)\\
            &= -\int_{\partial D_{\phi_{k}}\cap F_k}\frac{\partial g}{\partial{\bn_{\phi_k}}}\frac{w}{\lvert\partial_{\bn_{\phi_k}}\phi_k\rvert}~\mathrm{d}S
            +\sum_{S\in\mathcal{F}_h\setminus\partial D}\int_{S \cap F_k \cap \partial D_{\phi_{k}}}\jump{\bn_{\phi_{k}} g}\frac{w}{\lvert\partial_{\boldsymbol{N}^S}\phi_{k}\rvert}~\mathrm{d}\gamma
            \\&\quad+\int_{\partial D \cap \partial D_{\phi_{k}} \cap F_k}\bn_{\partial D}\cdot\bn_{\phi_{k}} g\frac{w}{\lvert\partial_{\boldsymbol{N}^D}\phi_{k}\rvert}~\mathrm{d}\gamma
            +\sum_{D_{\phi_i}\in\mathcal{D}\setminus \{D_{\phi_k}\}}\int_{\partial D_{\phi_i} \cap \partial D_{\phi_{k}} \cap F_{ik}}\pm\bn_{\phi_i}\cdot\bn_{\phi_{k}} g\frac{w}{\lvert\partial_{\boldsymbol{N}^i}\phi_{k}\rvert}~\mathrm{d}\gamma,
        \end{aligned}
    \end{equation}
    where $\bn_{\phi_k}$ is the outward normal to $\partial D_{\phi_k}$; $\jump{\bn_{\phi_{k}} g} = \bn_{S,1}\cdot\bn_{\phi_{k},1} g_1+\bn_{S,2}\cdot\bn_{\phi_{k},2} g_2$ in which $\bn_{S,j}$ are the outward normals to the plane $S\in\mathcal{F}_h$ and $\bn_{\phi_{k},j}$ are the outward normals on $\partial D_{\phi_k}$ on either side of $S$; $\bn_{\partial D}$ and $\bn_{\phi_i}$ are the unit normals to $\partial D$ and $\partial D_{\phi_i}$ outward with respect to $D$ and $\partial D_{\phi_i}$, respectively; $\boldsymbol{N}^D$, $\boldsymbol{N}^i$, and $\boldsymbol{N}^S$ are the outward unit normals to $\partial D_{\phi_{k}}$ inside $\partial D$, $\partial D_{\phi_i}$, and $S$, respectively. The sign in the final term is positive if $D_{\phi_i}\in\mathcal{D}_p$ and negative if $D_{\phi_i}\in\mathcal{D}_c$. In addition, the index $i$ in the final term corresponds to $D_{\phi_i}\in\mathcal{D}\setminus D_{\phi_k}$ under the sum.
\end{theorem}
\begin{proof}
    Suppose $\phi_{t}(\bx)=\phi(\bx)+tw(\bx)$ for some $t \in (0, t_{max}]$, according to Assumption~\ref{assumption 3} and suppose that we have $E_t=(D_{\phi_k}\setminus\overline{D_{\phi_{k},t}})\cap F_{k}$. Owing to Assumption~\ref{assumption 2}, almost everywhere on $\partial D_{\phi_{k},\tau}$, for $\tau\in[0,t]$, the outwards normal can be written as
    \begin{equation}
        \bn_{\phi_{k},\tau}=\left.\frac{\bnabla\phi_{k,\tau}}{\lVert\bnabla\phi_{k,\tau}\rVert}\right\rvert_{\partial D_{\phi_{k},\tau}}.
    \end{equation}
    Suppose $K\in\mathcal{T}_h$ has a non-empty intersection with $E_t$. Identifying $E_t\cap K\equiv\mathcal{E}_t$, we may apply Lemma~\ref{lemma: lemma 3} to obtain a diffeomorphism $\hat{\boldsymbol{X}}$ on $\mathcal{E}_t$, or equivalently on $E_t\cap K$. Using this, we can define a function $\boldsymbol{\Psi}_K^{(t)}:E_t\cap K\rightarrow\mathbb{R}^d$ such that
    \begin{equation}
        \boldsymbol{\Psi}_K^{(t)}(\hat{\boldsymbol{X}}\left(\bx_S,\tau)\right)=\frac{\bnabla\phi_{k,\tau}(\hat{\boldsymbol{X}}\left(\bx_S,\tau)\right)}{\left\lVert\bnabla\phi_{k,\tau}(\hat{\boldsymbol{X}}\left(\bx_S,\tau)\right)\right\rVert},\quad \forall(\bx_S,\tau)\in S\times(0,t).
    \end{equation}
    We may then define $\boldsymbol{\Psi}^{(t)}\in L^\infty(E_t)^d$ such that $\left.\boldsymbol{\Psi}^{(t)}\right\rvert_{E_t\cap K}=\boldsymbol{\Psi}_K^{(t)}$. By construction, we then have
    \begin{equation}\label{eqn: proof thrm 3 normal inter}
        \left.\boldsymbol{\Psi}^{(t)}\right\rvert_{\partial D_{\phi_{k},\tau}\cap K}=\left.\bn_{\phi_{k},\tau}\right\rvert_{\partial D_{\phi_{k},\tau}\cap K},\quad\forall\tau\in(0,t).
    \end{equation}
    The function $\boldsymbol{\Psi}^{(t)}$ now fulfils the conditions for $\bTheta$ in Lemma~\ref{lemma: int by parts}, which implies that for $g\in\prod_{K\in\mathcal{T}_h} C^1(\overline{K})$,
    \begin{equation}\label{eqn: int by parts, thrm 3 inter 1}
        \begin{aligned}
        \sum_{K\in\mathcal{T}_h}\int_{E_t\cap K}[g\bnabla\cdot\boldsymbol{\Psi}^{(t)}+\boldsymbol{\Psi}^{(t)}&\cdot\bnabla g]~\mathrm{d}\bx=\int_{\partial E_t\setminus (F_k^\partial \cup \partial D) }\bn\cdot\boldsymbol{\Psi}^{(t)} g~\mathrm{d}S
        \\&+\int_0^t\int_{\partial D \cap \partial D_{\phi_{k},\tau} \cap F_k}\bn_{\partial D}\cdot\boldsymbol{\Psi}^{(t)} g\frac{w}{\lvert\partial_{\boldsymbol{N}_\tau^D}\phi_{k,\tau}\rvert}~\mathrm{d}\gamma~\mathrm{d}\tau
        \\&+\sum_{D_{\phi_i}\in\mathcal{D}\setminus \{D_{\phi_k}\}}\int_0^t\int_{\partial D_{\phi_i} \cap \partial D_{\phi_{k},\tau} \cap F_{ik}}\pm\bn_{\phi_i}\cdot\boldsymbol{\Psi}^{(t)} g\frac{w}{\lvert\partial_{\boldsymbol{N}_\tau^i}\phi_{k,\tau}\rvert}~\mathrm{d}\gamma~\mathrm{d}\tau
        \\&+\sum_{S\in\mathcal{F}_h\setminus\partial D}\int_0^t\int_{S \cap F_k \cap \partial D_{\phi_{k},\tau}}\jump{\boldsymbol{\Psi}^{(t)} g}\frac{w}{\lvert\partial_{\boldsymbol{N}_\tau^S}\phi_{k,\tau}\rvert}~\mathrm{d}\gamma~\mathrm{d}\tau,
    \end{aligned}
    \end{equation}
    where we use Assumption~\ref{assumption 4} to ensure that the intersection $\partial D_{\phi_i} \cap \partial D_{\phi_{k},\tau}$ does not change to become empty or non-empty for some $i$ depending on whether we take $t\rightarrow0^+$ or $t\rightarrow0^-$.

    Using Lemma~\ref{lemma: 4} and \eqref{eqn: proof thrm 3 normal inter}, we can rewrite each term in the sum in left-hand side of \eqref{eqn: int by parts, thrm 3 inter 1} as
    \begin{equation}\label{eqn: thrm 3 b}
        \begin{aligned}
            \int_{E_t\cap K}[g\bnabla\cdot\boldsymbol{\Psi}^{(t)}+\boldsymbol{\Psi}^{(t)}\cdot\bnabla g]~\mathrm{d}\bx &= \int_0^t\int_{\partial D_{\phi_{k},\tau}\cap F_k\cap K}[g\bnabla\cdot\bn_{\phi_{k},\tau}+\bn_{\phi_{k},\tau}\cdot\bnabla g]\frac{w}{\lvert\partial_{\bn}\phi_{k,\tau}\rvert}~\mathrm{d}S~\mathrm{d}\tau\\
            &=\int_0^t\int_{\partial D_{\phi_{k},\tau}\cap F_k\cap K}\frac{\partial g}{\partial{\bn_{\phi_{k},\tau}}}\frac{w}{\lvert\partial_{\bn}\phi_{k,\tau}\rvert}~\mathrm{d}S~\mathrm{d}\tau,
        \end{aligned}
    \end{equation}
    where we have identified $E_t\cap K\equiv\mathcal{E}_t$, $\Pi_\tau\equiv\partial D_{\phi_{k},\tau}\cap K$, and $\mathcal{P}\equiv F_k\cap K$, and used that $\bnabla\cdot\bn_{\phi_{k},\tau}$ is the mean curvature that vanishes because $\partial D_{\phi_{k},\tau}\cap K$ is planar.

    By construction, the first term on the right-hand side of \eqref{eqn: int by parts, thrm 3 inter 1} consists only of the boundaries $\partial D_{\phi_{k}}$ and $\partial D_{\phi_{k},t}$ and the integrand satisfies
    \begin{equation}
        \left.\bn\cdot\boldsymbol{\Psi}^{(t)} g\right\rvert_{\partial E_t\setminus (F_k^\partial \cup \partial D)}=\begin{cases}
            g&\text{on }\partial D_{\phi_{k}}\cap E_t\\
            -g&\text{on }\partial D_{\phi_{k},t}\cap E_t
        \end{cases},
    \end{equation}
    where the minus sign comes from the fact that $\bn$ is the outward normal to $E_t$, which aligns with $\boldsymbol{\Psi}^{(t)}$ on $\partial D_{\phi_k}$ but points in the opposite direction on $\partial D_{\phi_{k},t}$. Using this, we can write the first term on the right-hand side of \eqref{eqn: int by parts, thrm 3 inter 1} as
    \begin{equation}\label{eqn: thrm 3 a}
        \int_{\partial E_t\setminus (F_k^\partial \cup \partial D) }\bn\cdot\boldsymbol{\Psi}^{(t)} g~\mathrm{d}S = \int_{\partial D_{\phi_{k}}\cap E_t}g~\mathrm{d}S - \int_{\partial D_{\phi_{k},t}\cap E_t}g~\mathrm{d}S = \int_{\partial D_{\phi_{k}}\cap F_k}g~\mathrm{d}S - \int_{\partial D_{\phi_{k},t}\cap F_k}g~\mathrm{d}S,
    \end{equation}
    by definition of $E_t$.

    Substituting \eqref{eqn: thrm 3 b} and \eqref{eqn: thrm 3 a} into \eqref{eqn: int by parts, thrm 3 inter 1} and rearranging, we obtain
    \begin{equation}
        \begin{aligned}
            &\int_{\partial D_{\phi_{k}}\cap F_k}g~\mathrm{d}S - \int_{\partial D_{\phi_{k},t}\cap F_k}g~\mathrm{d}S =
            \\&=\int_0^t\int_{\partial D_{\phi_{k},\tau}\cap F_k}\frac{\partial g}{\partial{\bn_{\phi_{k},\tau}}}\frac{w}{\lvert\partial_{\bn_{\phi_k,\tau}}\phi_{k,\tau}\rvert}~\mathrm{d}S~\mathrm{d}\tau
            -\sum_{S\in\mathcal{F}_h\setminus\partial D}\int_0^t\int_{S \cap F_k \cap \partial D_{\phi_{k},\tau}}\jump{\bn_{\phi_{k},\tau} g}\frac{w}{\lvert\partial_{\boldsymbol{N}_\tau^S}\phi_{k,\tau}\rvert}~\mathrm{d}\gamma~\mathrm{d}\tau
            \\&\quad-\int_0^t\int_{\partial D \cap \partial D_{\phi_{k},\tau} \cap F_k}\bn_{\partial D}\cdot\bn_{\phi_{k},\tau} g\frac{w}{\lvert\partial_{\boldsymbol{N}_\tau^D}\phi_{k,\tau}\rvert}~\mathrm{d}\gamma~\mathrm{d}\tau
            \\&\quad-\sum_{D_{\phi_i}\in\mathcal{D}\setminus \{D_{\phi_k}\}}\int_0^t\int_{\partial D_{\phi_i} \cap \partial D_{\phi_{k},\tau} \cap F_{ik}}\pm\bn_{\phi_i}\cdot\bn_{\phi_{k},\tau} g\frac{w}{\lvert\partial_{\boldsymbol{N}_\tau^i}\phi_{k,\tau}\rvert}~\mathrm{d}\gamma~\mathrm{d}\tau.
        \end{aligned}
    \end{equation}
    Dividing by $t$ and taking the limit completes the proof.
\end{proof}
\begin{remark}
    As discussed in Remark~\ref{remark: point wise 1}, when $d=2$ the expressions on the intersections of boundaries become pointwise evaluations.
\end{remark}
\begin{remark}
    To avoid introducing additional geometric objects that may lose meaning in a general $\mathbb{R}^d$ setting, we do not use the triple product as in Lemma~6.10 and Theorem~6.11 of \citet{Berggren_2023}. Note that if the triple product were used, the last three terms in \eqref{thrm 3 main} would have a factor of $-1$, and in the case of a single level-set function we recover Theorem~3 of \citet{WEGERT2025118203}.
\end{remark}
\begin{remark}
Theorem~\ref{theorem 3} was proven under the assumption that $g\in C^1(\overline{\mathcal T}_{\!h})$, independent of $\phi_k$.
For the derivation of one of  directional derivatives in Section~3.6 below, we will, however, encounter a case when the integrand also is a function of $\phi_k$, that is, when the objective function reads
\begin{equation}
\hat J(\phi_k) = \int_{\partial D_{\phi_k}\cap F_k} g(\phi_k)\,\mathrm dS.
\end{equation}
Denote by $g(t)$ the function $t\mapsto g(\phi_{k,t})$.
Provided that $t\mapsto g(t)$ and $t\to g'(t)$ are continuous and $g(t), g'(t)\in C^1(\overline{\mathcal T}_{\!h})$ in the vicinity of $t=0$, we may use the product rule of differentiation to conclude that
\begin{equation}
d\hat J(\phi_k)(w) =   \int_{\partial D_{\phi_k}\cap F_k} g'\,\mathrm dS + dJ(\phi_k)(w),
\end{equation}
where $dJ(\phi_k)(w)$ is given by expression~\eqref{thrm 3 main}.
\end{remark}

\subsection{Shape calculus example}\label{subsec: example shape calc}\noindent
To demonstrate how these derivatives can be used in the context of shape and topology optimisation, we return to the example discussed in Section~\ref{sec: multi-phase unfitted fes}. In particular,  consider minimising the total thermal dissipation for both phases, that is, the quantity
\begin{equation}\label{eqn: F def}
    F(U,U;\phi_1,\phi_2)=(\bq_1(u_1),\bnabla u_1)_{\Omega_1}+(\bq_2(u_2),\bnabla u_2)_{\Omega_2},
\end{equation}
quadratic in $U = [u_1,u_2]$.
For the optimisation, we introduce the \textit{reduced objective function}
\begin{equation}\label{eqn: reduced objective}
J(\phi_1,\phi_2) = F\bigl(U(\phi),U(\phi)(\phi_1,\phi_2);\phi_1,\phi_2\bigr),
\end{equation}
where $\phi = [\phi_1,\phi_2]$ and
\begin{equation}\label{eqn: state eqn example}
\begin{gathered}
U\in\mathcal{V}_{1,h}\times\mathcal{V}_{2,h} \quad\text{such that}
\\
A(U,V;\phi_1,\phi_2)=l(V;\phi_1,\phi_2)\qquad \forall V\in\mathcal{V}_{1,h}\times\mathcal{V}_{2,h},
\end{gathered}
\end{equation}
 and consider the optimisation problem
\begin{equation}\label{eqn: optim problem example}
\min_{(\phi_1,\phi_2)\in\mathcal{V}_{h}^2} J(\phi_1, \phi_2).
\end{equation}

For this example, assume that the flux is given by $\bq_i=\Ab_i\bnabla u_i$, $i=1,2$, where matrices $\Ab_i$ are symmetric, positive definite, and constant on $\overline D_{\phi_i}$.
For the weighting parameters $\kappa_1$ and $\kappa_2$, we take inspiration from \citet{10.1016/j.cma.2019.01.009_2019} and define these on $\Gamma_{ij}$ to be
\begin{align}\label{eqn: weighting aniso}
    \kappa_1=\frac{c_j}{c_i+c_j},\quad
    \kappa_2=\frac{c_i}{c_i+c_j},
\end{align}
where $c_i=\Ab_i:(\bn\otimes\bn)=\bn\cdot\Ab_i\bn$, which accounts for the possible material anisotropy.

Next, recall that $\Omega_1(\phi_1,\phi_2)=D_{\phi_1}\cap  \overline{D}_{\phi_2}^\complement$, $\Omega_2(\phi_1,\phi_2)=D_{\phi_1}\cap D_{\phi_2}$, and $\Gamma_{12}(\phi_1,\phi_2)=D_{\phi_1}\cap\partial D_{\phi_2}$ (see Fig.~\ref{fig:fig2-new}). As a result, $F$, $A$ and $l$ can be written as
\begin{equation}
    F(U,U;\phi_1,\phi_2)=(\Ab_1\bnabla u_1,\bnabla u_1)_{D_{\phi_1}\cap  \overline{D}_{\phi_2}^\complement}+(\Ab_2\bnabla u_2,\bnabla u_2)_{D_{\phi_1}\cap D_{\phi_2}},
\end{equation}
\begin{equation}\label{eqn: A explicit}
    \begin{aligned}
        A(U,V;\phi_1,\phi_2) &= (\Ab_1\bnabla v_1,\bnabla v_1)_{D_{\phi_1}\cap  \overline{D}_{\phi_2}^\complement}+(\Ab_2\bnabla u_2,\bnabla v_2)_{D_{\phi_1}\cap D_{\phi_2}}+j([u_1,u_2],[v_1,v_2])\\
        &\quad
        -(\mean{\Ab\bnabla u},\jump{v})_{D_{\phi_1}\cap\partial D_{\phi_2}}-(\mean{\Ab\bnabla v},\jump{u})_{D_{\phi_1}\cap\partial D_{\phi_2}}\\&\quad+(\mu\jump{u},\jump{v})_{D_{\phi_1}\cap\partial D_{\phi_2}},
    \end{aligned}
\end{equation}
and
\begin{equation}\label{eqn: l explicit}
    l(V;\phi_1,\phi_2) = \sum_i(g,v_i)_{\Gamma_N}.
\end{equation}

To differentiate $J$ with respect to $\phi_1$ and $\phi_2$, we use Céa's formal method \citep{10.1051/m2an/1986200303711_1986}, which is a version of the method of Lagrange multipliers applied to shape optimisation problems.
Since this method also is encoded in our later use of forward mode automatic differentiation for shape optimisation, we will give a brief review of the approach for the current optimisation problem~\eqref{eqn: optim problem example}.

The method is based on a Lagrangian functional constructed by subtracting the residual of the state equation from the the objective function, that is,
\begin{equation}\label{eqn: Lagrangian def}
    \mathcal{L}(P,\Lambda;\phi_1,\phi_2)=F(P,P;\phi_1,\phi_2)-A(P,\Lambda;\phi_1,\phi_2)+l(\Lambda;\phi_1,\phi_2),
\end{equation}
where $P=[p_1,p_2]$ and $\Lambda=[\lambda_1,\lambda_2]$.

Since the Lagrangian is linear in $\Lambda$ (the test function in the variational form), differentiation of the Lagrangian with respect to $\Lambda$ in the direction $W$ generates the state equation residual, that is,
\begin{equation}
d_\Lambda\mathcal L(P,\Lambda;\phi_1,\phi_2)(W) \equiv
\frac d{dt}\mathcal L(P, \Lambda + tW;\phi_1,\phi_2)\big|_{t=0} =
-A(P, W; \phi_1,\phi_2) + l(W;\phi_1,\phi),
\end{equation}
which vanishes for each $W\in \mathcal V_{1,h}\times\mathcal V_{2,h}$ when $P=U\coloneqq[u_1,u_2]$ due to state equation~\eqref{eqn: state eqn example}.

Moreover, differentiation of the Lagrangian with respect to $P$ in the direction $R$ generates a residual of the so-called adjoint equation, that is,
\begin{equation}\label{eqn: adj residual example}
\begin{aligned}
d_P\mathcal L(P,\Lambda;\phi_1,\phi_2)(R)  &\equiv
\frac d{dt}\mathcal L(P+tR,\Lambda;\phi_1,\phi_2)\big|_{t=0}
= d_P F(P,P;\phi_1,\phi_2)(R) - d_P A(P, \Lambda;\phi_1,\phi_2)(R)
\\&
= 2F(P,R;\phi_1,\phi_2) - A(R, \Lambda;\phi_1,\phi_2).
\end{aligned}
\end{equation}
The adjoint state $\hat\Lambda$ is the $\Lambda$ for which the residual~\eqref{eqn: adj residual example}, evaluated for $P=U$, vanishes for each $R\in\mathcal V_{1,h}\times\mathcal V_{2,h}$, that is, the solution to the following problem: find $\hat{\Lambda}\in\mathcal{V}_{1, h}\times\mathcal{V}_{2, h}$ such that
\begin{equation}\label{eqn: adjoint equation example}
    A(R, \hat\Lambda;\phi_1,\phi_2) = 2F(U,R;\phi_1,\phi_2)\qquad\forall R\in\mathcal{V}_{1, h}\times\mathcal{V}_{2, h}.
\end{equation}

Finally, the directional derivative of $J$ in the direction $w$ is obtained by differentiating the Lagrangian with respect to $\phi_i$ in the direction $w$ and evaluating the resulting expression for $P=U$ and $\Lambda = \hat\Lambda$, that is,
\begin{equation}\label{eqn: lagrangian deriv}
\begin{aligned}
d_{\phi_i}J(\phi_1,\phi_2)(w) &= d_{\phi_i}\mathcal L(P,\Lambda;\phi_1,\phi_2)\big|_{P=U, \Lambda=\hat\Lambda}(w)
\\
&= d_{\phi_i} F(U,U;\phi_1,\phi_2)(w) - d_{\phi_i} A(U,\hat\Lambda;\phi_1,\phi_2)(w) + d_{\phi_i}l(\hat\Lambda;\phi_1,\phi_2)(w).
\end{aligned}
\end{equation}
\begin{remark}
It may not be entirely obvious why the directional derivative of the reduced objective function is obtained according to expression~\eqref{eqn: lagrangian deriv}, which is why we here give a short justification.
This explanation also highlights the main reason why the procedure is merely formal.
According to definitions~\eqref{eqn: reduced objective} and~\eqref{eqn: Lagrangian def}, the reduced objective satisfies
\begin{equation}
J(\phi_1,\phi_2) = \mathcal L\bigl(U(\phi),\Lambda;\phi_1,\phi_2\bigr)
\qquad\forall\Lambda\in\mathcal{V}_{1, h}\times\mathcal{V}_{2, h}.
\end{equation}
Differentiating with respect to $\phi_i$ in direction $w$ and applying the chain rule yields
\begin{equation}\label{eqn: dJ=dL equality}
d_{\phi_i}J(\phi_1,\phi_2)(w) = d_P\mathcal L\bigl(U(\phi),\Lambda;\phi_1,\phi_2\bigr)(dU)
+ d_{\phi_i}\mathcal L\bigl(U(\phi),\Lambda;\phi_1,\phi_2\bigr)(w)
\qquad\forall\Lambda\in\mathcal{V}_{1, h}\times\mathcal{V}_{2, h},
\end{equation}
where $\phi = [\phi_1,\phi_2]$, and where
\begin{equation}
dU = \frac d{dt} U(\phi + tw)\big|_{t=0}
\end{equation}
is the \textit{shape derivative} of the state $U$.
Now note that, due to adjoint equation~\eqref{eqn: adjoint equation example}, the left side of expression~\eqref{eqn: adj residual example} vanishes when $P=U$ and $\Lambda=\hat\Lambda$, that is,
\begin{equation}\label{eqn: dPL vanishes}
d_P\mathcal L\bigl(U(\phi),\hat\Lambda;\phi_1,\phi_2\bigr)_{}(R) = 0
\qquad\forall R\in\mathcal{V}_{1, h}\times\mathcal{V}_{2, h}.
\end{equation}
Thus, if the shape derivative $dU$ exists and resides in $\mathcal{V}_{1, h}\times\mathcal{V}_{2, h}$, we may choose $R = dU$ in equation~\eqref{eqn: dPL vanishes}, which makes the first term on the right side of expression~\eqref{eqn: dJ=dL equality} to vanish for $\Lambda=\hat\Lambda$, leaving
\begin{equation}
d_{\phi_i}J(\phi_1,\phi_2)(w) = d_{\phi_i}\mathcal L\bigl(U(\phi),\hat\Lambda;\phi_1,\phi_2\bigr)(w),
\end{equation}
which was the claim in expression~\eqref{eqn: lagrangian deriv}.
Hence, the main assumption in this scheme is that \textit{the shape derivative $dU$ exists and resides in $\mathcal{V}_{1, h}\times\mathcal{V}_{2, h}$}.
The existence and regularity of shape derivatives of state variables are generally delicate issues.
For instance, if body-fitted discretisations are used, and mesh deformations are utilised to keep the mesh conformed to changes in the level-set function, the shape derivative of a continuous state variable, if it exists, will typically contain jump discontinuities~\cite[\S~4.2]{Berggren_2023}.
That is, the shape derivative of the state is then less regular than the state itself.
In contrast, for the unfitted discretisation used here, the shape derivative, if it exists, will stay in the same space as the state variable itself~\cite[\S~5.4]{Berggren_2023}.
\end{remark}

Using the procedure above, we  arrive at the following result for the directional derivative of $J$. Note that for brevity, we only state the derivative with respect to $\phi_1$. 

    Suppose $D_{\phi_1}$ and $D_{\phi_2}$ satisfy Assumption~\ref{assumption 2}. Then, given a perturbation $\phi_{1,t}(\bx)=\phi_1(\bx)+tw(\bx)$, where $w$ is a Lagrangian basis function, the directional derivative of objective function~\eqref{eqn: reduced objective} is
    \begin{equation}\label{eqn: shape deriv}
        \begin{aligned}
            \mathit d_{\phi_1}J(\phi_1,\phi_2)(w)&=\int_{\partial D_{\phi_1}\cap \overline{D}_{\phi_2}^\complement}(\Ab_1\bnabla u_1)\cdot(\bnabla \hat\lambda_1-\bnabla u_1)\frac{w}{\lvert\partial_{\bn_{\phi_1}}\phi_1\rvert}~\mathrm{d}S
            \\&\quad+\int_{\partial D_{\phi_1}\cap D_{\phi_2}}(\Ab_2\bnabla u_2)\cdot(\bnabla \hat\lambda_2-\bnabla u_2)\frac{w}{\lvert\partial_{\bn_{\phi_1}}\phi_1\rvert}~\mathrm{d}S
            \\&\quad\added{-}\int_{\partial D_{\phi_1}\cap\partial D_{\phi_2}}\left(\mean{\Ab\bnabla u}\cdot\jump{\hat\lambda}+\mean{\Ab\bnabla \hat\lambda}\cdot\jump{u}-\mu\jump{u}\cdot\jump{\hat\lambda}\right)\frac{w}{\lvert\partial_{\bN^2}\phi_1\rvert}~\mathrm{d}\gamma,
        \end{aligned}
    \end{equation}
    where $\hat\lambda_1$ and $\hat\lambda_2$ satisfy the adjoint problem: find $[\hat\lambda_1,\hat\lambda_2]\in\mathcal{V}_{1,h}\times\mathcal{V}_{2,h}$ such that
    \begin{equation}
        \begin{aligned}
            &(\Ab_1\bnabla r_1,\bnabla \hat\lambda_1)_{D_{\phi_1}\cap \overline{D}_{\phi_2}^\complement}+(\Ab_2\bnabla r_2,\bnabla \hat\lambda_2)_{D_{\phi_1}\cap D_{\phi_2}}+j([r_1,r_2],[\hat\lambda_1,\hat\lambda_2])+(\mu\jump{r},\jump{\hat\lambda})_{D_{\phi_1}\cap\partial D_{\phi_2}}\\
            &\quad
            -(\mean{\Ab\bnabla \hat\lambda},\jump{r})_{D_{\phi_1}\cap\partial D_{\phi_2}}-(\mean{\Ab\bnabla r},\jump{\hat\lambda})_{D_{\phi_1}\cap\partial D_{\phi_2}}=2F(U, [r_1,r_2];\phi_1,\phi_2),
        \end{aligned}
    \end{equation}
    for all $[r_1,r_2]\in\mathcal{V}_{1,h}\times\mathcal{V}_{2,h}$. The derivation of this result is given in the Supplementary Material.

\begin{remark}
The differentiation method outlined above constitutes a general procedure to devise expressions for reduced gradients, a procedure that is applicable to a wide range of shape and topology optimisation problems.
However, in some particular cases, the optimisation problem can be formulated to possess symmetries such that no separate adjoint state is needed.
In fact, the thermal dissipation example we consider here can be reformulated to provide such a symmetry by minimising $j(\phi_1,\phi_2) = l(U;\phi_1,\phi_2) = \sum_i (g,u_i)_{\Gamma_N}$ instead of $J(\phi_1,\phi_2) = F(U,U; \phi_1, \phi_2)$.
After making this change and carrying out the above differentiation scheme, it turns out that the adjoint state for the reformulated problem satisfies $\tilde\lambda_i = u_i$; that is, a separate adjoint equation will not be needed.
The gradient expression for the reformulated problem will then be
\begin{equation}
\begin{aligned}
\mathit d_{\phi_1}j(\phi_1,\phi_2)(w)&=\int_{\partial D_{\phi_1}\cap \overline{D}_{\phi_2}^\complement}(\Ab_1\bnabla u_1)\cdot\bnabla u_1\frac{w}{\lvert\partial_{\bn_{\phi_1}}\phi_1\rvert}~\mathrm{d}S
+\int_{\partial D_{\phi_1}\cap D_{\phi_2}}(\Ab_2\bnabla u_2)\cdot\bnabla u_2\frac{w}{\lvert\partial_{\bn_{\phi_1}}\phi_1\rvert}~\mathrm{d}S
\\&\qquad\added{-}\int_{\partial D_{\phi_1}\cap\partial D_{\phi_2}}\Bigl(2\mean{\Ab\bnabla u}\cdot\jump{u}-\mu\jump{u}\cdot\jump{u}\Bigr)\frac{w}{\lvert\partial_{\bN^2}\phi_1\rvert}\,\mathrm{d}\gamma,
\end{aligned}
\end{equation}
The quantity $l(U;\phi_1,\phi_2)$ is the measure of heat power supplied to the system through the boundary $\Gamma_N$, whereas $F(U,U; \phi_1, \phi_2)$ is the heat power dissipated in the domain.
Due to energy conservation, these quantities are the same for the problem before discretisation, and the difference between these problem formulations is only due to discretisation effects --- namely the Nitsche interface terms --- that will vanish as $h\to0$.
\end{remark}

This example serves to show how level-set shape derivatives \textit{can} be computed by hand for a multi-phase problem. However, not only are the resulting expressions complicated, particularly for the case given in the Supplementary Material, their implementation in source code is tedious, error prone, and narrowly specific, making it an arduous task to update for new cases.
To this end, in the next section, we develop automatic shape differentiation techniques to provide an alternate way to compute these quantities.

\section{Automatic shape differentiation}\label{sec: asd}\noindent
In this section, we discuss how the automatic differentiation techniques developed by \citet{WEGERT2025118203,GridapTopOpt} can be extended as an alternative to using the formulas devised in Theorems~\ref{theorem 1}, \ref{theorem 2}, and \ref{theorem 3}.

\subsection{Preliminaries}\noindent
There are two important concepts that we discuss to motivate the rest of this section. The first is forward-mode automatic differentiation via \textit{dual numbers}. This approach involves computing the derivative of a composition of functions by propagating dual numbers from input to output \cite{Baydin_Pearlmutter_Radul_Siskind_2018}. A dual number is defined as $a+b\varepsilon$ with the property $\varepsilon^2=0$ where $a\in\mathbb{R}$ is the real component and $b\in\mathbb{R}$ is the dual component. The Taylor series of a differentiable function $f(a+b\varepsilon)$ about $a$ is
\begin{equation}\label{eqn: taylor duals}
    f(a+b\varepsilon)=f(a)+b f'(a)\varepsilon
\end{equation}
where the higher order terms vanish because $\varepsilon^2=0$. The derivative of $f$ can therefore be computed exactly by evaluating the dual component of $f(a+b\varepsilon)$. It is also possible to define $N$-dimensional dual numbers $\boldsymbol{\varepsilon}$ with the property that $\varepsilon_i\varepsilon_j=0$ \citep{Revels_Lubin_Papamarkou_2016}. Using this, one can show for a function $g:\mathbb{R}^N\rightarrow\mathbb{R}$,
\begin{equation}\label{eqn: taylor duals higher dim}
    g\left(\bx+\boldsymbol{\varepsilon}\right)=g(\bx)+\sum_{i=1}^N\frac{\partial g}{\partial x_i}\varepsilon_i.
\end{equation}
As a result, all the partial derivatives of $g$ at $\boldsymbol{x}$ can be computed with a single evaluation of $g(\boldsymbol{x}+\boldsymbol{\varepsilon})$. We call this a \textit{dualised} function. We refer to \citet{Revels_Lubin_Papamarkou_2016} and the references therein for further discussion of vectorised forward-mode automatic differentiation.

In \citet{WEGERT2025118203}, dual numbers were utilised to compute level-set shape derivatives for unfitted discretisations defined by a single level-set function $\phi$ for a fixed integrand.
As an example, assume that we wish to compute the shape derivative of an integral over a domain in which $\phi<0$.
That is, we will compute the derivative of the integral with respect to the nodal values of the level-set function.
Fundamentally, this approach relies on propagating dual numbers through the \textit{local} geometric maps from the reference background cell to the physical cut subcell, illustrated in Figure~\ref{fig: wegert recreated}.
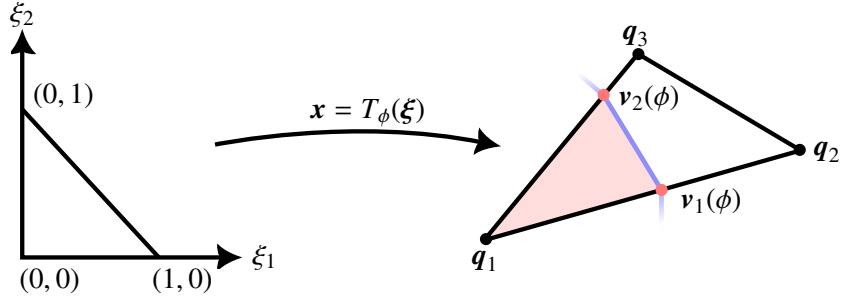
\begin{figure}[t]
    \centering
    \def\svgwidth{0.65\textwidth}
    \input{Unfitted_Discretisation_mapping_only_tex}
    \caption{An illustration of the geometric map $T_\phi:\hat{K}\rightarrow K_{\rm sub}(\phi)$ from the reference cell $\hat{K}$ to a cut subcell $K_{\rm sub}(\phi)$ (pink). The intersection points $\boldsymbol{v}_1(\phi)$ and $\boldsymbol{v}_2(\phi)$ of the cut subcell are computed using $\phi(\boldsymbol{q}_1)$, $\phi(\boldsymbol{q}_2)$, and $\phi(\boldsymbol{q}_3)$ via interpolation. Figure recreated from \citet{WEGERT2025118203}.}
    \label{fig: wegert recreated}
\end{figure}
In particular, given a cut subcell, the derivative can be evaluated by: 1) propagating dual numbers through the cell-wise degrees of freedom of the level-set function; 2) computing the \textit{dualised} nodal coordinates of the intersection of the interface defined by the level set and the boundary of the background cell; 3) constructing the geometric map from the reference background cell to the physical subcell; 4) evaluating the cell contribution to the integral via numerical quadrature and extracting the dual components. The local derivatives can then be assembled into a global gradient using routine finite element assembly. When the integrand $f$ is not fixed (that is, it depends on the solution to a PDE that in turn depends on $\phi$), an additional step is added in which the adjoint problem is constructed and solved.
Thus, the shape derivative can be computed in tandem with the assembly of the integral.

This approach assumes that there is a single level-set function that is used to cut the background mesh~\citep{WEGERT2025118203}. This means that intersections of the interface and the boundary of a background cell can be computed a priori using a precomputed non-conforming triangulation of the interface. It is therefore possible to implement forward-mode automatic differentiation without propagating dual numbers through the cutting algorithm. For the case of several level-set functions, more complex intersections between the zero level sets can occur within a cell. Therefore, propagation of dual numbers through the cutting algorithm is necessary and requires special care for the sake of efficiency. We outline our approach below.

\subsection{Cell-wise polytope cutter}\noindent
In the following, we develop a polytope cutter to extend the automatic shape differentiation approach of \citet{WEGERT2025118203} to multiple level-set functions, $\phi_i$ for $i=1,\dots,L$. We make the following design choices for the cutter:
\begin{enumerate}
    \item[(i)] Cell-wise: the cutter should operate on a cell-wise basis so that derivatives can be computed locally.
    \item[(ii)] Polytopal splitting: the cutter should be able to split an arbitrary polytope contained inside a cell $K$ with a local level-set finite element function $\phi^{(K)}(\bx)$. We denote this operation \texttt{Split}.
    \item[(iii)] Recursive: to support several level-set functions $\phi^{(K)}_i(\bx)$, the subsequent split polytopes from (ii) should themselves be splittable by the cutter with another level-set function.
    \item[(iv)] Simplex polytopes: once (iii) has been completed for all level-set functions, the resulting polytopes should be simplexified for the purpose of numerical quadrature. We denote this operation \texttt{Simplexify}.
\end{enumerate}
Figure \ref{fig:fig3} and \ref{fig:fig4} show examples of cell-wise \texttt{Split} and \texttt{Simplexify} operations in two and three dimensions. For \texttt{Split}, we reuse the methods from \citet[Sec. 4.3.1,][]{WEGERT2025118203}. In particular, we adapt the framework and algorithms in STLCutters.jl \cite{Badia_Martorell_Verdugo_2022,Martorell_Badia_2025} for the case where cuts are determined using a local level-set finite element function $\phi^{(K)}(\bx)$ inside a given cell $K$. For \texttt{Simplexify}, we adapt the algorithms from Gridap~\citep{Verdugo2022}. We refer the reader to \citet{Badia_Martorell_Verdugo_2022,Martorell_Badia_2025,Verdugo2022} for further discussion of the \texttt{Split} and \texttt{Simplexify} algorithms.
\begin{figure}[t]
    \centering
    \begin{subfigure}{0.49\textwidth}
        \centering
        \def\svgwidth{\textwidth}
        \input{split2d_tex}
        \caption{}
        \label{fig:fig3a}
    \end{subfigure}
    \begin{subfigure}{0.49\textwidth}
        \centering
        \def\svgwidth{\textwidth}
        \input{split3d_tex}
        \caption{}
        \label{fig:fig3b}
    \end{subfigure}
    \caption{Visualisation of the polytope splitting operation \texttt{Split} in (a) two dimensions and (b) three dimensions.}
    \label{fig:fig3}
 \end{figure}
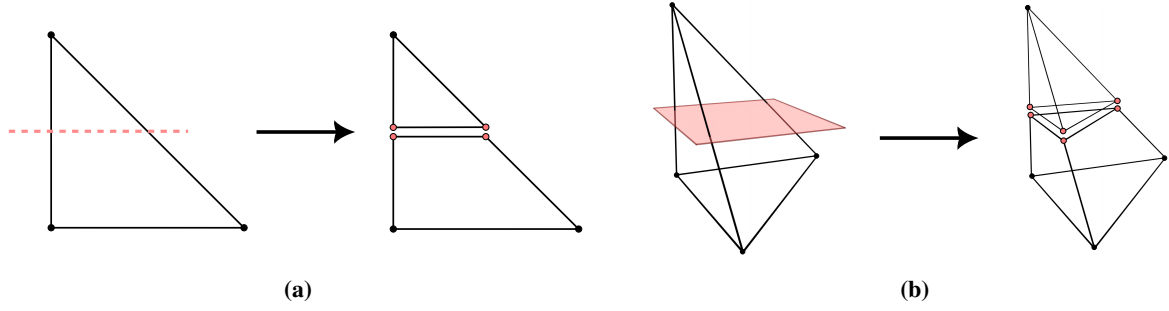
 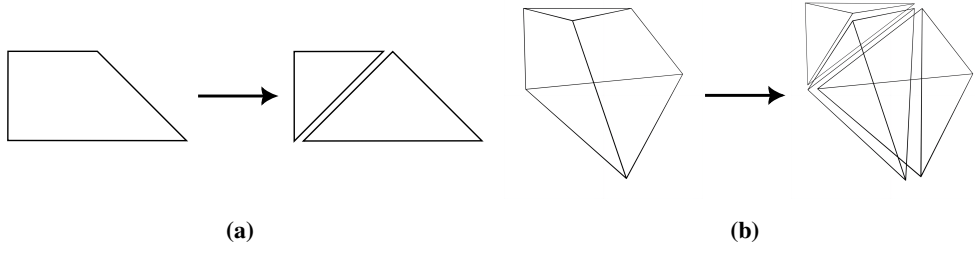
\begin{figure}[t]
    \centering
    \begin{subfigure}{0.4\textwidth}
        \centering
        \def\svgwidth{\textwidth}
        \input{simplexify2d_tex}
        \caption{}
        \label{fig:fig4a}
    \end{subfigure}
    \begin{subfigure}{0.4\textwidth}
        \centering
        \def\svgwidth{\textwidth}
        \input{simplexify3d_tex}
        \caption{}
        \label{fig:fig4b}
    \end{subfigure}
    \caption{Visualisation of the polytope simplexify operation \texttt{Simplexify} in (a) two dimensions and (b) three dimensions.}
    \label{fig:fig4}
\end{figure}

Using the \texttt{Split} and \texttt{Simplexify} operations, we can develop the \texttt{RecursiveCut} operation. Here, we utilise a binary tree with nodes that hold data related to a polytope $\mathcal{P}$ and a state vector $\boldsymbol{S}$ with components $S_j$, $j=1,\dots,L$. When $\mathcal{P}$ is a cell, $S_j$ indicates whether that cell is in or out with respect to each $\phi_j$. When $\mathcal{P}$ is a facet, $S_j$ indicates whether that facet is in, out, or is a cut generated by the zero level-set of $\phi_j$. We assume the following functions implemented for the binary tree, where $\mathcal{N}$ refers to a node within the tree:
\begin{itemize}
    \itemsep0em
    \item \texttt{Polytope}($\mathcal{N}$): Return the polytope stored in node $\mathcal{N}$.
    \item \texttt{State}($\mathcal{N}$): Return $\boldsymbol{S}$ which gives the in, out, or cut state stored in node $\mathcal{N}$.
    \item \texttt{Children}($\mathcal{N}$): Return the children, $\mathcal{N}_{\rm Left}$ and $\mathcal{N}_{\rm Right}$, of node $\mathcal{N}$.
    \item \texttt{Leaves}($\mathcal{N}$): Return a list of polytopes and their respective states at the end points, also known as leaves, of the branch associated with node $\mathcal{N}$ of the binary tree. When $\mathcal{N}$ is the root node, this returns all leaves of the binary tree.
    \item \texttt{Root}($\mathcal{P}$,$L$): Initialise the root node of a binary tree where $\mathcal{P}$ is the root cell or facet and $L$ is the total number of level-set functions.
\end{itemize}
Finally, the states in $\boldsymbol{S}$ are denoted to either be $\rm{IN}=-1$, OUT$=1$, or CUT$=0$. The recursive cutting algorithm is then as follows.
\begin{algorithm}[H]
    \caption{$\texttt{RecursiveCut}(\mathcal{N}, \boldsymbol{\phi}^{(K)},L,s)$}
    \justifying
    \begin{algorithmic}[1]
    \IF{$\texttt{IsEmpty}(\boldsymbol{\phi}^{(K)})$}
        \RETURN
    \ENDIF
    \STATE $\mathcal{P}\gets\texttt{Polytope}(\mathcal{N}),\quad \boldsymbol{S}\gets\texttt{State}(\mathcal{N}),\quad n \gets L - \texttt{Length}(\boldsymbol{\phi}^{(K)}) + 1$
    \IF{$s = n$}
        \STATE $S_n \gets \text{CUT}$
        \STATE $\texttt{RecursiveCut}(\mathcal{N}, [\phi_2^{(K)},\dots,\phi_N^{(K)}],L,s)$
        \RETURN
    \ENDIF
    \STATE $(\text{state}, \mathcal{P}_\text{in}, \mathcal{P}_\text{out}) \gets \texttt{Split}(\mathcal{P}, \phi_1^{(K)})$
    \IF{$\text{state} = \text{IN}$ \OR $\text{state} = \text{OUT}$}
        \STATE ${S}_n \gets \text{state}$
        \STATE $\texttt{RecursiveCut}(\mathcal{N}, [\phi_2^{(K)},\dots,\phi_N^{(K)}],L, s)$
        \RETURN
    \ENDIF
    \STATE $(\mathcal{N}_{\textrm{Left}}, \mathcal{N}_{\textrm{Right}}) \gets \texttt{Children}(\mathcal{N})$
    \STATE $\texttt{Polytope}(\mathcal{N}_{\textrm{Left}}) \gets \mathcal{P}_\text{in},\quad \texttt{Polytope}(\mathcal{N}_{\textrm{Right}}) \gets \mathcal{P}_\text{out}$
    \STATE $\texttt{RecursiveCut}(\mathcal{N}_{\textrm{Left}}, [\phi_2^{(K)},\dots,\phi_N^{(K)}],L, s)$
    \STATE $\texttt{RecursiveCut}(\mathcal{N}_{\textrm{Right}}, [\phi_2^{(K)},\dots,\phi_N^{(K)}],L, s)$
    \RETURN
    \end{algorithmic}\label{alg:recursively_cut}
\end{algorithm}
\noindent In the above, we denote $\boldsymbol{\phi}^{(K)}$ to be the vector of remaining local level-set finite element functions, $L$ is the total number of level-set functions, and $s$ is an integer specifying whether to skip a particular level-set function when cutting facets. The integer $s$ specifies which level-set was used to generate an initial facet and is only non-zero when cutting facets. In line 10, we split the polytope $\mathcal{P}$ with a local level-set finite element function. If the level-set function does not split $\mathcal{P}$, then we specify whether that polytope is IN or OUT with respect to the level-set function in `state'. In the case that $\mathcal{P}$ is split, we take `state = nothing' and return the two resulting polytopes from splitting $\mathcal{P}$. Note that Algorithm~\ref{alg:recursively_cut} returns nothing and instead modifies the binary tree.

Using the algorithm and operations described above, we now describe \texttt{CutAndSimplexifyCell} whereby a cell $K\in\mathcal{T}_h$ is cut recursively using a vector of level-set functions and then the resulting polytopes are simplexified. The resulting simplices make up $[\cup_i (D_{\phi_i}\cap K)]\cup[\cup_i (D_{\phi_i}^\complement\cap K)]$ and cover $K$.
\begin{algorithm}[H]
    \caption{$\texttt{CutAndSimplexifyCell}(K,\boldsymbol{\phi}^{(K)})$}
    \justifying
    \begin{algorithmic}[1]
    \STATE $L \gets \texttt{Length}(\boldsymbol{\phi}^{(K)})$
    \STATE $\mathcal{N}\gets\texttt{Root}(K,L)$
    \STATE $\texttt{RecursiveCut}(\mathcal{N}, \boldsymbol{\phi}^{(K)},L,0)$
    \STATE $\boldsymbol{\mathcal{L}}\gets\texttt{Leaves}(\mathcal{N})$
    \RETURN $\texttt{Simplexify}(\boldsymbol{\mathcal{L}})$
    \end{algorithmic}\label{alg:cut_and_simplexify_cell}
\end{algorithm}
\noindent Note that \texttt{Simplexify} is applied element-wise to the vector of polytopes and returns a list of all simplicies and their state $\boldsymbol{S}$. Figure~\ref{fig:fig5} shows a visualisation of this algorithm. Note that the first three rows of this figure show the polytopes associated with the nodes of the binary tree.
\begin{figure}[p]
    \centering
    \def\svgwidth{0.9\textwidth}
    \large
    \input{rec_cutter_tex}
    \caption{Visualisation of Algorithm~\ref{alg:cut_and_simplexify_cell} in which a two-dimensional simplex is recursively cut and simplexified with two level-set functions. The dashed blue and red lines are the cuts defined by $\phi_1$ and $\phi_2$ at the cell degrees of freedom, respectively. The coloured vertices represent the level-set functions used for splitting, for example, a blue and red dot means that both $\phi_1$ and $\phi_2$ contributed to the splitting process.}
    \label{fig:fig5}
\end{figure}
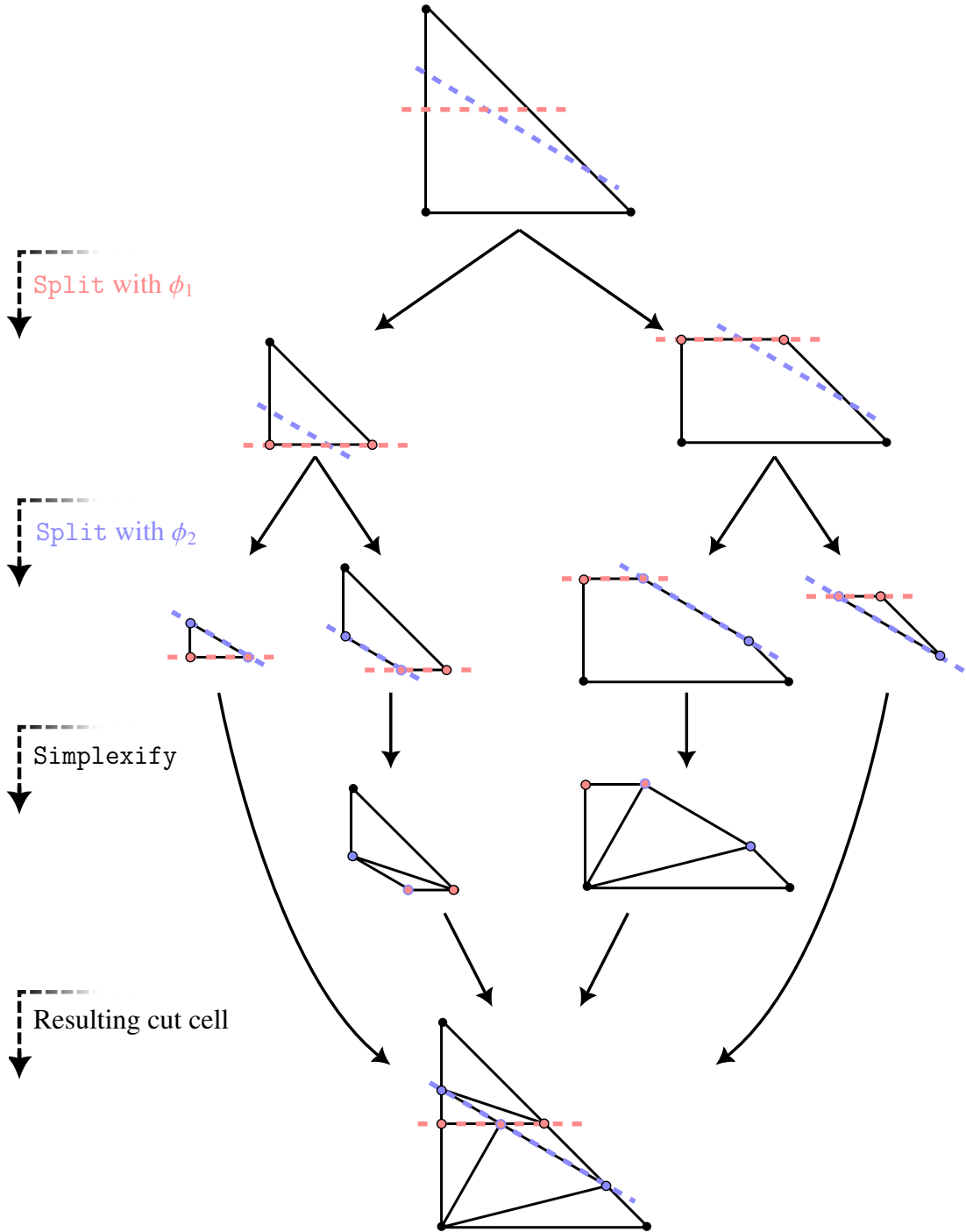

Finally consider the \texttt{CutAndSimplexifyFacet} algorithm. This algorithm splits a cell $K\in\mathcal{T}_h$ with the $i^{\rm th}$ level-set function, constructs the facet connecting the two subcells, then recursively cuts the facet with the remaining level-set functions, and returns a list of simplices making up the original facet, $\partial D_{\phi_i}\cap K$.
\begin{algorithm}[H]
    \caption{\texttt{CutAndSimplexifyFacet}$(K,\boldsymbol{\phi}^{(K)},i)$}
    \justifying
    \begin{algorithmic}[1]
    \STATE $(\text{state}, \mathcal{P}_\text{in}, \mathcal{P}_\text{out}) \gets \texttt{Split}(K, \phi_i^{(K)})$
    \IF{$\text{state} = \text{IN}$ \OR $\text{state} = \text{OUT}$}
        \STATE \textbf{continue}
    \ENDIF
    \STATE $\mathcal{F}_0 \gets \partial\mathcal{P}_\text{in} \cap \partial\mathcal{P}_\text{out}$,\quad$L \gets \texttt{Length}(\boldsymbol{\phi}^{(K)})$
    \STATE $\mathcal{N}\gets\texttt{Root}(\mathcal{F}_0,L)$
    \STATE $\texttt{RecursiveCut}(\mathcal{N}, \boldsymbol{\phi}^{(K)},L ,i)$
    \STATE $\boldsymbol{\mathcal{L}}\gets\texttt{Leaves}(\mathcal{N})$
    \RETURN $\texttt{Simplexify}(\boldsymbol{\mathcal{L}})$
    \end{algorithmic}\label{alg:cut_and_simplexify_facet}
\end{algorithm}
\noindent All of the facets $(\cup_i\partial D_{\phi_i})\cap K$ can then be constructed by looping over the index of each level-set function and applying Algorithm~\ref{alg:cut_and_simplexify_facet}.

\begin{figure}[p]
    \centering
    \def\svgwidth{0.9\textwidth}
    \large
    \input{rec_cutter_facet_tex}
    \caption{Visualisation of Algorithm~\ref{alg:cut_and_simplexify_facet} in which a one-dimensional facet is computed from a level-set function, recursively cut with the remaining level-set functions, then simplexified. Here we start with a facet $\mathcal{F}_0$ computed by splitting the polytope with $\phi_2$ whose cut is shown by the dashed blue line. The resulting facet is then cut with $\phi_1$ and $\phi_3$ whose cuts are shown by the red and green dashed lines, respectively. The coloured vertices represent the level-set functions used for splitting. Note that in this example, facets are line segements and so \texttt{Simplexify} is simply the identity.}
    \label{fig:fig5_facet}
\end{figure}
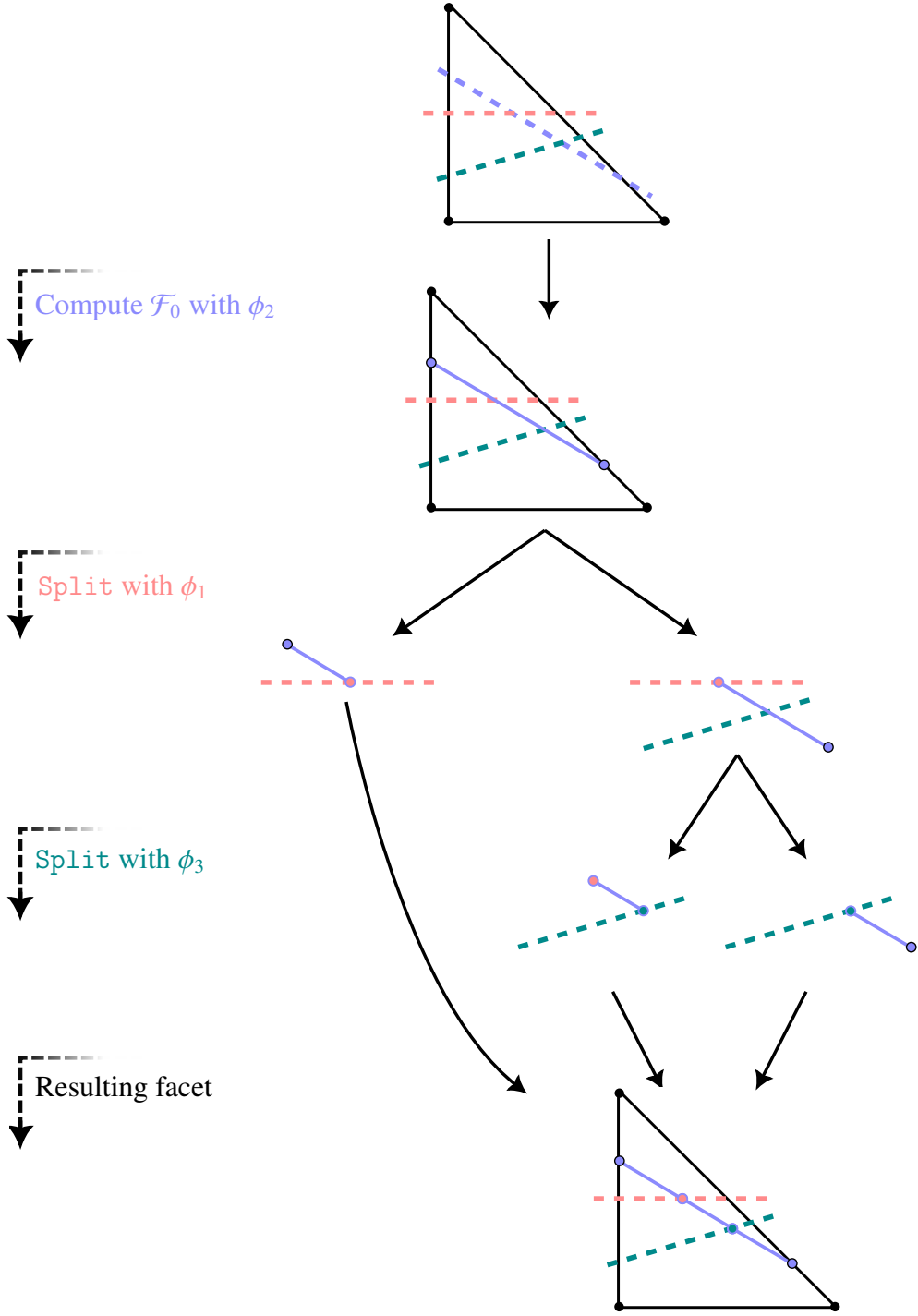

Applying Algorithm~\ref{alg:cut_and_simplexify_cell} and \ref{alg:cut_and_simplexify_facet} over the entire mesh constructs a non-conforming triangulation for $(\cup_i D_{\phi_i})\cup(\cup_i D_{\phi_i}^\complement)$ that covers the entire computational domain $D$. Figure~\ref{fig: building-phys-trian a} shows an example. Using the states $\boldsymbol{S}$ for each simplex, we can isolate a particular domain $D_{\phi_j}$ or apply more complicated set operations to several domains. Figure~\ref{fig: building-phys-trian b} shows an example of the resulting triangulation of $D_{\phi_1}\cap \overline{D}_{\phi_2}^\complement\cap \overline{D}_{\phi_3}^\complement$.

\begin{figure}[t]
    \centering
    \begin{subfigure}{0.32\textwidth}
        \centering
        \includegraphics[width=\textwidth]{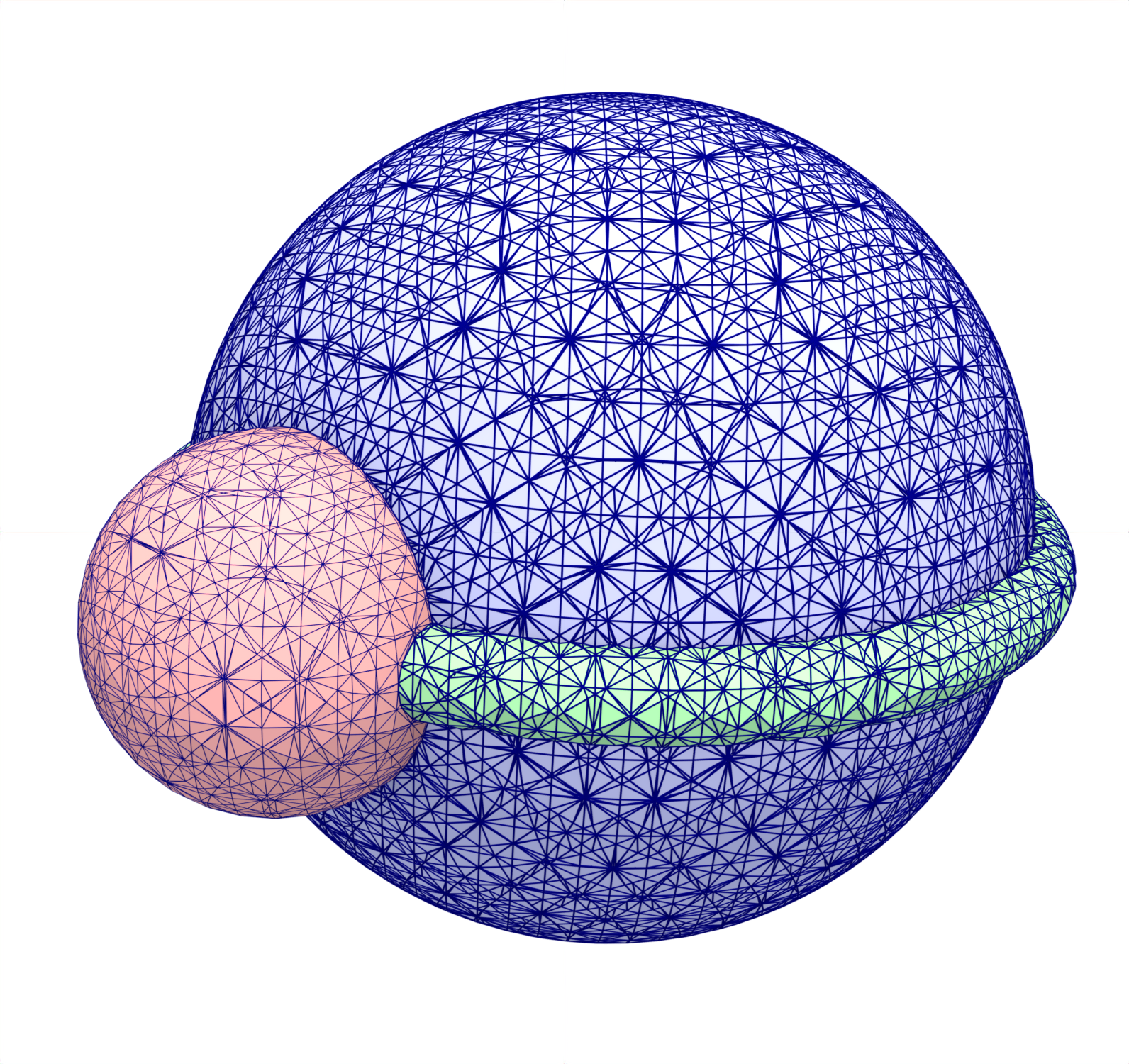}
        \caption{}
        \label{fig: building-phys-trian a}
    \end{subfigure}
    \begin{subfigure}{0.32\textwidth}
        \centering
        \includegraphics[width=\textwidth]{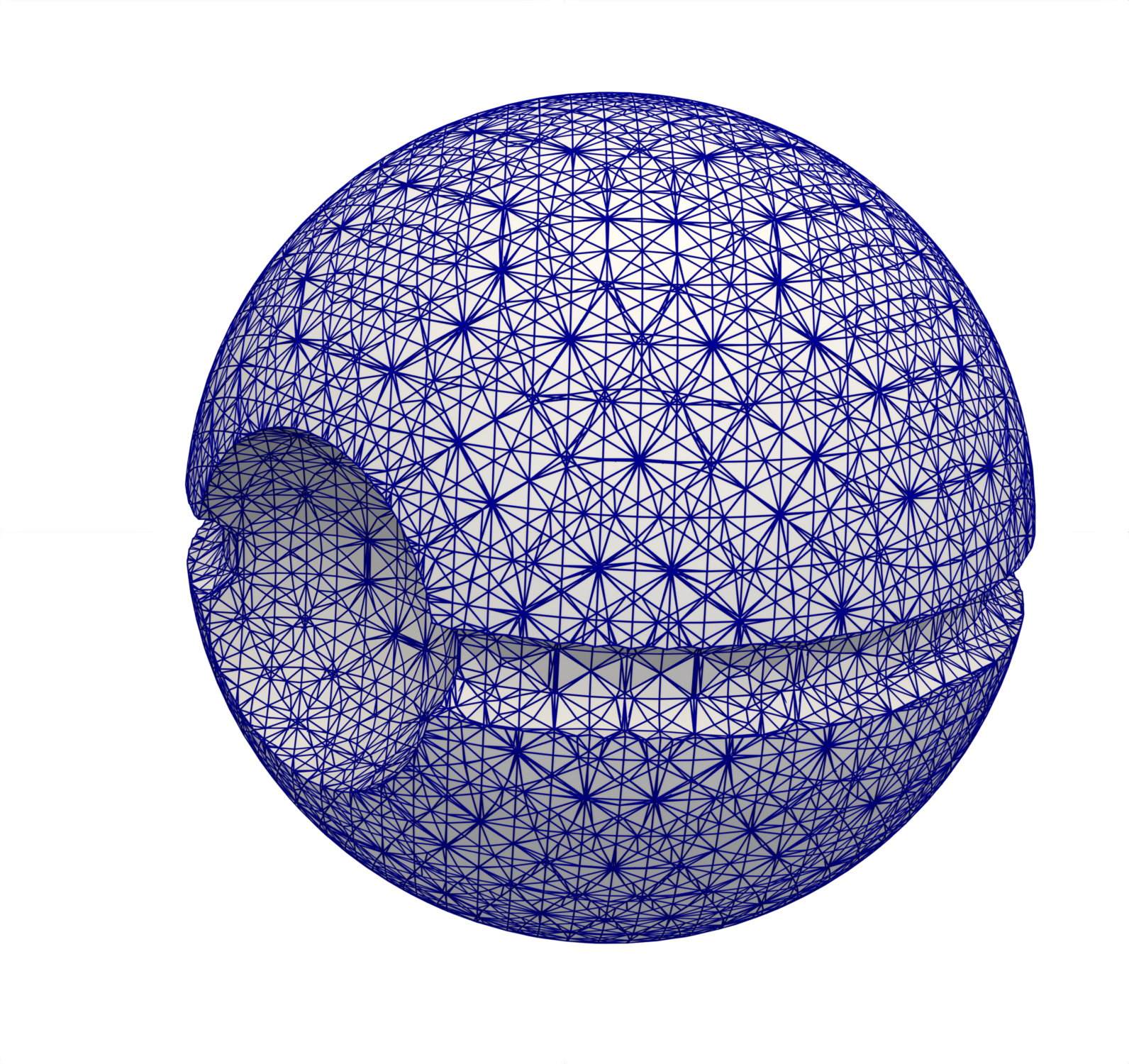}
        \caption{}
        \label{fig: building-phys-trian b}
    \end{subfigure}
    \begin{subfigure}{0.32\textwidth}
        \centering
        \includegraphics[width=0.8\textwidth]{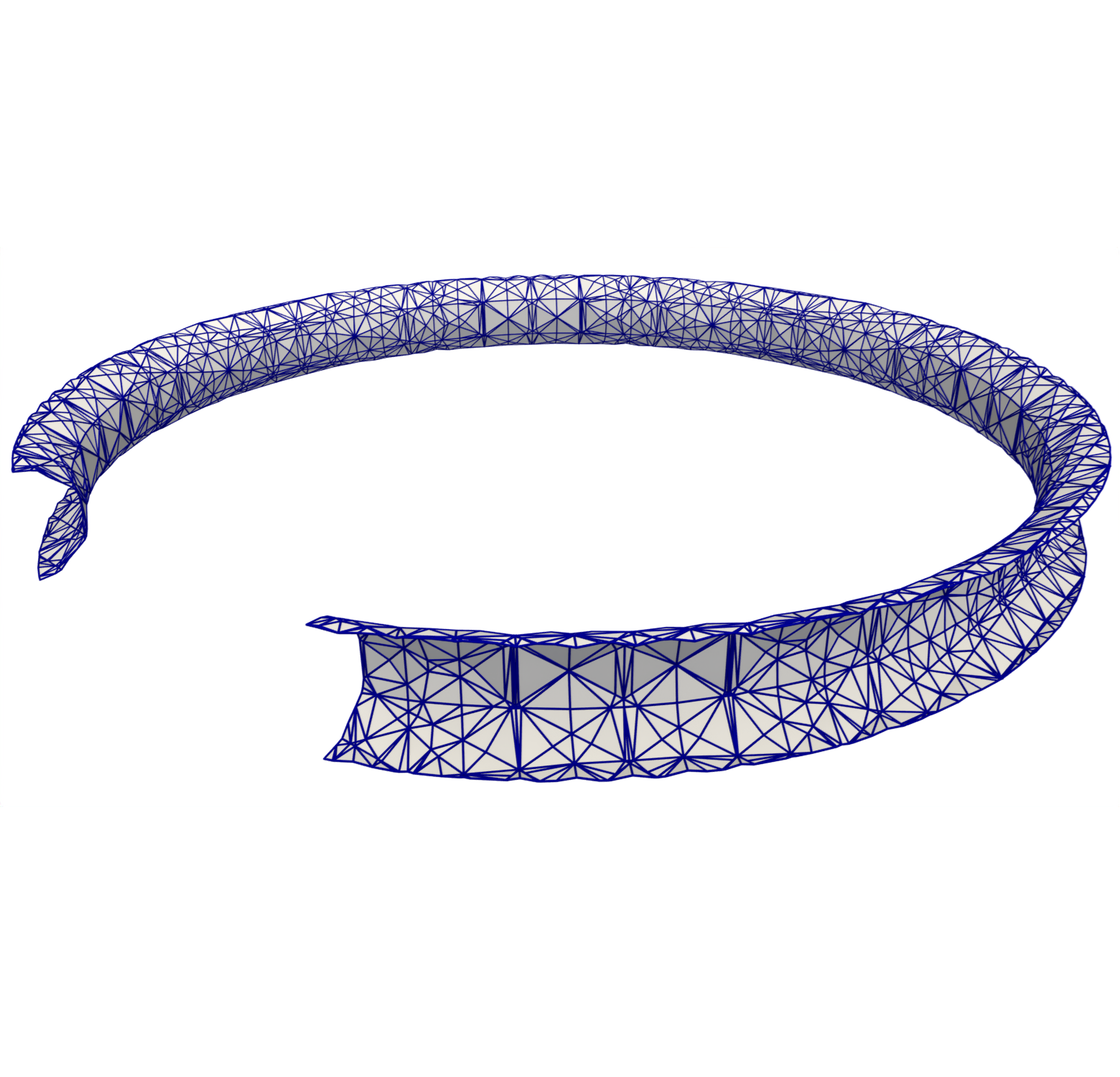}
        \caption{}
        \label{fig: building-phys-trian c}
    \end{subfigure}
    \caption{(a) Three non-conforming triangulations of $D_{\phi_i}$ ($i=1,2,3$) constructed from level-set functions $\phi_i$ via Algorithm~\ref{alg:cut_and_simplexify_cell}. (b) The triangulation of a domain $\Omega=D_{\phi_1}\cap \overline{D}_{\phi_2}^\complement\cap \overline{D}_{\phi_3}^\complement$ that is constructed by masking with the vectors $\boldsymbol{S}$ for each simplex that result from cutting. (c) The triangulation of the interface $\partial\Omega\cap\partial D_{\phi_3}$ resulting from Algorithm~\ref{alg:cut_and_simplexify_facet} and masking. The code to generate these is given in Listing~\ref{code:implementation example}.}
    \label{fig: building-phys-trian}
\end{figure}


\subsection{Forward-mode automatic differentiation}\label{sec: forward-mode ad}\noindent
We now consider how to utilise forward-mode automatic differentiation through Algorithms~\ref{alg:recursively_cut}, \ref{alg:cut_and_simplexify_cell}, and \ref{alg:cut_and_simplexify_facet}. These algorithms are cell-wise operations and rely on the local level-set finite element functions $\phi_i^{(K)}(\bx)$. To compute directional derivatives in $\phi_i$ using dual numbers, we can construct a \textit{dualised} version of $\phi_i^{(K)}(\bx)$ whereby the basis functions $\boldsymbol{\eta}_{i}$ are assigned a vector of duals $\boldsymbol{\varepsilon}_i$, where $i$ denotes that these are related to the $i^\mathrm{th}$ level set. These dual numbers then append the corresponding values of $\phi_i^{(K)}$ at each basis node. The following \texttt{Dualise} operation constructs a vector of level-set functions in which the $k^\mathrm{th}$ level set is dualised.
\begin{algorithm}[H]
    \caption{\texttt{Dualise}$(\boldsymbol{\phi}^{(K)},k)$}
    \justifying
    \begin{algorithmic}[1]
     \STATE $\boldsymbol{\eta}_k\gets\texttt{Basis}({\phi}^{(K)}_k)$
     \STATE $\phi^{(K)}_{\varepsilon,k}(\bx) \gets \sum_{n=1}^{N_j}(\phi_{n,k}^{(K)}+\varepsilon_n)\eta_{n,k}(\bx)$,
     \STATE $\phi_{\epsilon,j}^{(K)} \gets \phi_{j}^{(K)}$ for $j\neq k$
    \RETURN $\boldsymbol{\phi}^{(K)}_\varepsilon(\bx)$
    \end{algorithmic}\label{alg:dualise}
\end{algorithm}
\noindent In the above, $\eta_{n,j}$ denotes the $n^{\rm th}$ basis of $\phi_j^{(K)}$, $\phi_{n,j}^{(K)}$ is the value of $\phi_j^{(K)}$ at the node where the $n^{\rm th}$ basis is unity, and $\delta_{jk}$ is the Kronecker delta. Note that only the $k^\mathrm{th}$ index of $\boldsymbol{\phi}^{(K)}_\varepsilon(\bx)$ that relates to level-set function $\phi_k^{(K)}$ is dualised.

We now consider how a shape derivative under a perturbation of $\phi_k$ can be computed using forward-mode automatic differentiation. In the following, assume that $\Omega\subset\mathbb R^d$ is constructed using set operations on the domains $D_{\phi_i}$. The level-set shape derivative of a functional $J(\boldsymbol{\phi})=\int_{\Omega(\boldsymbol{\phi})}f~\mathrm{d}\bx$, where $f\in C^0(\overline{\mathcal T}_h)$, with respect to $\phi_k$ can be computed as follows:
\begin{algorithm}[H]
    \caption{$\texttt{ShapeDerivative}(f,\Omega,\mathcal{T}_h,\boldsymbol{\phi},k)$}
    \justifying
    \begin{algorithmic}[1]
    \STATE $\mathcal{T}_{h,\rm cut}\gets\{K\in\mathcal{T}_h:{\Large{\cup}}_{i}[K\cap \partial D_{\phi_i}]\neq\emptyset\}$
    \FOR{$K \in \mathcal{T}_{h,\rm cut}$}
        \STATE $\boldsymbol{\phi}^{(K)}\gets \boldsymbol{\phi}\rvert_{K}$
        \STATE $\boldsymbol{\phi}^{(K)}_\varepsilon\gets\texttt{Dualise}(\boldsymbol{\phi}^{(K)},k)$
        \STATE $(\boldsymbol{\mathcal{P}}_\mathrm{subcells},\boldsymbol{S})\gets\texttt{CutAndSimplexifyCell}(K,\boldsymbol{\phi}^{(K)}_\varepsilon)$
        \FOR{${\mathcal{P}}\in\texttt{Mask}(\boldsymbol{\mathcal{P}}_\mathrm{subcells},\boldsymbol{S},\Omega)$}
            \STATE $({T},{J}_T)\gets \texttt{CellMap}(K,{\mathcal{P}})$
            \STATE $\boldsymbol{\mathrm{d}J}_K \gets \boldsymbol{\mathrm{d}J}_K + \mathrm{Du}(\int_{{T}^{-1}({\mathcal{P}})}(f\circ T)J_T~\mathrm{d}\boldsymbol{\xi})$
        \ENDFOR
    \ENDFOR
    \RETURN $\boldsymbol{\mathrm{d}J}_\mathrm{glob}\gets\texttt{Assemble}(\boldsymbol{\mathrm{d}J})$
    \end{algorithmic}\label{alg:auto-diff}
\end{algorithm}
\noindent In the above, \texttt{Mask} should return the simplices that partition $\Omega$ as defined by the state vector $\boldsymbol{S}$; \texttt{CellMap}$(K,\mathcal{P})$ should return the map $T$ from the reference cell to the simplex $\mathcal{P}$ cell, as well as its Jacobian determinant $J_T$; \texttt{Assemble} should take a vector of local contributions and assemble them into a global vector using standard finite element assembly; $\rm{Du}$ returns the dual part of a dual number or vector; and $\boldsymbol{\mathrm{d}J}_K$ is a vector of local contributions in a cell $K$ with length equal to the number of basis functions of $\phi_k^{(K)}$. An equivalent algorithm can be used to compute the shape derivative of a functional defined on an interface by replacing the cell-wise cutting with facet-wise cutting and modifying the integrand accordingly.

\begin{remark}\label{remark: single pass}
    It is possible to compute all derivatives of $J$ in a single pass of Algorithm~\ref{alg:auto-diff}. This is achieved by modifying Algorithm~\ref{alg:dualise} to dualise all level-set functions by replacing $\varepsilon_n$ with dual numbers $\varepsilon_{nj}$ that satisfy $\varepsilon_{nj}\varepsilon_{ml}=0$ for all $n,j,m,l$. 
\end{remark}

\subsection{Implementation}\label{subsec: implementation}\noindent
We extend the GridapTopOpt package \citep{GridapTopOpt} to include the polytope cutting and forward-mode automatic differentiation algorithms discussed above.
These are compatible with the wider Gridap package ecosystem~\citep{Badia2020,Verdugo2022}. Our code depends on and extends GridapEmbedded \citep{GridapEmbedded}, which is the Gridap package that implements unfitted discretisations. As a result, much of the functionality available in GridapEmbedded (for instance the cell aggregation technique AgFEM \citep{Badia_Verdugo_Martin_2018}) could be used in our context, although we do not explore that in this work. Our implementation utilises Gridap's \texttt{LazyArray} functionality to cache and defer the execution of operations until the specific elements of an array are explicitly accessed. Furthermore, we interface with GridapDistributed \cite{Badia2022} so that our implementation can be leveraged in both serial and distributed CPU computing frameworks. The distributed computing implementation readily follows from the serial implementation because all computations are cell-wise and can be mapped over each local mesh partition.

Like Gridap, our implementation attempts to provide a clear and intuitive application programming interface (API). We demonstrate this in Listing~\ref{code:implementation example}, whereby we construct the domain $\Omega$ and interface $\Gamma$ shown in Figures~\ref{fig: building-phys-trian b} and \ref{fig: building-phys-trian c}, respectively, and computes the derivatives of a functional $J(\phi_i)=\int_{\Omega(\phi_i)}1~\mathrm{d}\bx+\int_{\Gamma(\phi_i)}1~\mathrm{d}S$ with respect to the three level-set functions $\phi_i$ for $i=1,\dots3$.

\begin{listing}[!t]
\caption{An example script demonstrating how the polytopal cutter and forward-mode automatic differentiation techniques can be leveraged in the wider Gridap ecosystem. The domain $\Omega$ and interface $\Gamma$ are shown in Figures~\ref{fig: building-phys-trian b} and \ref{fig: building-phys-trian c}.}
\begin{minted}[fontsize=\footnotesize,breaklines,linenos,xleftmargin=1.5em]{julia}
using Gridap, Gridap.Adaptivity, Gridap.Geometry
using GridapEmbedded, GridapEmbedded.LevelSetCutters
using GridapTopOpt
# Background mesh
base_model = CartesianDiscreteModel((0,1,0,1,0,1),(41,41,41))
ref_model = refine(UnstructuredDiscreteModel(base_model), refinement_method = "barycentric")
model = get_model(ref_model)
# Level-set functions and geometries Dφᵢ
order = 1
reffe = ReferenceFE(lagrangian,Float64,order)
Vφᵢ = MultiFieldFESpace([TestFESpace(model,reffe) for i in 1:3])
φ₁ = x->sqrt((x[1]-0.5)^2+(x[2]-0.5)^2+(x[3]-0.5)^2)-0.25
φ₂ = x->sqrt((x[1]-0.75)^2+(x[2]-0.5)^2+(x[3]-0.5)^2)-0.1
φ₃ = x->sqrt((0.25-sqrt((x[1]-0.5)^2+(x[2]-0.5)^2))^2 + (x[3]-0.5)^2) - 0.025
φhᵢ = interpolate([φ₁,φ₂,φ₃],Vφᵢ)
Dφ₁ = DiscreteGeometryFromFEFunction(φhᵢ[1],model)
Dφ₂ = DiscreteGeometryFromFEFunction(φhᵢ[2],model)
Dφ₃ = DiscreteGeometryFromFEFunction(φhᵢ[3],model)
# Triangulation of Ω and Γ
Ω_geo = Dφ₁ ∩ !Dφ₂ ∩ !Dφ₃
cutgeo = cut(PolytopalLevelSetCutter(),model,Ω_geo)
Ω = DifferentiableTriangulation(cutgeo,Ω_geo)
Γ = DifferentiableEmbeddedBoundary(cutgeo,Ω_geo,Dφ₃)
dΩ = Measure(Ω,2*order)
dΓ = Measure(Γ,2*order)
# Functional and gradient
J(φhᵢ) = ∫(1)dΩ + ∫(1)dΓ
dJ = gradient(J,φhᵢ)
vec_dJ = assemble_vector(dJ,Vφᵢ)
\end{minted}
\label{code:implementation example}
\end{listing}

Our forward-mode automatic differentiation implementation is fully compatible with GridapTopOpt's adjoint methods for arbitrary PDE-constrained maps. At the lowest level, GridapTopOpt uses Céa's method to compute derivatives for PDE-constrained functionals. For the case of a variational form similar as the one in \eqref{eqn: optim problem example}, the method proceeds analogously  as described in Section~\ref{subsec: example shape calc}.
In particular, the adjoint equation is obtained using Gridap's element-wise, forward-mode automatic differentiation to compute $d_P\mathcal{L}(P,\Lambda;\phi_1,\phi_2)(R)\big|_{P=U}$.
For this computation, $R$ and $\Lambda$ are, in each element, replaced with the local finite element basis functions, which assembled together produces the adjoint equation,  whose solution $\hat\Lambda$ satisfies $d_P\mathcal{L}(U, \hat\Lambda;\phi_1,\phi_2)=0$.
The computation of the adjoint equation through this method exhibits the same computational complexity as a regular assembly of a finite-element matrix.
The terms $d_{\phi_i}F$, $d_{\phi_i}A$, and $d_{\phi_i}l$ appearing in \eqref{eqn: lagrangian deriv} are then computed using the forward-mode automatic differentiation techniques proposed in Section~\ref{sec: forward-mode ad} instead of via Theorem~\ref{theorem 1}, \ref{theorem 2} and \ref{theorem 3}. For more complicated PDE-constrained maps (for instance $\phi$ mapping to $J(\phi)$ through several variational problems), the highest layer of GridapTopOpt's automatic differentiation system enables reverse-mode automatic differentiation tools such as Zygote.jl \citep{innes2019dontunrolladjointdifferentiating} to efficiently accumulate the derivative using the adjoint methods available in GridapTopOpt. For further discussion of this, we refer to the GridapTopOpt documentation.

\subsection{Verification and benchmarks}\label{subsec: valid and bench}\noindent
In Table~\ref{tab:verify} we verify our implementation of the forward-mode automatic differentiation implementation against the derivative expressions given in Theorems~\ref{theorem 1}, \ref{theorem 2}, and \ref{theorem 3} and finite differences for the geometries shown in Figure~\ref{fig:fig6}. Our implementation of the derivative formulas is not robust enough to handle randomly generated geometries (Fig.~\ref{fig:fig6c} \& \ref{fig:fig6d}), we therefore only use these geometries for validating the robustness of the automatic differentiation implementation against finite differences. For brevity, we only report relative errors in Table~\ref{tab:verify} for a few domain combinations. Our implementation in GridapTopOpt includes an extensive suite of unit tests that cover all derivative cases for one, two and three level-set functions as well as tests for solving PDEs on the unfitted discretisations generated by the polytopal cutter.

\begin{figure}[!t]
    \centering
    \begin{subfigure}{0.4\textwidth}
        \centering
        \includegraphics[width=0.8\linewidth]{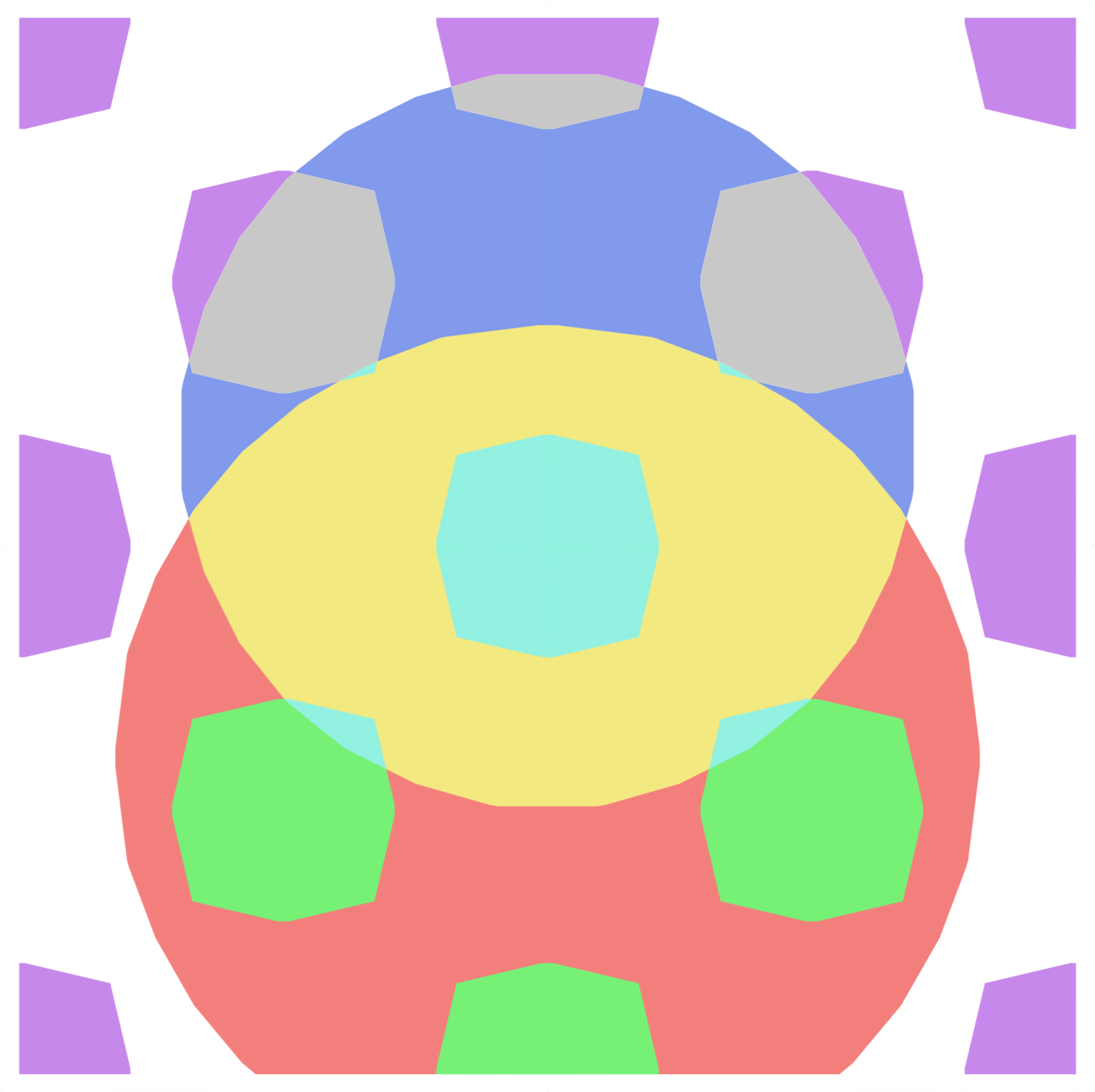}
        \caption{}
        \label{fig:fig6a}
    \end{subfigure}
        \begin{subfigure}{0.4\textwidth}
        \centering
        \includegraphics[width=0.8\linewidth]{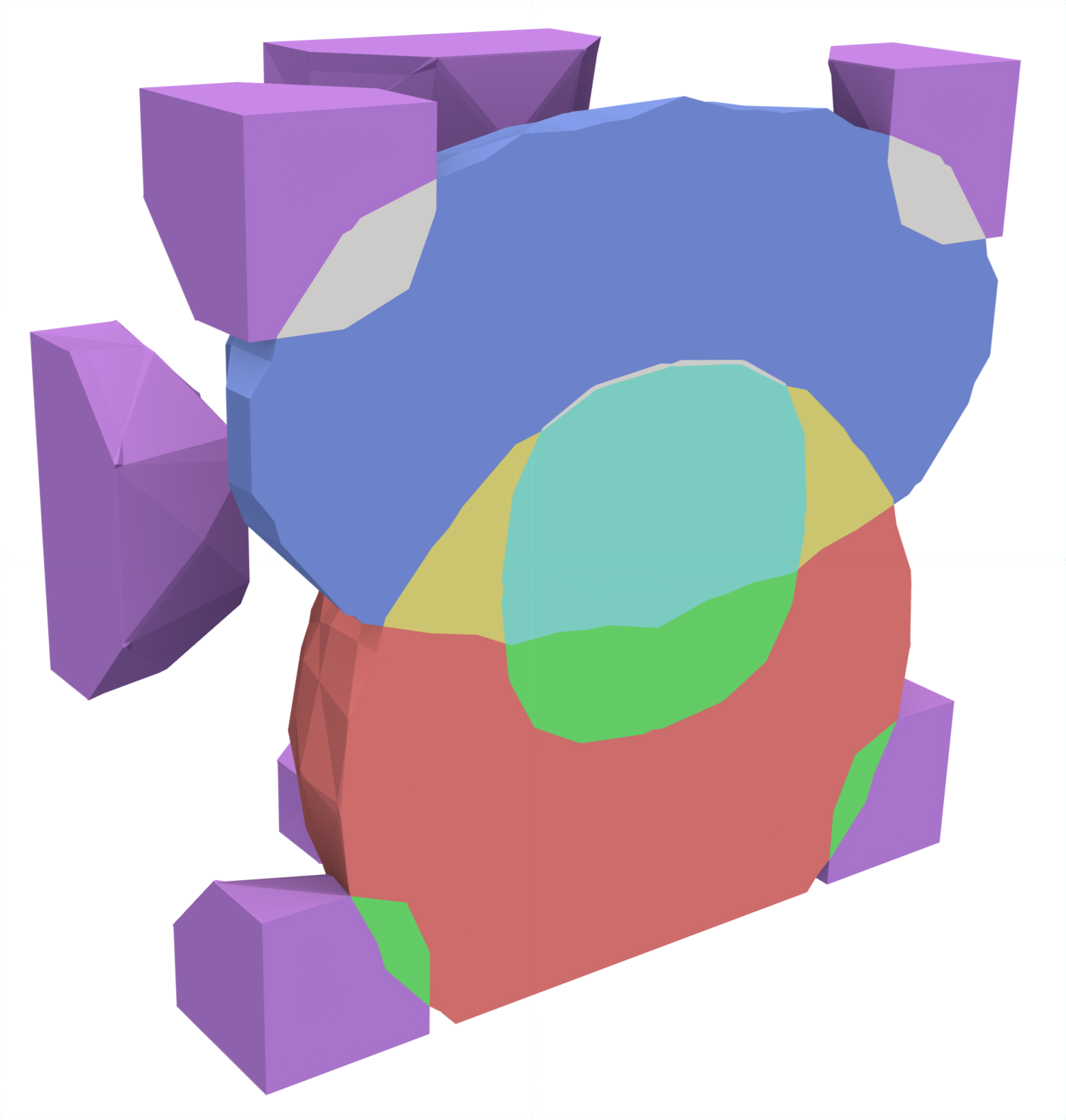}
        \caption{}
        \label{fig:fig6b}
    \end{subfigure}
    \begin{subfigure}{0.4\textwidth}
        \centering
        \includegraphics[width=0.8\linewidth]{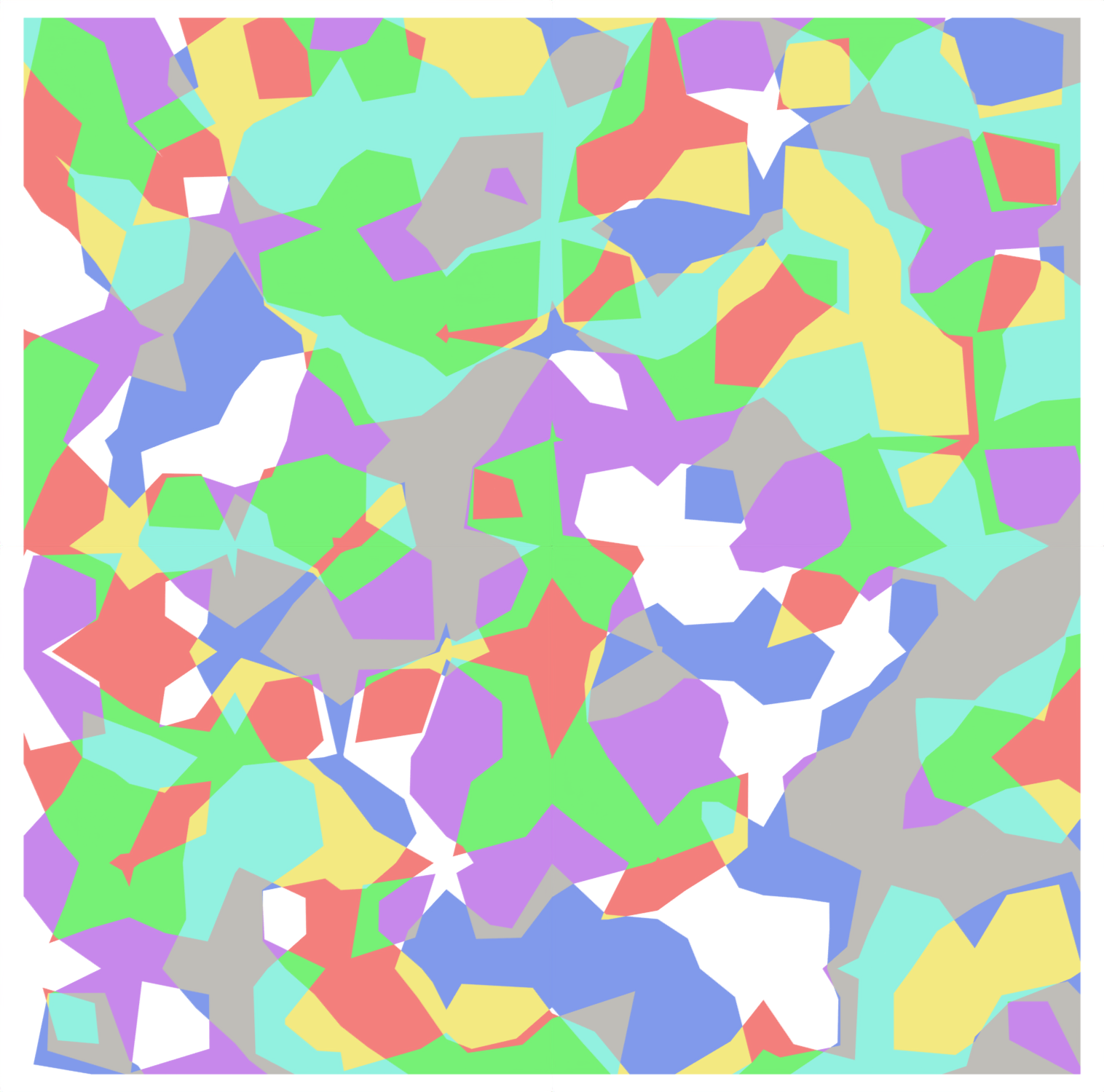}
        \caption{}
        \label{fig:fig6c}
    \end{subfigure}
        \begin{subfigure}{0.4\textwidth}
        \centering
        \includegraphics[width=0.8\linewidth]{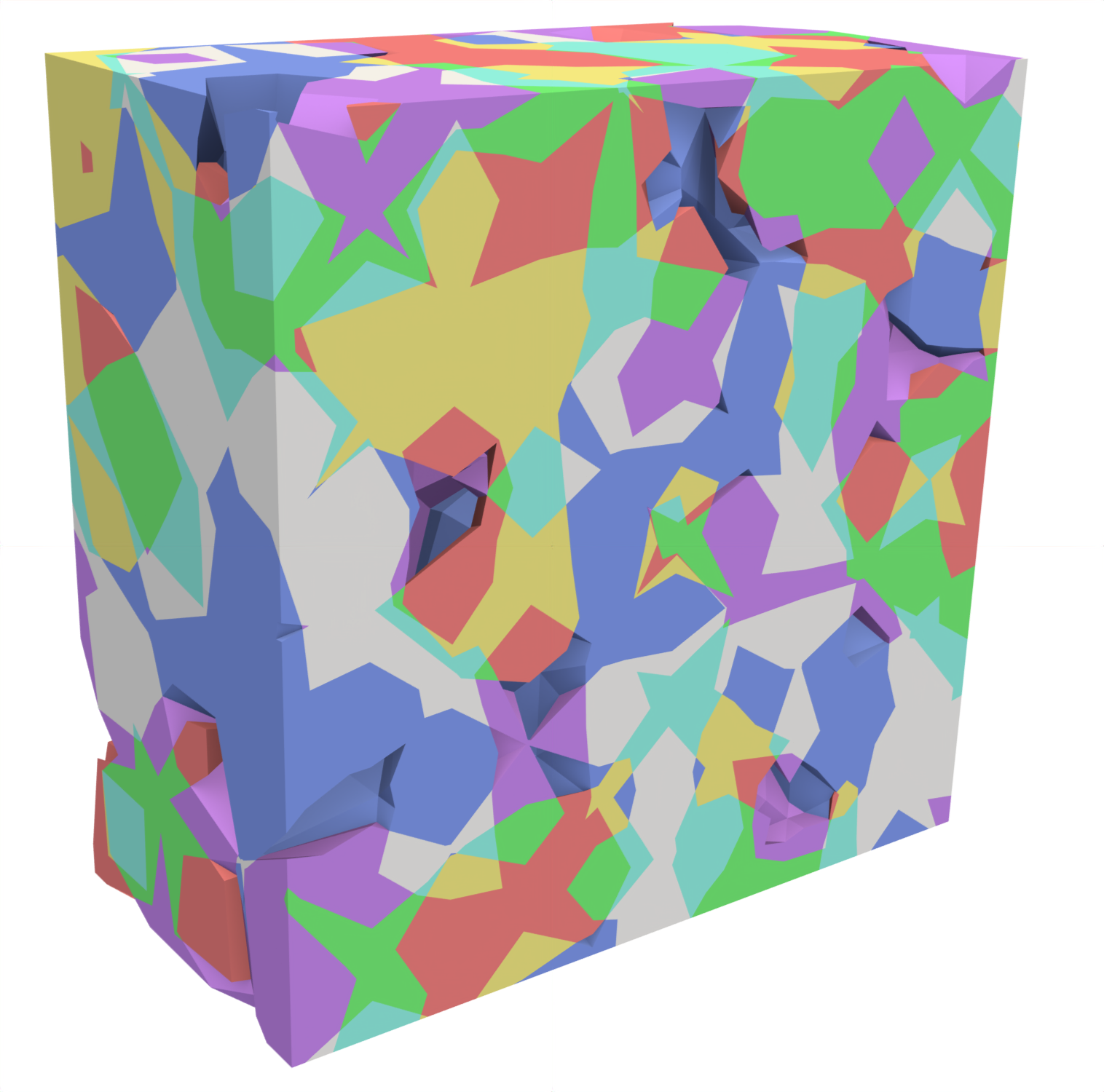}
        \caption{}
        \label{fig:fig6d}
    \end{subfigure}
    \caption{Geometries generated by three level-set functions for comparing the forward-mode automatic differentiation implementation with the exact derivatives provided in Section~\ref{sec: ls sd} and finite differences. (a) and (c) are in two dimensions. (b) and (d) are in three dimensions showing a visualisation of half of the background domain and with one phase removed to reveal internal structure. (c) and (d) are generated using random seeded values for the level-set functions.}
    \label{fig:fig6}
%
    \centering
    \captionof{table}{Verification of the forward-mode automatic differentiation implementation ($\mathrm{d}F_{\rm AD}$) against finite differences ($\mathrm{d}F_{\rm FDM}$) and the exact derivatives provided in Section~\ref{sec: ls sd} ($\mathrm{d}F$) for the functional $\phi\mapsto\int_{\Omega}\sin(x)\sin(y)~\mathrm{d}\Omega$ with respect to all three level-set functions.}
    \label{tab:verify}
    \footnotesize
    \begin{tabular}{c|c|c|c|c|c|c}
    & \multicolumn{4}{c|}{${\lVert\mathrm{d}F_{\mathrm{AD}}-\mathrm{d}F_{\mathrm{FDM}}\rVert_{\infty}}/{\lVert\mathrm{d}F_{\mathrm{AD}}\rVert_{\infty}}$} & \multicolumn{2}{c}{${\lVert\mathrm{d}F_{\mathrm{AD}}-\mathrm{d}F\rVert_{\infty}}/{\lVert\mathrm{d}F\rVert_{\infty}}$} \\ \hline
    \diagbox{$\Omega$}{Case} & Fig.~\ref{fig:fig6a} & Fig.~\ref{fig:fig6b} & Fig.~\ref{fig:fig6c} & Fig.~\ref{fig:fig6d} & Fig.~\ref{fig:fig6a} & Fig.~\ref{fig:fig6b} \\ \hline
    $D_{\phi_1}\cap D_{\phi_2} \cap D_{\phi_3}$ & 4.46e-8 & 1.34e-8 & 1.86e-9 & 2.82e-9 & 1.02e-14 & 5.24e-16 \\\hline
    $\overline{D}_{\phi_1}^\complement\cap \overline{D}_{\phi_2}^\complement \cap \overline{D}_{\phi_3}^\complement$ & 3.53e-8 & 1.51e-8 & 7.69e-9 & 2.69e-9 & 9.12e-15 & 2.50e-15  \\\hline
    $\partial D_{\phi_1}\cap D_{\phi_2} \cap D_{\phi_3}$ & 7.28e-8 & 4.31e-8 & 9.27e-10 & 1.08e-8 & 5.17e-15 & 1.93e-14  \\\hline
    $D_{\phi_1}\cap\partial D_{\phi_2} \cap \overline{D}_{\phi_3}^\complement$ & 5.94e-8 & 4.63e-8 & 8.88e-10 & 2.00e-9 & 6.74e-15 & 8.49e-15  \\\hline
    $D_{\phi_1}\cap \overline{D}_{\phi_2}^\complement \cap\partial D_{\phi_3}$ & 8.89e-9 & 2.91e-8 & 1.78e-9 & 3.00e-9 & 3.62e-15 & 5.26e-15  
    \end{tabular}
\end{figure}

\begin{figure}[t]
    \centering
    \begin{subfigure}{0.7\textwidth}
        \centering
        \includegraphics[width=\linewidth]{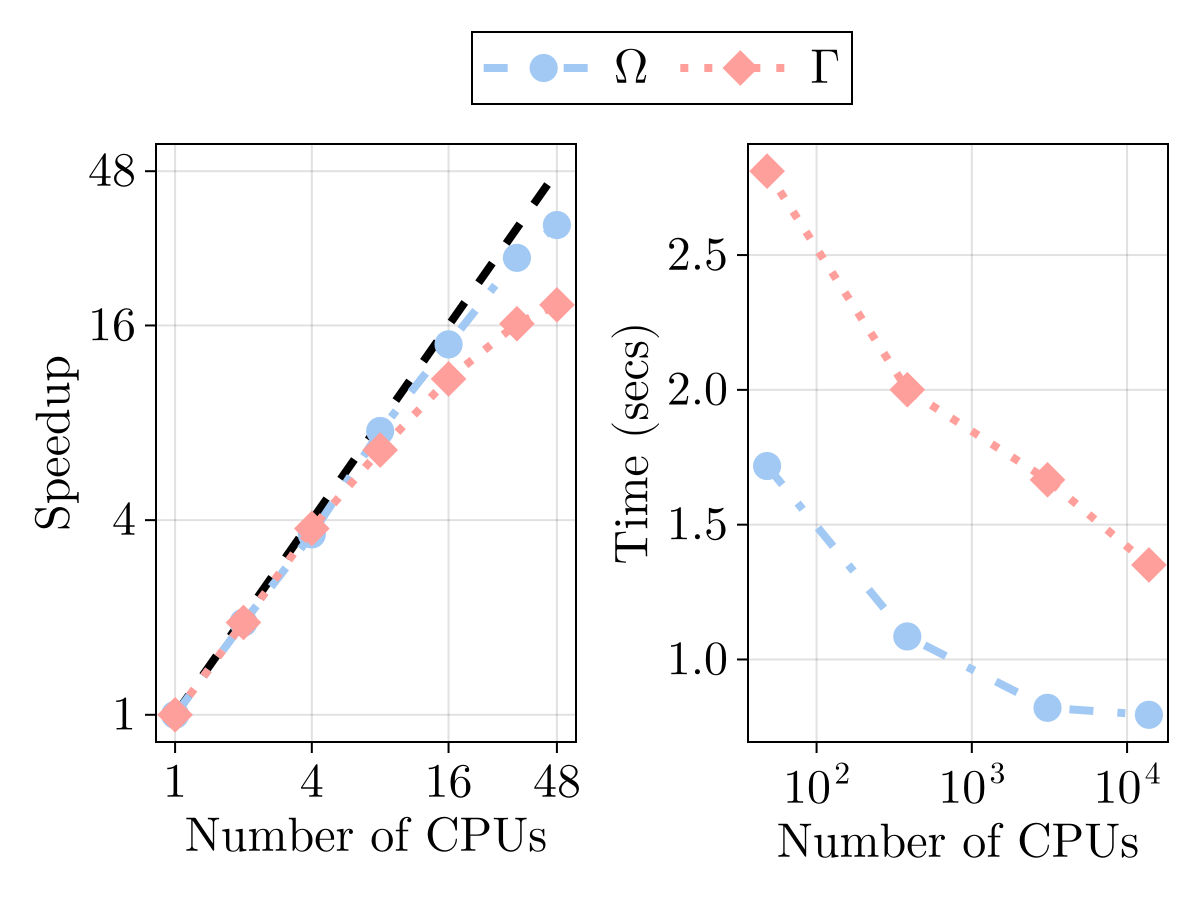}
        \caption{}
        \label{fig:bench}
    \end{subfigure}
    \hfill
    \begin{subfigure}{0.25\textwidth}
    \begin{subfigure}{\textwidth}
        \centering
        \includegraphics[width=\linewidth]{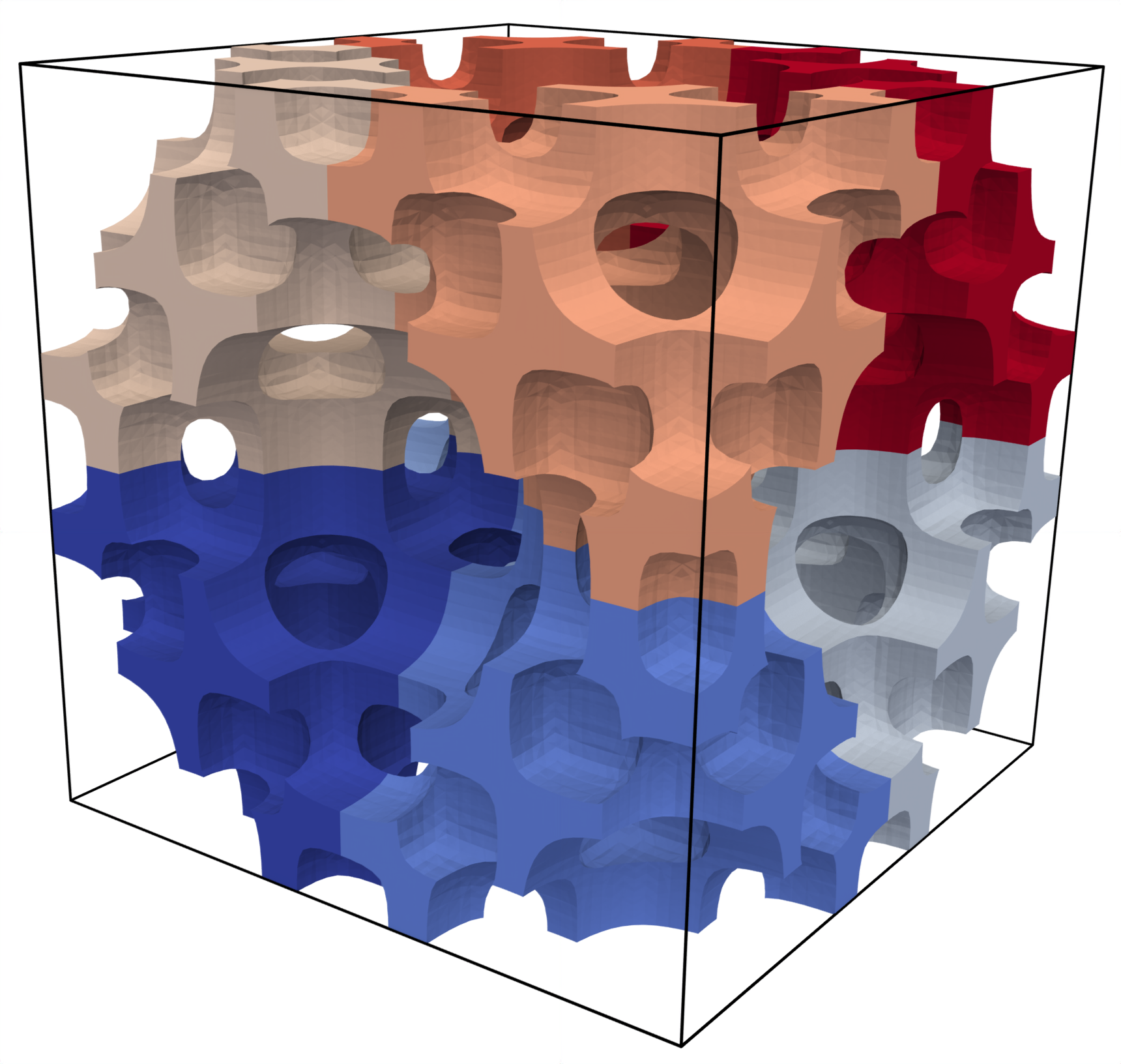}
        \caption{$\Omega=D_{\phi_1}\cap D_{\phi_2}\cap D_{\phi_3}$}
        \label{fig:benchgeo_a}
    \end{subfigure}
    \begin{subfigure}{\textwidth}
        \centering
        \includegraphics[width=\linewidth]{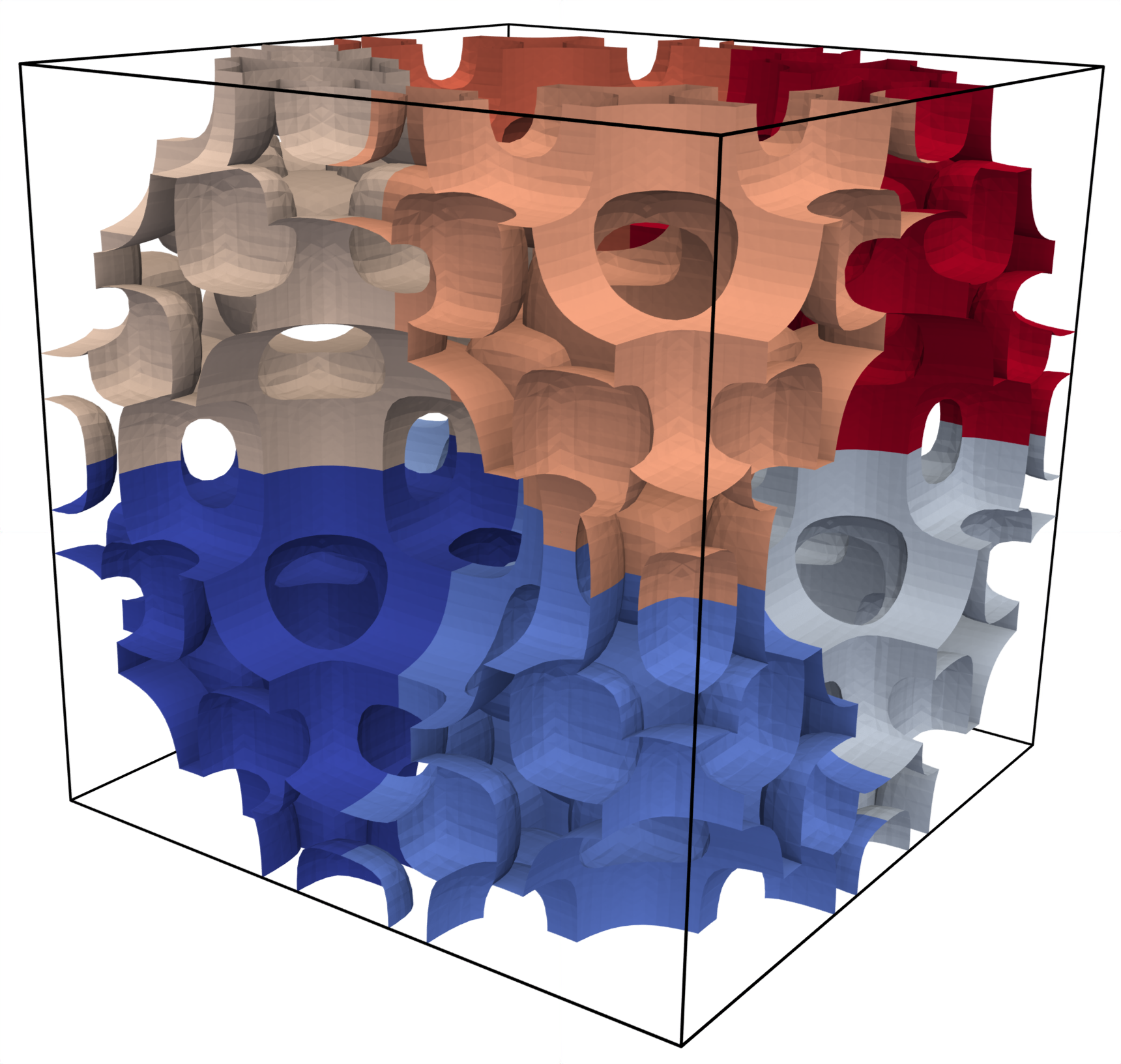}
        \caption{$\Gamma=\partial(D_{\phi_1}\cap D_{\phi_2}\cap D_{\phi_3})$}
        \label{fig:benchgeo_b}
    \end{subfigure}
    \end{subfigure}
    \caption{(a) Scalability benchmarks for computing all derivatives of $(\phi_1,\phi_2,\phi_3)\mapsto\int_{M}\cos(x)\sin(y)~\mathrm{d}M$ where $M=\Omega$ or $\Gamma$. The strong scaling benchmark is given in the left plot and weak scaling in the right plot. (b) and (c) Visualisations of the domains $\Omega$ and $\Gamma$, respectively, that are generated by $\phi_i(x,y,z)=\cos(a_i(x-b_i))\cos(a_iy)\cos(a_iz)-c$ for $i=1,2,3$ with $a_i=(2,4,6)$, $b_i=(0,1/4,1/6)$, and $c=\sqrt{2}/10$. The colours show the partitioning over eight CPUs for the strong scaling benchmark.}
    \label{fig:benchgeo}
\end{figure}

Next, we consider strong and weak scaling benchmarks of the automatic differentiation functionality for the geometries shown in Figure~\ref{fig:benchgeo}. We utilise the Gadi supercomputing cluster at the NCI (Canberra, Australia). Each node used has 192 GB of RAM and an Intel Xenon Platinum 8274 (Cascade Lake) processor. We use OpenMPI 5.0.8 and Julia 1.12.6 for all computation. For strong scaling benchmarks, we use a $\sim $5M cell tetrahedral background mesh constructed from a Cartesian mesh with $62^3$ cells that undergo one round of barycentric refinement. We partition this mesh over several CPUs, $P=\{1,2,4,8,16,32,48\}$. For weak scaling benchmarks, we set the number of cells per CPU to be 120,000 and scale up the number of CPUs, $P=\{48,384,3072,13824\}$.

Figure~\ref{fig:bench} shows the strong- and weak-scaling results. For strong scaling, the performance is suboptimal, which is expected due to load imbalance across the processors. In contrast, the weak-scaling experiments demonstrate nearly ideal scalability up to 1,658,880,000 cells across 13824 CPUs. In fact, the runtime even decreases as the processor count grows, which is expected because the number of cut cells is still decreasing for finer meshes. Finally, across all our numerical experiments, we observe that computing all derivatives in a single pass of Algorithm~\ref{alg:auto-diff} is more efficient than separately for each $\phi_i$.

Finally, it should be noted that the use of automatic differentiation will invariably require more floating-point operations than an assembly of the directional derivatives given in Section~\ref{sec: ls sd}. As an illustration, we consider taking $J(\phi_1,\phi_2,\phi_3)=\int_\Omega\cos(x)\cos(y)\cos(z)~\mathrm{d}x$ in Listing~\ref{code:implementation example} and benchmarking Line 28--29 against the assembly of the exact derivative. For this case, we observe that assembling the exact derivative takes approximately 0.44 seconds, whereas automatic differentiation takes approximately 1.46 seconds. As per the scalability results above, the forward-mode automatic differentiation can be substantially accelerated using a distributed CPU framework. In practice, the decision to rely on automatic differentiation should be balanced against the intellectual effort needed to derive and implement the exact expressions. In situations where the explicit formulas are cumbersome (for instance $d_{\phi_2}J(\phi_1,\phi_2)(w)$ in the Supplementary Material), or where the formulation of the optimisation problem is subject to change, automatic differentiation provides a practical route to reduce the time-to-first-result, to avoid errors in sensitivity derivations and implementation, and to make the code easier to maintain and update.

\section{Level-set topology optimisation}\label{sec: Level-set topology optimisation}\noindent
In the following, we briefly discuss our implementation of multi-phase level-set topology optimisation with unfitted finite elements. We follow the approach of \citet{WEGERT2025118203} for each level-set function.

\subsection{Boundary evolution}\noindent
To update the level-set functions we use the approach proposed by \citet{10.1002/nme.4823_2015}. Namely, we evolve the $i^{\rm th}$ level-set function $\phi_i$ by solving the transport equation
\begin{equation}\label{eqn: transport}
    \begin{cases}
        \displaystyle\pderiv{\phi_i(t,\bx)}{t}+\boldsymbol{\beta}_i\cdot\boldsymbol{\bnabla}\phi_i(t,\bx)=0,\\
        \phi_i(0,\bx)=\phi_{i,0}(\bx),\\
        \bx\in D,~t\in(0,T),
    \end{cases}
\end{equation}
where $\boldsymbol{\beta}_i=g_i\boldsymbol{n}_i$ is the velocity field and $\phi_{i,0}$ is the initial value of the $i^{\rm th}$ level-set function.
In this work, $g_i$ is the gradient resulting from the Hilbertian extension approach \cite{10.1016/bs.hna.2020.10.004_978-0-444-64305-6_2021}, and the normal $\boldsymbol{n}_i$ is computed as $\boldsymbol{n}_i={\boldsymbol{\bnabla}\phi_i/}{\lVert\boldsymbol{\bnabla}\phi_i\rVert}$. The Hilbertian extension approach identifies the directional derivative as a gradient in a Hilbert space $H$, defined on the background domain $D$, by solving the variational problem
\begin{equation}\label{eqn: hilb extension wf}
\begin{aligned}
&\text{Find $g_i\in H$ such that}
\\
&\langle g_i,w\rangle_H=\mathrm{d}J(\phi_i;w),~\forall w\in H,
\end{aligned}
\end{equation}
where $\langle\cdot,\cdot\rangle_H$ is the inner product on $H$.
Here we use a weighted $H^1(D)$ bilinear form and solve using standard Lagrangian finite elements \citep{WEGERT2025118203}.
This choice provided a smooth gradient whose support is extended throughout the background domain $D$.

To solve \eqref{eqn: transport} and \eqref{eqn: hilb extension wf}, we utilise the approach of \citet{WEGERT2025118203} for each level-set function. In particular, the transport equation in \eqref{eqn: transport} is solved using an interior penalty approach and Crank--Nicolson for the discretisation in time \citep{10.1016/j.cma.2017.09.005_2018,Burman_Elfverson_Hansbo_Larson_Larsson_2017,Burman_Fernández_2009}.
\subsection{Reinitialisation}\noindent
A smooth evolution of the boundary described above requires that the level-set function $\phi_i$ has well-behaved gradients. This is usually achieved by regularly reinitialising $\phi_i$ to be an approximation of the signed distance function \citep{978-0-387-22746-7_2006} for the corresponding domain $D_{\phi_i}$. That is, $\phi_i$ is reinitialised as
\begin{equation}
\phi_i(\bx)=\begin{cases}
-d(\bx, \partial D_{\phi_i})&\text{if }\bx \in D_{\phi_i}, \\
0&\text{if } \bx \in \partial D_{\phi_i}, \\
d(\bx, \partial D_{\phi_i})&\text{if } \bx \in D \backslash \overline{D}_{\phi_i},
\end{cases}
\end{equation}
where $d(\bx, \partial D_{\phi_i}):=\min _{\boldsymbol{p} \in \partial D_{\phi_i}}\lvert\bx-\boldsymbol{p}\rvert$ is the minimum distance from $\bx$ to the boundary $\partial D_{\phi_i}$. To find an approximate signed distance function for each level-set function, we utilise the approach of \citet{mallon2024neurallevelsettopology,WEGERT2025118203}. In particular, we solve the reinitialisation problem using Picard iterations on the eikonal equation,  stabilising the scheme via an artificial viscosity term, and prevent boundary movement via a surface penalty term. We refer to \citet{mallon2024neurallevelsettopology} for further discussion.

\subsection{Isolated volumes}\label{sec: isolated vol tagging}\noindent
Non-connected domains with isolated volumes may easily appear during the optimisation iterations.
In the context of elliptic problems like those in Sections~\ref{sec: multi-phase unfitted fes} and~\ref{sec: anisotropic diffusion}, the state equation restricted to an isolated volume will become singular due to the natural boundary condition on the volume’s boundary.
In this case, there will be a local one-dimensional null space of constant functions in the isolated volume.
To deal with this issue in the context of multi-phase unfitted finite elements, we follow the approach of \citet{WEGERT2025118203}. In particular, we construct an indicator function $\chi$ using isolated volume tagging \citep{WEGERT2025118203} and add a term to the bilinear form that constrains the average of the solution to be zero in the regions marked by $\chi$.
This addition removes the null space from isolated regions appearing in intermediate designs.

\section{Examples}\label{sec: Examples}\noindent
In this section, we consider example multi-phase topology optimisation problems using unfitted finite elements to demonstrate the capabilities provided by our approach. All scripts are provided open-source as described in the \textit{Data availability} section.

\subsection{Anisotropic diffusion}\label{sec: anisotropic diffusion}\noindent
We first consider minimising the total thermal diffusion in two-dimensional multi-phase anisotropic media.
This problem is also known as the \textit{thermal compliance problem}, due to the resemblance with the analogue problem in linear elasticity.
The formulation is as discussed in Sections~\ref{sec: multi-phase unfitted fes} and \ref{subsec: example shape calc}, with the addition of volume constraints.

The optimisation problem is given by
\begin{equation}\label{ansi optim}
    \begin{aligned}
        \min_{(\phi_1,\phi_2)\in\mathcal{V}_{h}^2}&~(\Ab_1\bnabla u_1,\bnabla u_1)_{D_{\phi_1}\cap \overline{D}_{\phi_2}^\complement}+(\Ab_2\bnabla u_2,\bnabla u_2)_{D_{\phi_1}\cap D_{\phi_2}}\\
        {\rm s.t.}&~\mathrm{Vol}(D_{\phi_1}\cap \overline{D}_{\phi_2}^\complement)=0.05\mathrm{Vol}(D),\\
        &~\mathrm{Vol}(D_{\phi_1}\cap D_{\phi_2})=0.05\mathrm{Vol}(D),
    \end{aligned}
\end{equation}
where $u_1$ and $u_2$ satisfy $A([u_1,u_2],[v_1,v_2];\phi_1,\phi_2)=l([v_1,v_2];\phi_1,\phi_2)$ with $A$ and $l$ given by 
\begin{equation}\label{eqn: A aniso example}
    \begin{aligned}
        A([u_1,u_2]&,[v_1,v_2];\phi_1,\phi_2)= (\Ab_1\bnabla v_1,\bnabla v_1)_{D_{\phi_1}\cap \overline{D}_{\phi_2}^\complement}+(\Ab_2\bnabla u_2,\bnabla v_2)_{D_{\phi_1}\cap D_{\phi_2}}
        \\&\quad+\sum_{F\in\mathcal{F}_{G_1}}(\eta_{1}(h)\jump{\partial_{\boldsymbol{n}_F}u_1},\jump{\partial_{\boldsymbol{n}_F}v_1})_F+\sum_{F\in\mathcal{F}_{G_2}}(\eta_{2}(h)\jump{\partial_{\boldsymbol{n}_F}u_2},\jump{\partial_{\boldsymbol{n}_F}v_2})_F,
        \\&\quad-(\mean{\Ab\bnabla_i u_i},\jump{ v_i})_{D_{\phi_1}\cap\partial D_{\phi_2}}-(\mean{\Ab\bnabla_i v_i},\jump{ u_i})_{D_{\phi_1}\cap\partial D_{\phi_2}}+(\mu(h)\jump{u_i},\jump{v_i})_{D_{\phi_1}\cap\partial D_{\phi_2}}
        \\&\quad+(\chi u_1,v_1)_{D_{\phi_1}\cap \overline{D}_{\phi_2}^\complement}+(\chi u_2,v_2)_{D_{\phi_1}\cap D_{\phi_2}},
    \end{aligned}
\end{equation}
and
\begin{equation}
    l(V;\phi_1,\phi_2) = \sum_i(g,v_i)_{\Gamma_N}.
\end{equation}
The last two terms of \eqref{eqn: A aniso example} constrains $u_1$ and $u_2$ to be zero within isolated volumes marked by $\chi$ (see Section \ref{sec: isolated vol tagging}). In this example, we take $D$ to be the unit square and choose $\Gamma_N$ and $\Gamma_D$ to be given by the dark blue and black regions in Figure~\ref{fig:fig7a}, respectively. Furthermore, we take the flux $g$ on $\Gamma_N$ to be unitary and the anisotropic diffusion tensors to be
\begin{equation}
    \Ab_1=\begin{pmatrix}
        5 & 0\\
        0 & 1
    \end{pmatrix}, \quad \Ab_2=\begin{pmatrix}
        1 & 0\\
        0 & 5
    \end{pmatrix}.
\end{equation}
For the weighting parameters $\kappa_1$ and $\kappa_2$, we use expressions~\eqref{eqn: weighting aniso}. Finally,  the ghost penalty parameters $\eta_i(h)$ and the Nitsche penalty parameter $\mu(h)$ are set to
\begin{equation}
    \eta_i(h) = 0.1\alpha_i h,\quad \mu(h)=\frac{10^2\max(\alpha_1,\alpha_2)}{h},
\end{equation}
where $\alpha_i=\lVert \Ab_i\rVert_\infty$. Finally, for the numerical experiments, we fix the geometry near the Dirichlet region to be a small ball. This is done by creating a level-set function for the region with domain $D_{\rm Diri}$ and modifying the descriptions of the phases to be $(D_{\phi_1}\cap \overline{D}_{\phi_2}^\complement)\setminus \overline{D}_{\rm Diri}$ and $(D_{\phi_1}\cap D_{\phi_2})\cup D_{\rm Diri}$.

We discretise the above problem over a structured background mesh made of 40,000 triangles and solve the constrained optimisation problem in \eqref{ansi optim} using a projection method \citep{Wegert_2023b}. Automatic shape differentiation computes all sensitivities.

Figure~\ref{fig:fig7} visualises the material phases at some iterations of the optimisation algorithm and shows the iteration history of the objective and volume constraints. The final optimised design has a cross-like structure, with vertical struts comprised of the material phase that has a large vertical diffusion coefficient, and horizontal struts comprised of the material phase that has a large horizontal diffusion coefficient. The design therefore gives efficient transfer of heat from the sources (Neumann boundary conditions) to the sink (Dirichlet boundary condition). The increase of the objective function over the optimisation that we see in Figure~\ref{fig:fig7} is expected due to the high volume fraction of each phase in the initial design.

We can include obstacles using the same machinery that is used to fix the geometry around the Dirichlet region. In the example in Figure~\ref{fig:aniso_2d_obst}, we solve the anisotropic minimum thermal compliance problem with four square obstacles blocking the straight paths from the Neumann boundaries to the Dirichlet region, prescribing zero Neumann boundary conditions one the obstacle boundaries. For this example, the volume fraction constraint is $0.1$ rather than $0.05$ for each phase.

\begin{figure}[!t]
    \centering
    \begin{subfigure}{0.32\textwidth}
        \centering
        \includegraphics[width=\linewidth]{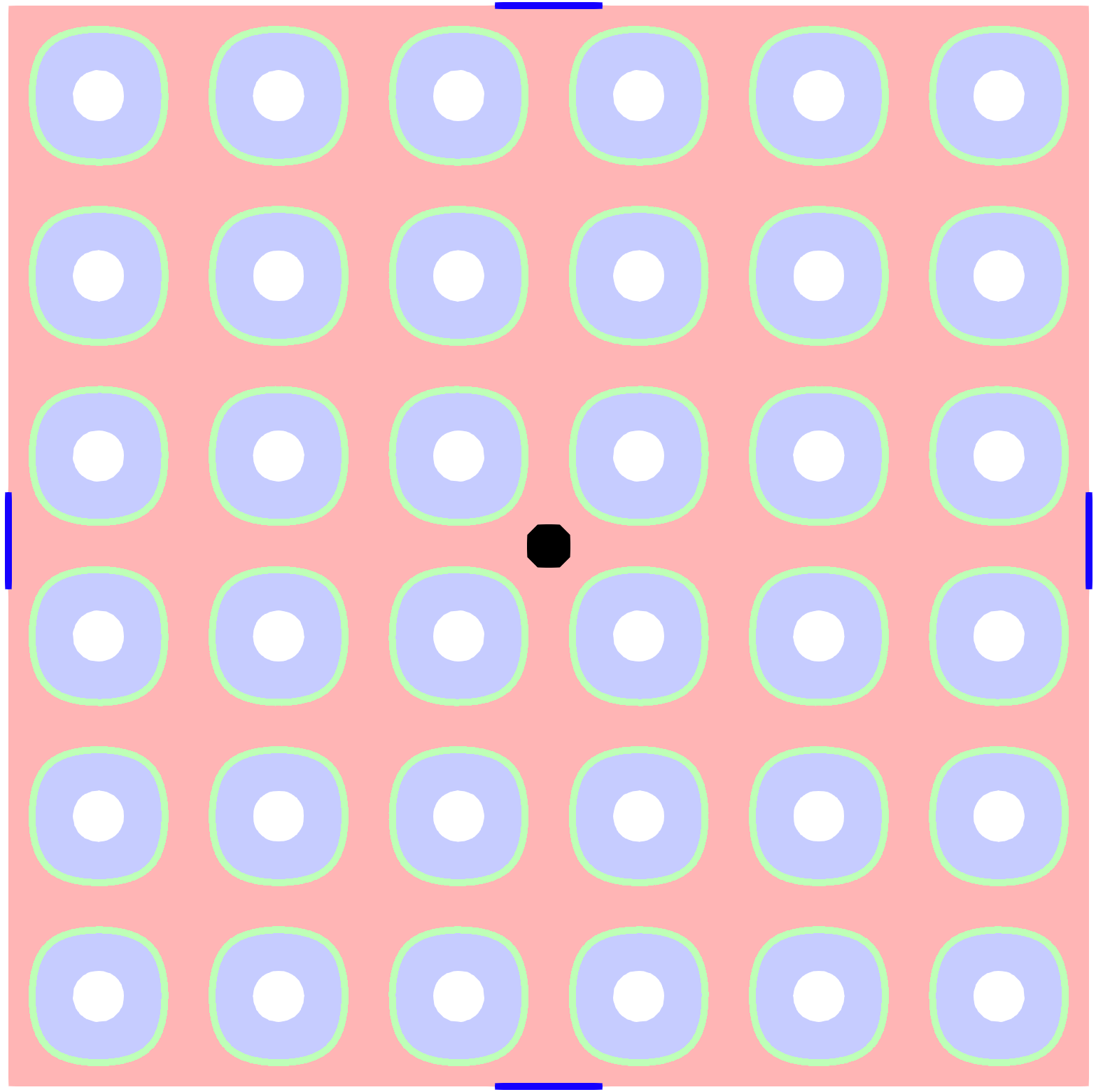}
        \caption{Initial iteration}
        \label{fig:fig7a}
    \end{subfigure}
    \begin{subfigure}{0.32\textwidth}
        \centering
        \includegraphics[width=\linewidth]{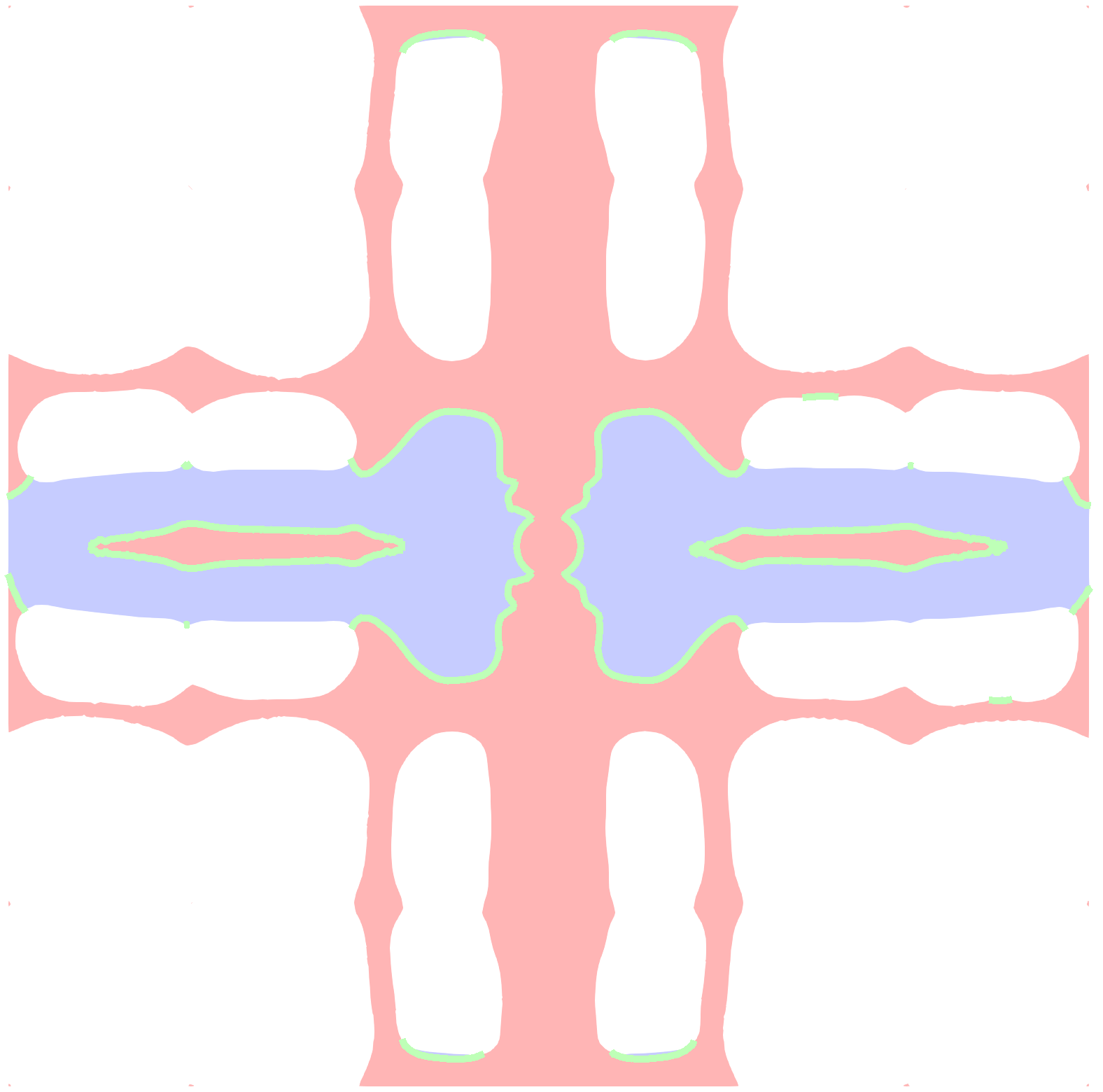}
        \caption{Intermediate iteration}
    \end{subfigure}
    \begin{subfigure}{0.32\textwidth}
        \centering
        \includegraphics[width=\linewidth]{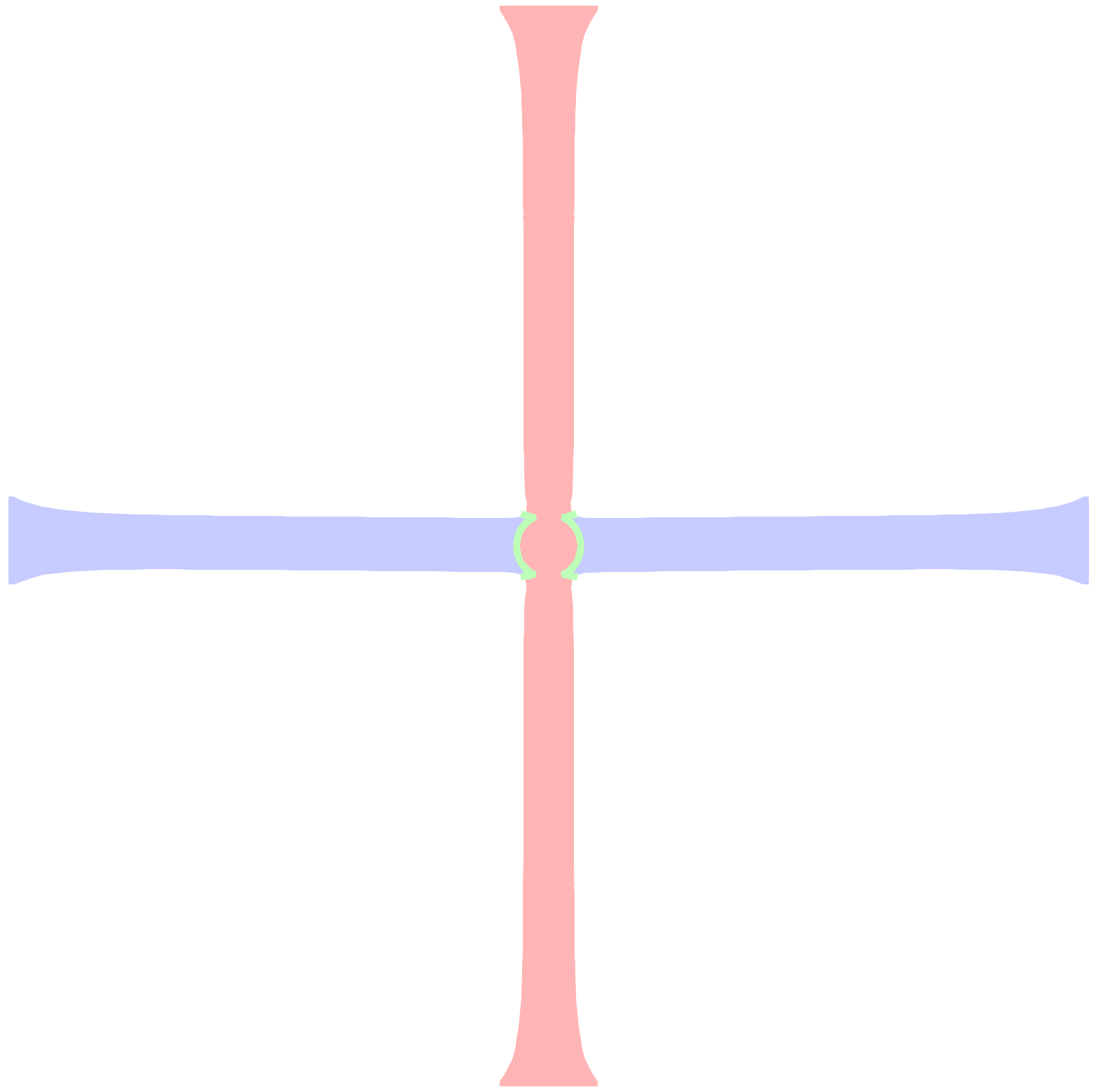}
        \caption{Final iteration}
    \end{subfigure}
    \begin{subfigure}{0.8\linewidth}
        \centering
        \includegraphics[width=\linewidth]{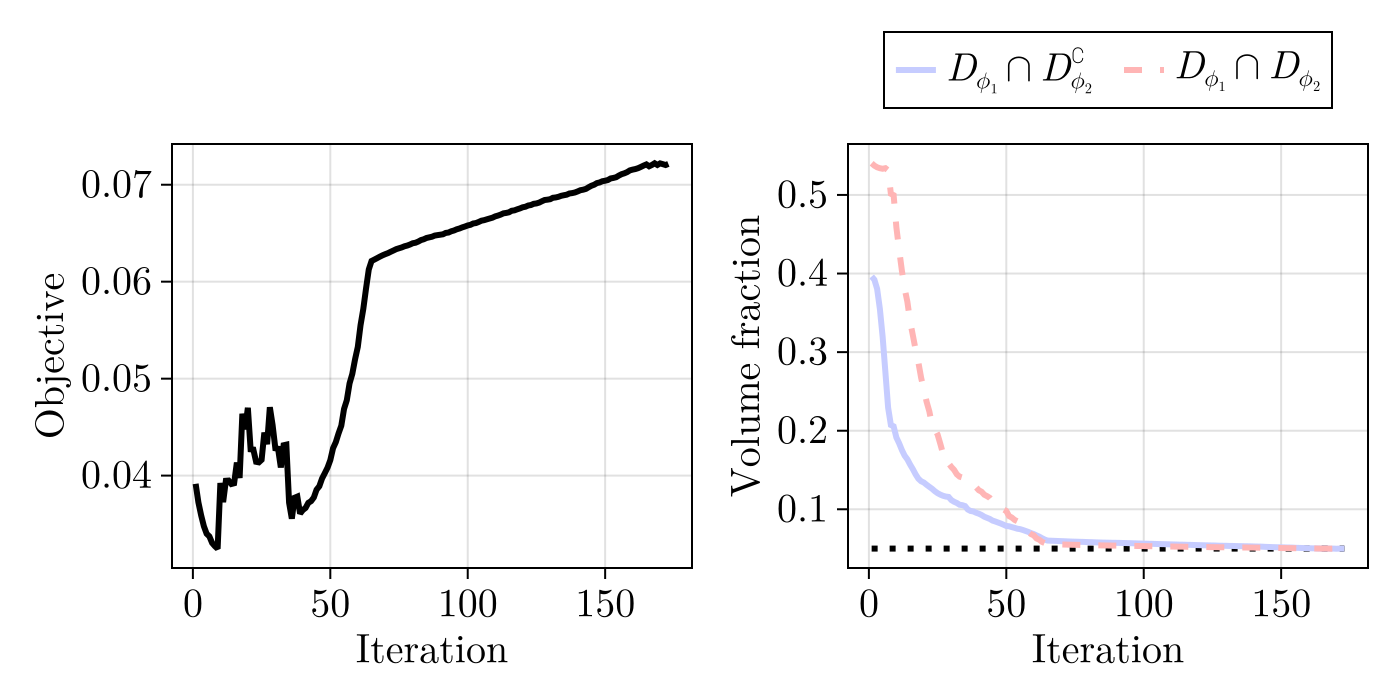}
        \caption{}
    \end{subfigure}
    \caption{(a)--(c) Visualisations of the material phases for the two-dimensional anisotropic minimum thermal compliance problem. (d) Iteration history for the objective and volume fractions of each phase compared to the required volume fraction of 0.05. The blue and red phases correspond to materials with large diffusion coefficients in the horizontal and vertical directions, respectively. In green we visualise the interface between the two material phases. In (a), we show the Neumann and Dirichlet regions in dark blue and black, respectively.}
    \label{fig:fig7}
\end{figure}

\begin{figure}[!t]
    \centering
    \begin{subfigure}{0.32\textwidth}
        \centering
        \includegraphics[width=\linewidth]{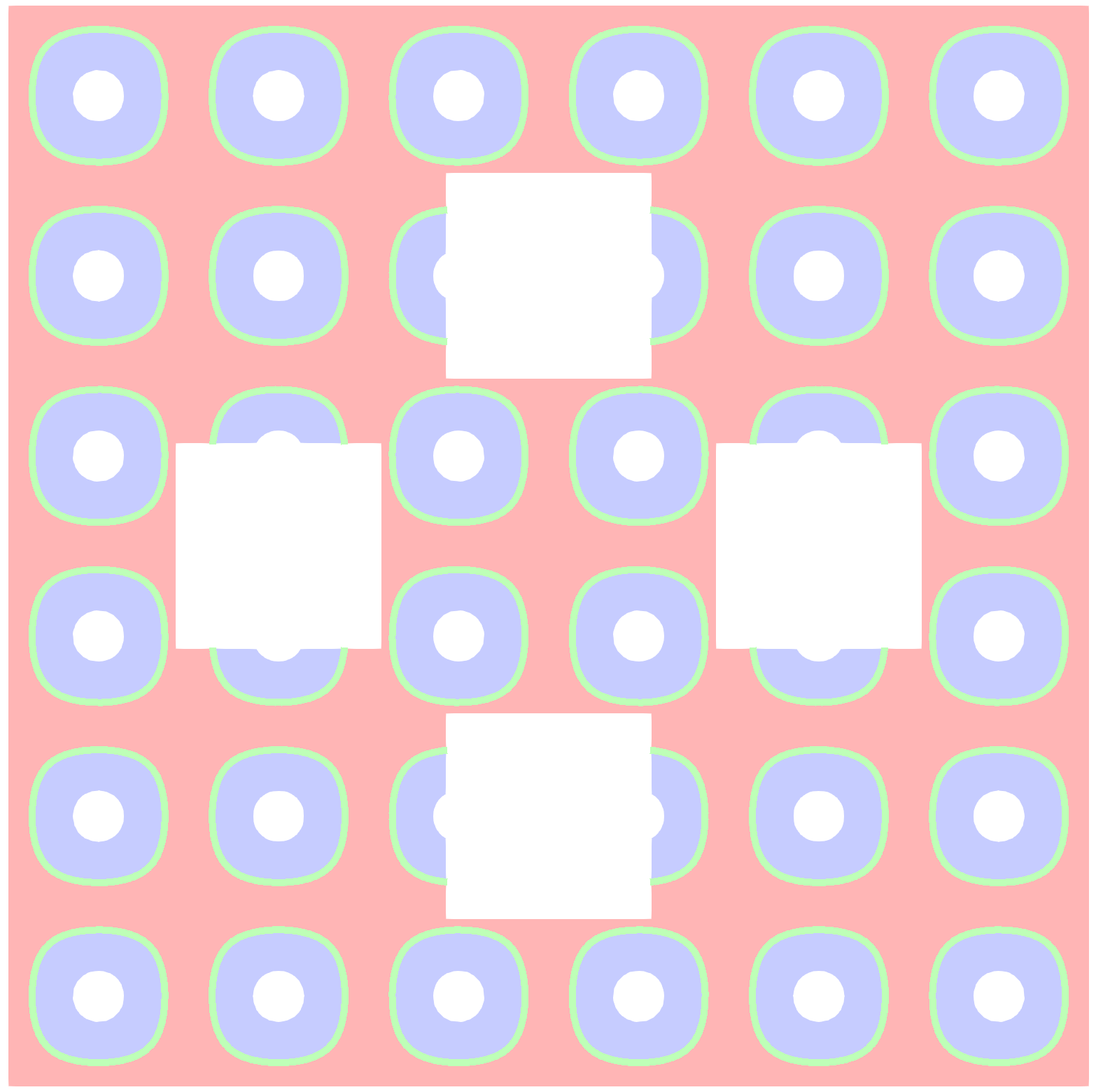}
        \caption{Initial iteration}
    \end{subfigure}
    \begin{subfigure}{0.32\textwidth}
        \centering
        \includegraphics[width=\linewidth]{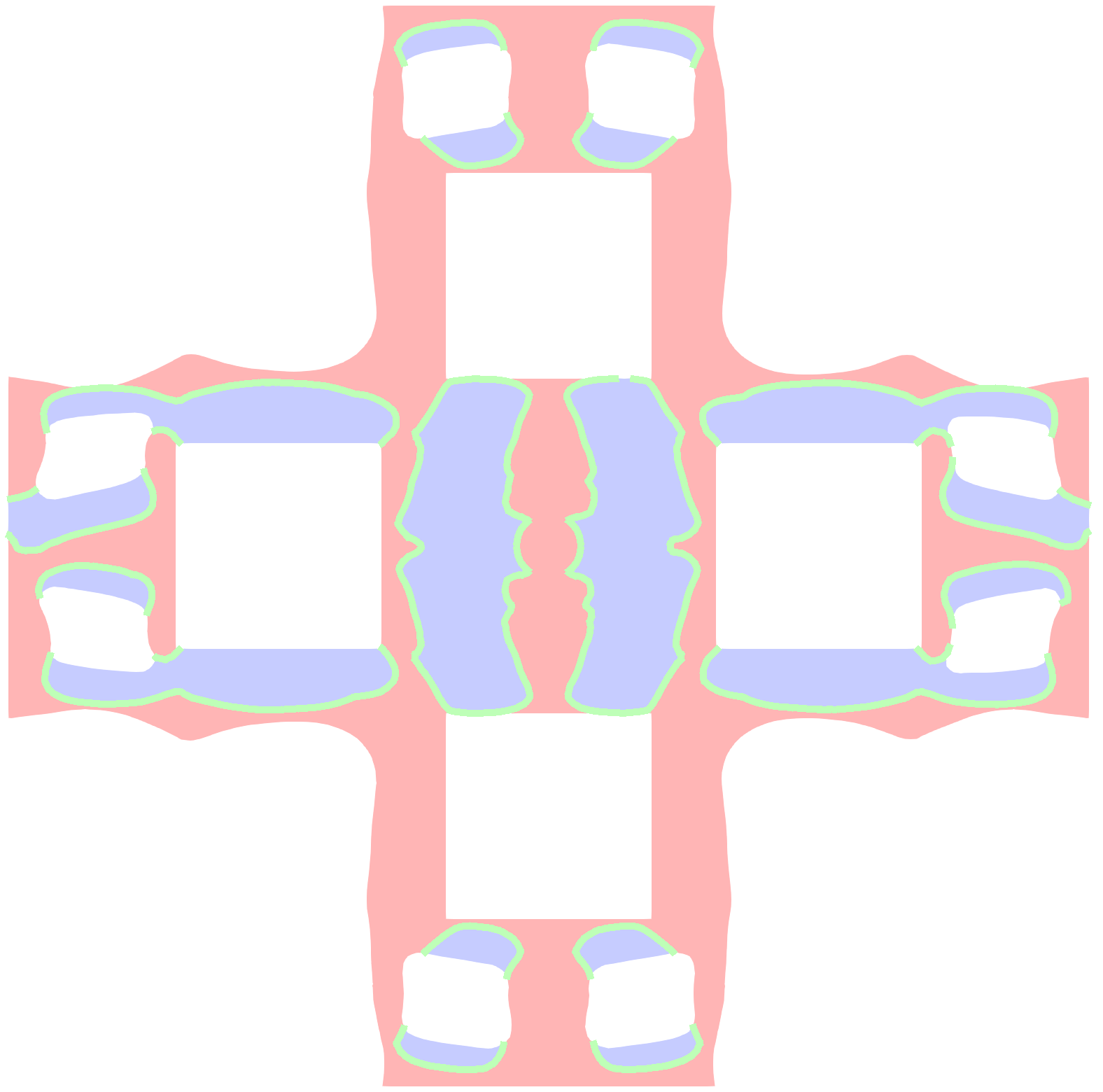}
        \caption{Intermediate iteration}
    \end{subfigure}
    \begin{subfigure}{0.32\textwidth}
        \centering
        \includegraphics[width=\linewidth]{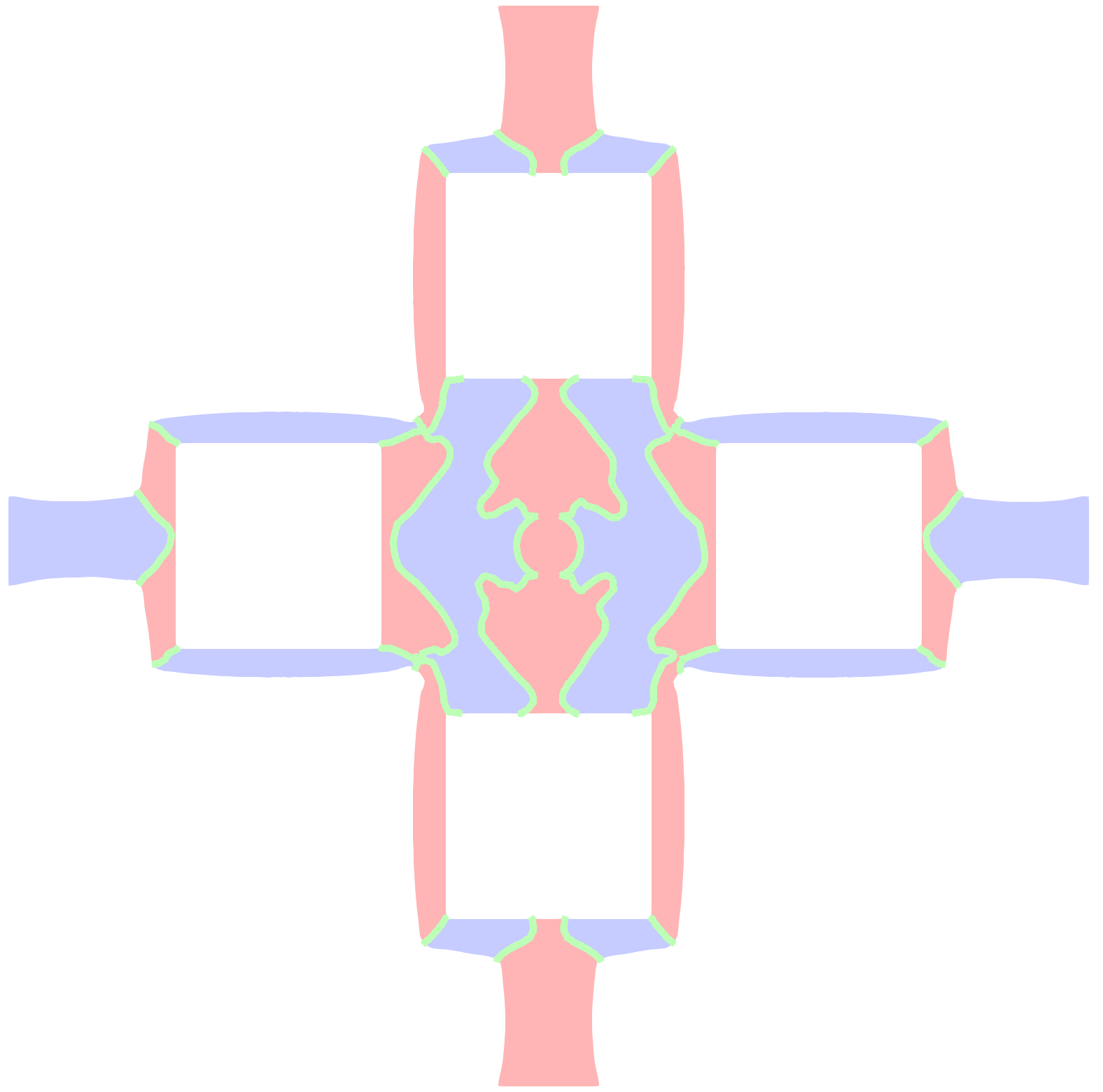}
        \caption{Final iteration}
    \end{subfigure}
    \caption{Visualisations of the material phases for the two-dimensional anisotropic minimum thermal compliance problem with the inclusion of four obstacles. The blue and red phases corresponds to materials with large diffusion coefficients in the horizontal and vertical directions, respectively. In green we visualise the interface between the two material phases.}
    \label{fig:aniso_2d_obst}
\end{figure}

We also minimise the thermal compliance of a three-dimensional multi-phase anisotropic media. This example uses add a third phase that has an anisotropic diffusion tensor with a large $z$ component. We constrain the three material phases to each have a volume fraction of $0.02$. The boundary conditions are analogous to the two dimensional case; that is, the Dirichlet region is now a small spherical inclusion at the centre of the domain and the Neumann boundaries are square regions in the middle of each face. We solve this three-dimensional problem on the Gadi supercomputing cluster using 16 CPUs.  The Hilbertian extension problem \eqref{eqn: hilb extension wf} is solved using an algebraic multigrid preconditioned conjugate gradient method as outlined in \cite{GridapTopOpt}, whereas the direct parallel solver MUMPS \cite{MUMPS:1,MUMPS:2} is employed for the state equation, the transport equation, and the reinitialisation equation. Finally, we use 1/8$^{\rm th}$ symmetry of the problem to further reduce the computational burden. We note that the investigation of iterative methods is important for solving these problems at large scales; however this is outside the scope of the current contribution.

In Figure~\ref{fig:aniso_3d} we show results for this three-dimensional anisotropic thermal diffusion topology optimisation problem. Similar to the two-dimensional case, the final design has a cross-like structure where each member has a large diffusion coefficient in the most appropriate direction for efficient transfer of heat from the sources to the sink. In Figure~\ref{fig:aniso_3da} we see that the green phase is separated from the void phase by a thin layer consisting of the other material phases. This occurs because we have used two level-set functions to describe the three material phases and the void phase. The interface connecting the green phase $D_{\phi_1}\cap D_{\phi_2}$ to the void phase $\overline{D}_{\phi_1}^\complement\cap \overline{D}_{\phi_2}^\complement$ can therefore only ever have measure zero (see Fig.~\ref{fig:fig2-new}). To allow an interface between these phases, another level-set function could be added \citep[e.g.,][]{10.1016/j.cma.2014.11.002_2015}); however, further exploration of this is outside the scope of this work.

\begin{figure}[!t]
    \centering
    \begin{subfigure}{0.32\textwidth}
        \centering
        \includegraphics[width=\linewidth]{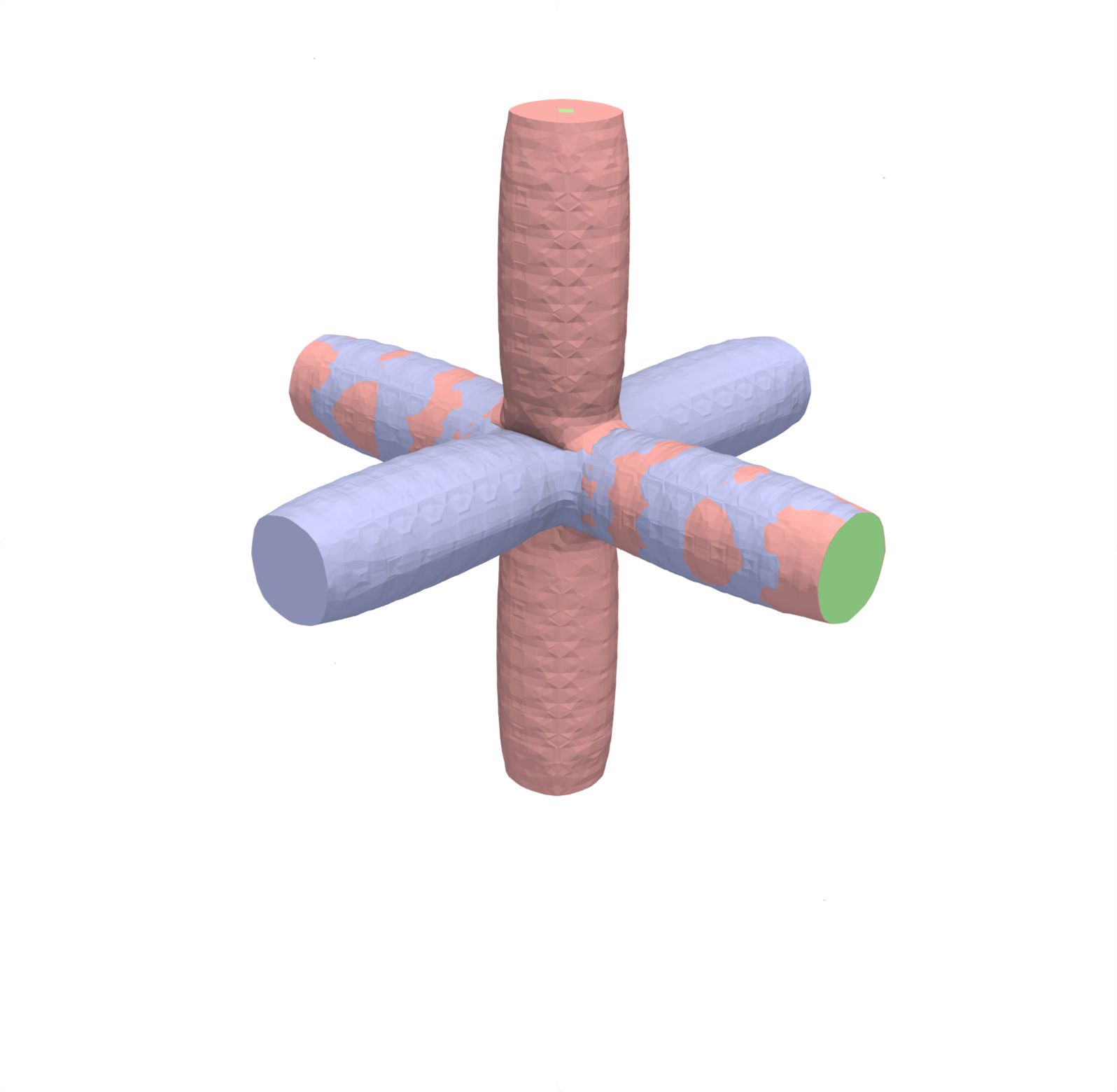}
        \caption{Final iteration}
        \label{fig:aniso_3da}
    \end{subfigure}
    \begin{subfigure}{0.32\textwidth}
        \centering
        \includegraphics[width=\linewidth]{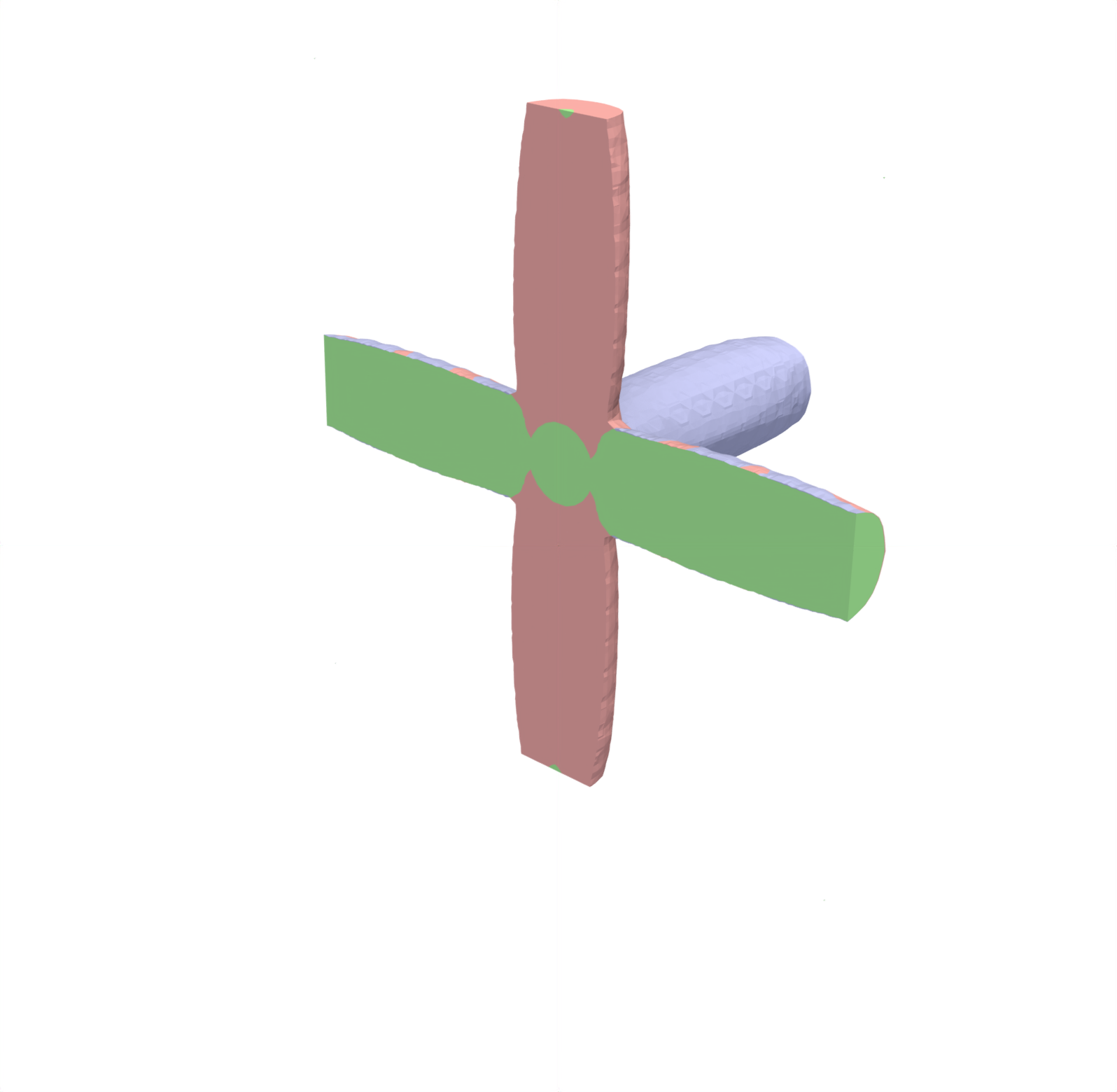}
        \caption{}
    \end{subfigure}
    \begin{subfigure}{0.32\textwidth}
        \centering
        \includegraphics[width=\linewidth]{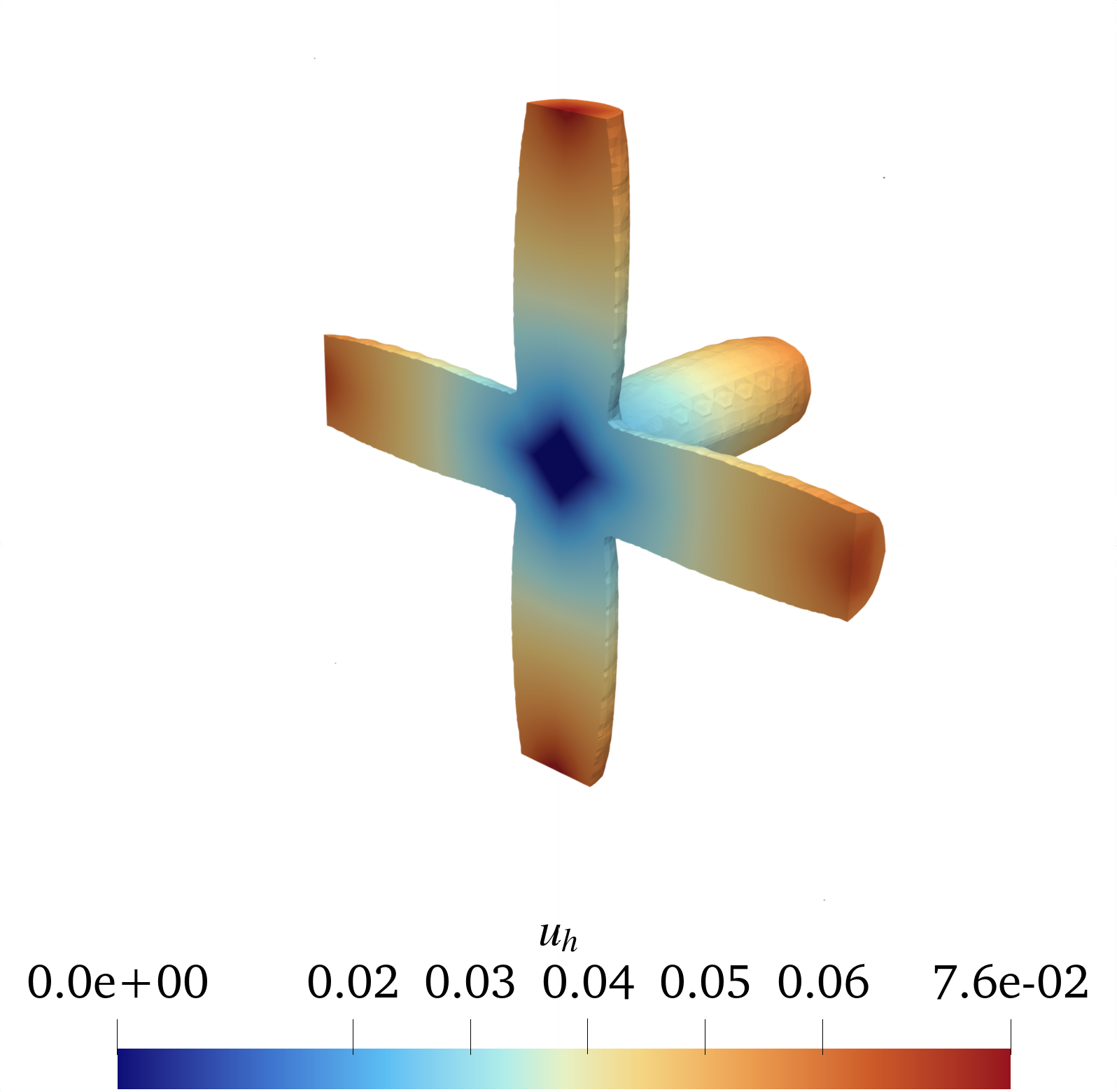}
        \caption{}
    \end{subfigure}
    \caption{(a) Visualisation of the optimised material phases for the three-dimensional anisotropic minimum thermal compliance problem with three phases. (b) The internal structure of the phases. (e) Visualisation of the temperature field across all phases.}
    \label{fig:aniso_3d}
\end{figure}

\subsection{Cantilever}\noindent
The next example that we consider is minimising the compliance of a multi-phase linear elastic cantilever subject to a volume constraint. This is a classical problem in the literature, and we extend the formulation described by \citet{WEGERT2025118203} and references therein.

The optimisation problem is given by
\begin{equation}\label{canti optim}
    \begin{aligned}
        \min_{(\phi_1,\phi_2)\in\mathcal{V}_{h}^2}&~(\boldsymbol{\varepsilon}(\boldsymbol{u}_1),\boldsymbol{\sigma}_1(\boldsymbol{u}_1))_{\overline{D}_{\phi_1}^\complement\cap D_{\phi_2}}+(\boldsymbol{\varepsilon}(\boldsymbol{u}_2),\boldsymbol{\sigma}_2(\boldsymbol{u}_2)_{D_{\phi_1}\cap D_{\phi_2}}\\
        {\rm s.t.}&~\mathrm{Vol}(\overline{D}_{\phi_1}^\complement\cap D_{\phi_2})=0.15\mathrm{Vol}(D),\\
        &~\mathrm{Vol}(D_{\phi_1}\cap D_{\phi_2})=0.15\mathrm{Vol}(D),
    \end{aligned}
\end{equation}
where the displacement fields $\boldsymbol u_1$ and $\boldsymbol  u_2$ satisfy the variational form of the two-phase linear elasticity equations:
\begin{equation}\label{eq: 2phase linear elasticity}
\begin{aligned}
&\text{Find $[\boldsymbol u_1, \boldsymbol u_2]\in\mathcal W_{1,h}\times\mathcal W_{2,h}$ such that}
\\
&B([\boldsymbol{u}_1,\boldsymbol{u}_2],[\boldsymbol{v}_1,\boldsymbol{v}_2];\phi_1,\phi_2)=l_s([\boldsymbol{v}_1,\boldsymbol{v}_2];\phi_1,\phi_2)\qquad\forall [\boldsymbol v_1,\boldsymbol v_2]\in\mathcal W_{1,h}\times\mathcal W_{2,h},
\end{aligned}
\end{equation}
where
\begin{equation}\label{eqn: A canti example}
    \begin{aligned}
        B([\boldsymbol{u}_1,\boldsymbol{u}_2]&,[\boldsymbol{v}_1,\boldsymbol{v}_2];\phi_1,\phi_2)= (\boldsymbol{\varepsilon}(\boldsymbol{v}_1),\boldsymbol{\sigma}_1(\boldsymbol{u}_1))_{\overline{D}_{\phi_1}^\complement\cap D_{\phi_2}}+(\boldsymbol{\varepsilon}(\boldsymbol{v}_2),\boldsymbol{\sigma}_2(\boldsymbol{u}_2))_{D_{\phi_1}\cap D_{\phi_2}}
        \\&\quad+\sum_{F\in\mathcal{F}_{G_2}}(\eta_{1}(h)\jump{\partial_{\boldsymbol{n}_F}\boldsymbol{u}_1},\jump{\partial_{\boldsymbol{n}_F}\boldsymbol{v}_1})_F+\sum_{F\in\mathcal{F}_{G_3}}(\eta_{2}(h)\jump{\partial_{\boldsymbol{n}_F}\boldsymbol{u}_2},\jump{\partial_{\boldsymbol{n}_F}\boldsymbol{v}_2})_F,
        \\&\quad-(\mean{\boldsymbol{\sigma}_i(\boldsymbol{u}_i)},\jump{ \boldsymbol{v}_i})_{\partial D_{\phi_1}\cap D_{\phi_2}}-(\mean{\boldsymbol{\sigma}_i(\boldsymbol{v}_i)},\jump{ \boldsymbol{u}_i})_{ \partial D_{\phi_1}\cap D_{\phi_2}}+(\mu(h)\jump{\boldsymbol{u}_i},\jump{\boldsymbol{v}_i})_{\partial D_{\phi_1}\cap D_{\phi_2}}
        \\&\quad+(\chi \boldsymbol{u}_1,\boldsymbol{v}_1)_{\overline{D}_{\phi_1}^\complement\cap D_{\phi_2}}+(\chi \boldsymbol{u}_2,\boldsymbol{v}_2)_{D_{\phi_1}\cap D_{\phi_2}},
    \end{aligned}
\end{equation}
and
\begin{equation}
    l_s([\boldsymbol{v}_1,\boldsymbol{v}_2];\phi_1,\phi_2) = \sum_i(\boldsymbol{g},\boldsymbol{v}_i)_{\Gamma_N}.
\end{equation}
Here, $\mathcal{W}_{i,h}=\bigl\{\,\boldsymbol{v}\in C^0(\mathcal{T}_{i,h})^d:\boldsymbol{v}\rvert_{K}\in\mathbb{P}_1(K)^d~\forall K\in\mathcal{T}_{i,h},~\boldsymbol{v}\rvert_{\Gamma_D}=\boldsymbol{0}\,\bigr\}$ and $\mathcal{T}_{1,h}$ and $\mathcal{T}_{2,h}$ are the active elements for $\overline{D}_{\phi_1}^\complement\cap D_{\phi_2}$ and $D_{\phi_1}\cap D_{\phi_2}$, respectively; $\boldsymbol{\sigma}_i$ is the stress tensor in each phase given by
\begin{equation}\label{eqn: stress}
    \boldsymbol{\sigma}_i(\boldsymbol{u}_i)=\lambda_i\operatorname{tr}(\boldsymbol{\varepsilon}(\boldsymbol{u}_i))\boldsymbol{I}+2\mu_i\boldsymbol{\varepsilon}(\boldsymbol{u}_i),\quad i=1,2
\end{equation}
where $\boldsymbol{\varepsilon}(\boldsymbol{u})=\frac{1}{2}\left(\boldsymbol{\nabla u}+(\boldsymbol{\nabla u})^\intercal\right)$ is the strain tensor, and $\lambda_i$ and $\mu_i$ are the Lamé parameters corresponding to the Young's modulus $E_i$ and Poisson's ratio $\nu_i$ of each material phase; $\boldsymbol{g}$ is the loading on $\Gamma_N$ given by $\boldsymbol{g}=(0,-1)$; and $\Gamma_N$ and $\Gamma_D$ are given by the dark blue and black regions in Figure~\ref{fig:fig9a}, respectively. For the weighting parameters $\kappa_1$ and $\kappa_2$, we take
\begin{align}
    \kappa_1=\frac{E_2}{E_1+E_2},\quad
    \kappa_2=\frac{E_1}{E_1+E_2}.
\end{align}
Finally, the ghost penalty parameters $\eta_i(h)$ and Nitsche penalty parameter $\mu(h)$ are set to
\begin{equation}
    \eta_i(h) = 10^{-7}\alpha_i h^3,\quad \mu(h)=\frac{10^2\max(\alpha_1,\alpha_2)}{h},
\end{equation}
where we scale by $\alpha_i=\lambda_i+\mu_i$, based on \citet{10.1016/j.cma.2017.09.005_2018}.

We discretise the above problem over a structured background mesh made of 80,000 simplices and solve the constrained optimisation problem in \eqref{canti optim} using a projection method \citep{Wegert_2023b}, relying on automatic shape differentiation to compute all derivatives.

Figure~\ref{fig:fig9b} shows the optimisation result for $E_1=1.0$, $E_2=0.5$ and $\nu_1=\nu_2=0.3$. As expected from previous results in the literature \citep[e.g.,][]{10.1051/cocv/2013076_2014}, the material with a larger Young's modulus is concentrated at high-stress regions. It is important to note that although this problem is classical, the inclusion of a sharp interface between the phases enables several interesting extensions, such as more complicated interface conditions \citep[e.g.,][]{10.1002/nme.4823_2015}.

\begin{figure}[!t]
    \centering
    \begin{subfigure}{0.49\textwidth}
        \centering
        \includegraphics[width=\linewidth]{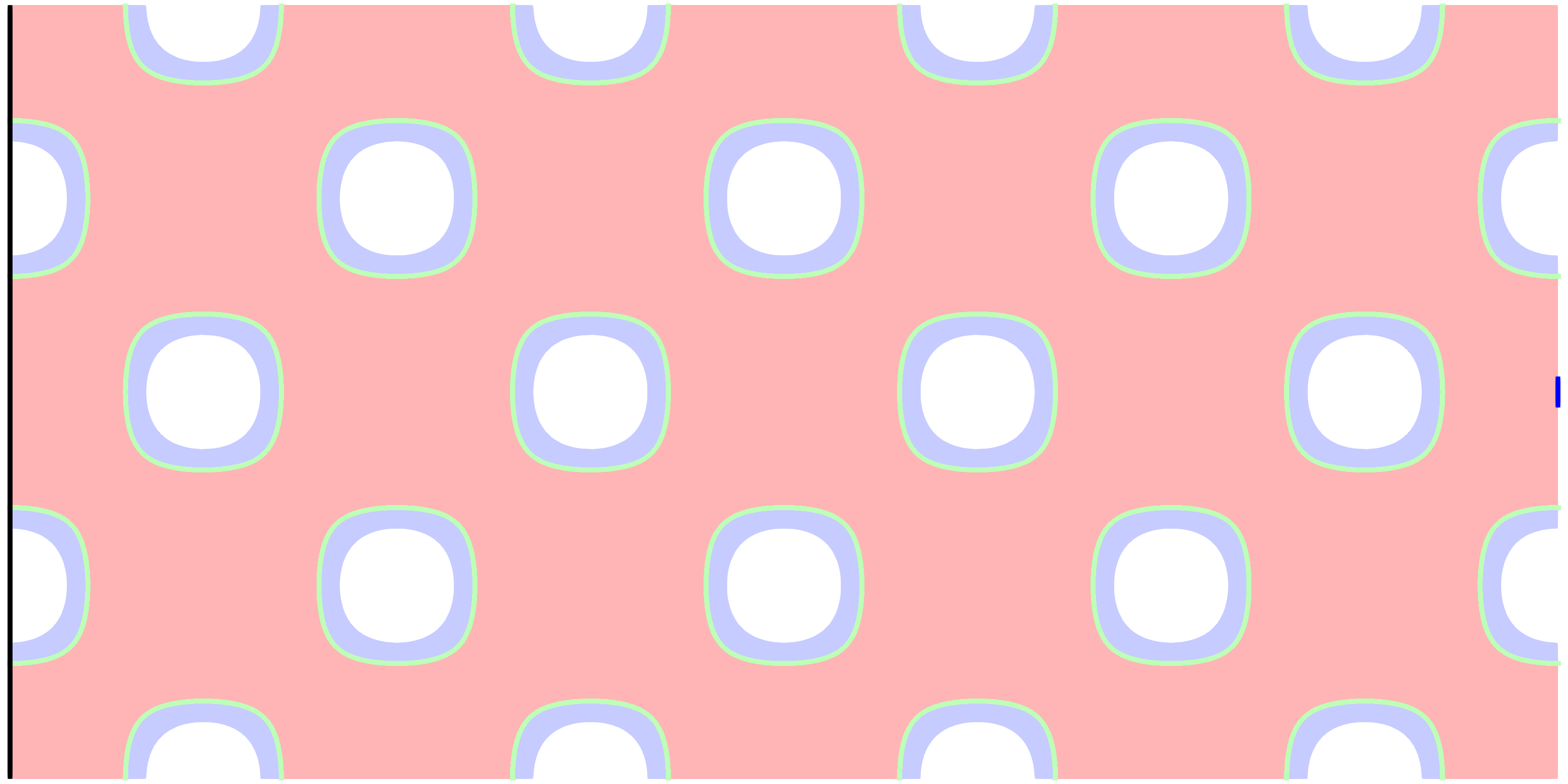}
        \caption{Initial iteration}
        \label{fig:fig9a}
    \end{subfigure}
    \begin{subfigure}{0.49\textwidth}
        \centering
        \includegraphics[width=\linewidth]{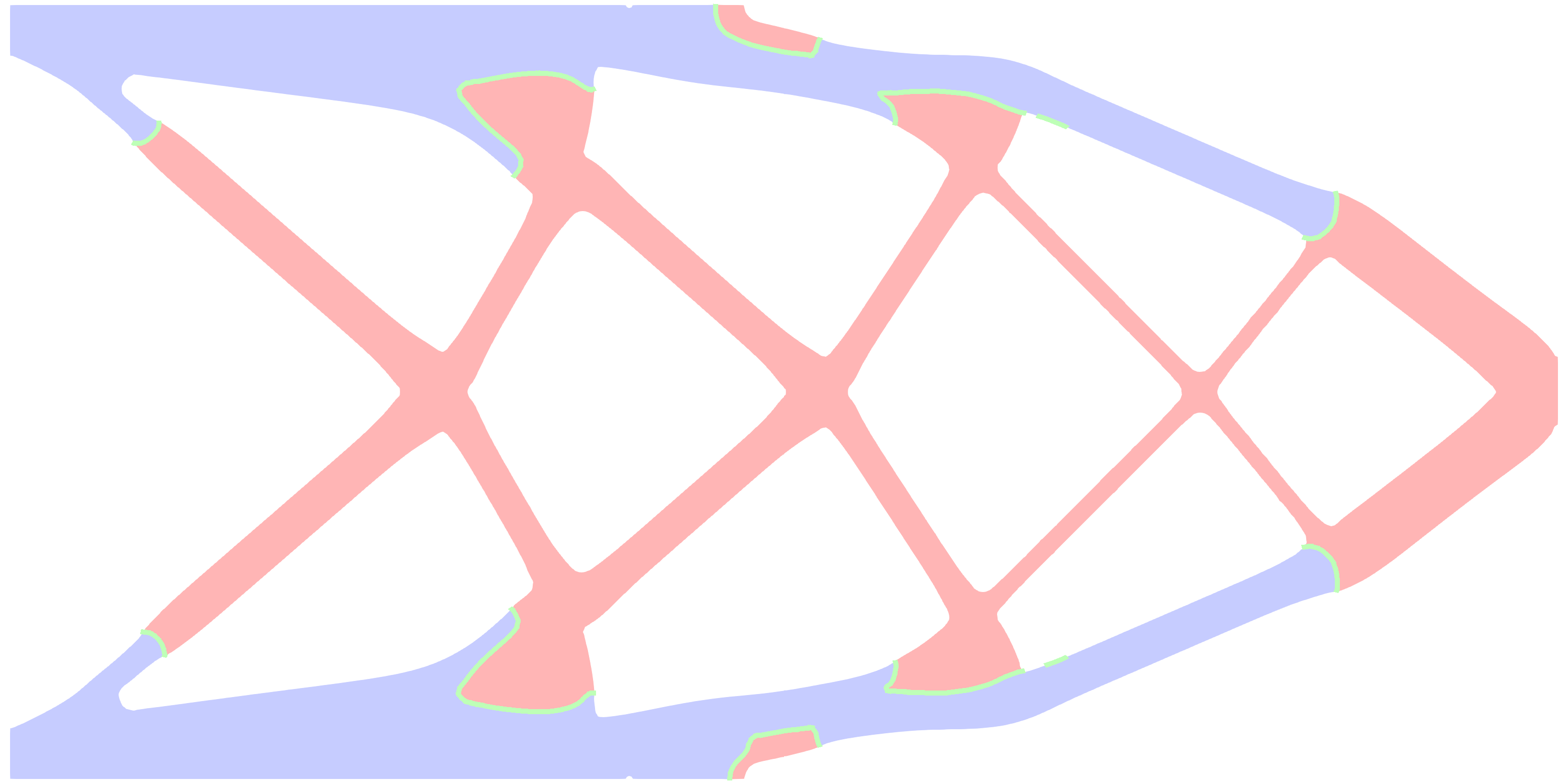}
        \caption{Final iteration}
        \label{fig:fig9b}
    \end{subfigure}
    \caption{Visualisation of the material phases for the multi-phase minimum compliance cantilever problem. The blue and red phases corresponds to linear elastic materials with a Young's modulus of 1 and $1/2$, respectively. In (a), we show the Neumann and Dirichlet regions in dark blue and black, respectively.}
    \label{fig:fig9}
\end{figure}

\subsection{Fluid--structure interaction}\noindent
The final example is a fluid--structure interaction problem, in which the solid structure is made of two material phases. To our knowledge, this is the first time topology optimisation for such a setup has been considered in the literature. This problem exemplifies the type of multi-physics and multi-phase topology optimisation problems with complex interface conditions that the methods we investigate here are able to address.

The formulation of this problem is similar to the example considered by \citet{WEGERT2025118203}, with the addition of multiple phases in the structural component. We assume that the structure is linear elastic and the fluid is governed by Stokes flow.
We consider displacements small enough such that the coupling in the fluid--structure problem is one way. This allows the fluid and elastic problems to be decoupled, so that they can be solved in a sequential manner. For the Stokes flow problem, we use the stabilised Nitsche fictitious domain method and face ghost-penalty stabilisation \cite{10.1002/nme.4823_2015} with continuous piecewise linear elements for the velocities and stabilised piecewise constant elements for the pressures.

The optimisation problem is given by
\begin{equation}\label{fsi optim}
    \begin{aligned}
        \min_{(\phi_1,\phi_2)\in\mathcal{V}_{h}^2}&~(\boldsymbol{\varepsilon}(\boldsymbol{d}_1),\boldsymbol{\sigma}_1(\boldsymbol{d}_1))_{\overline{D}_{\phi_1}^\complement\cap D_{\phi_2}}+(\boldsymbol{\varepsilon}(\boldsymbol{d}_2),\boldsymbol{\sigma}_2(\boldsymbol{d}_2))_{D_{\phi_1}\cap D_{\phi_2}}\\
        {\rm s.t.}&~\mathrm{Vol}(\overline{D}_{\phi_1}^\complement\cap D_{\phi_2})=0.035\mathrm{Vol}(D),\\
        &~\mathrm{Vol}(D_{\phi_1}\cap D_{\phi_2})=0.035\mathrm{Vol}(D),
    \end{aligned}
\end{equation}
where the displacement fields $\boldsymbol{d}_1$ and $\boldsymbol{d}_2$ satisfy the variational form of the two-phase linear elasticity equations:
\begin{equation}\label{eq: 2phase linear elasticity fsi}
\begin{aligned}
&\text{Find $[\boldsymbol d_1, \boldsymbol d_2]\in\mathcal W_{1,h}\times\mathcal W_{2,h}$ such that}
\\
&B([\boldsymbol{d}_1,\boldsymbol{d}_2],[\boldsymbol{s}_1,\boldsymbol{s}_2];\phi_1,\phi_2)=l_f([\boldsymbol{s}_1,\boldsymbol{s}_2];[\boldsymbol{u},p],\phi_1,\phi_2)\qquad\forall [\boldsymbol s_1,\boldsymbol s_2]\in\mathcal W_{1,h}\times\mathcal W_{2,h},
\end{aligned}
\end{equation}
where $B$ is as in \eqref{eqn: A canti example} and $l_f$ is given by
\begin{equation}
    l_f([\boldsymbol{s}_1,\boldsymbol{s}_2];[\boldsymbol{u},p],\phi_1,\phi_2)=\left((1-\chi_s)\boldsymbol{s}_1,(\boldsymbol{\sigma}_f(\boldsymbol{u},p)\cdot\bn)\right)_{\Gamma_{f,1}}+\left((1-\chi_s)\boldsymbol{s}_2,(\boldsymbol{\sigma}_f(\boldsymbol{u},p)\cdot\bn)\right)_{\Gamma_{f,2}}.
\end{equation}
Here, the structure is subject to stress $\sigma_f$ due to Stokes flow, where the velocity $\boldsymbol{u}$ and pressure $p$ satisfy the variational form of Stokes flow:
\begin{equation}\label{eqn: stokes}
\begin{aligned}
&\text{Find $[\boldsymbol u, p]\in\mathcal U_{h,\boldsymbol{u}_{in}}\times\mathcal Q_{h}$ such that}
\\
&R([\boldsymbol{u},p],[\boldsymbol{v},q];\phi_1,\phi_2)=0\qquad\forall [\boldsymbol v,\boldsymbol q]\in\mathcal U_{h,\boldsymbol{0}}\times\mathcal Q_{h},
\end{aligned}
\end{equation}
where
\begin{equation}
\begin{aligned}
    R([\boldsymbol{u},p]&,[\boldsymbol{v},q];\phi_1,\phi_2)=(\bnabla\boldsymbol{u},\bnabla\boldsymbol{v})_{\Omega_f}-(\partial_{\bn}\boldsymbol{u},\boldsymbol{v})_{\Gamma}-(\partial_{\bn}\boldsymbol{v},\boldsymbol{u})_{\Gamma}+(\gamma_\Gamma\boldsymbol{u},\boldsymbol{v})_{\Gamma}-(p,\bnabla\cdot \boldsymbol{v})_{\Omega_f}-(q,\bnabla\cdot \boldsymbol{u})_{\Omega_f}\\
    &-(p\bn,\boldsymbol{v})_{\Gamma}-(q\bn,\boldsymbol{u})_{\Gamma}+\sum_{F\in\mathcal{F}_{G,f}}(\gamma_{\boldsymbol{u}}\jump{\partial_{\boldsymbol{n}_F}\boldsymbol{u}},\jump{\partial_{\boldsymbol{n}_F}\boldsymbol{v}})_F-\sum_{F\in\mathcal{F}_{f}}(\gamma_{p}\jump{\boldsymbol{u}},\jump{\boldsymbol{v}})_F+(\chi_f p,q)_{\Omega_f}.
\end{aligned}
\end{equation}
Here, $\mathcal{W}_{i,h}=\bigl\{\,\boldsymbol{v}\in C^0(\mathcal{T}_{i,h})^d:\boldsymbol{v}\rvert_{K}\in\mathbb{P}_1(K)^d~\forall K\in\mathcal{T}_{i,h},~\boldsymbol{v}\rvert_{\Gamma_{\rm bottom}}=\boldsymbol{0}\,\bigr\}$ are the spaces of the displacements, where $\mathcal{T}_{1,h}$ and $\mathcal{T}_{2,h}$ are the active elements for $\overline{D}_{\phi_1}^\complement\cap D_{\phi_2}$ and $D_{\phi_1}\cap D_{\phi_2}$, respectively; $\mathcal{U}_{h,\boldsymbol{w}}=\bigl\{\,\boldsymbol{v}\in C^0(\mathcal{T}_{f,h})^d:\boldsymbol{v}\rvert_{K}\in\mathbb{P}_1(K)^d~\forall K\in\mathcal{T}_{i,h},~\boldsymbol{v}\rvert_{\Gamma_{\rm inlet}}=\boldsymbol{w},~\boldsymbol{v}\rvert_{\Gamma_{\rm bottom}\cup\Gamma_{\rm top}}=\boldsymbol{0}\,\bigr\}$ and $\mathcal{Q}_{h}=\bigl\{\,p\in L^2(\mathcal{T}_{f,h}):p\rvert_{K}\in\mathbb{P}_0(K)~\forall K\in\mathcal{T}_{f,h}\,\bigr\}$ are the velocity and pressure spaces, respectively, where $\mathcal{T}_{f,h}$ are the active elements for the fluid phase $\Omega_f=({D}_{\phi_1}\cap \overline{D}_{\phi_2}^\complement)\cup(\overline{D}_{\phi_1}^\complement\cap\overline{D}_{\phi_2}^\complement)$; $\Gamma_{f,i}$ is the interface between the fluid phase and the $i^{\rm th}$ structural phase with  $\Gamma=\Gamma_{f,1}\cup\Gamma_{f,2}$; $\mathcal{F}_{G,f}$ is the ghost skeleton for the fluid domain; $\mathcal{F}_{f}$ is the set of facets for the active domain defined by $\Omega_f$; $\chi_f$ is the indicator function for isolated volumes for the fluid phase; $\boldsymbol{\sigma}_f(\boldsymbol{u},p)=2\boldsymbol{\varepsilon}(\boldsymbol{u})-p\boldsymbol{I}$ is the Stokes stress tensor; $\chi_s$ is the indicator function for isolated volumes in the structural problem; $\gamma_\Gamma=100/h$ is the Nitsche parameter; $\gamma_{\boldsymbol{u}}=0.1h$ is the ghost penalty parameter; and $\gamma_p=0.25h$ is the symmetric pressure stabilisation parameter. Note that following \citet{WEGERT2025118203}, we include $1-\psi_s$ in $l_s$ to ensure that isolated volumes of the solid phase have zero displacement.

The background domain $D$ has a length and height of 1 and 1/2, respectively. For the boundary conditions on the fluid, we prescribe no-slip boundary conditions on the top and bottom of the bounding domain $D$, a parabolic inlet velocity of $\boldsymbol{u}_{\rm in}=(16y(0.5-y)\boldsymbol{e}_1$, a vanishing natural boundary condition for the Stokes variational form on the outlet, and a no-slip boundary condition on the fluid--structure interface. For the structural component, in addition to the forcing caused by the fluid, we have zero-displacement conditions on the bottom of the bounding domain. We include a non-designable support region for the structure shown in Figure~\ref{fig: fig10a}. Finally, the structural phases have a Young's modulus of $1$ and $0.1$, as well as a Poisson's ratio of $0.3$.

We discretise the above problem over an unstructured background mesh with refinement around the non-designable region to improve accuracy of the solutions near the interface of the elastic structure. The non-refined region has a maximum element size of 0.025, while the refined region has a maximum element size of 0.0025. We solve the constrained optimisation problem in \eqref{fsi optim} using a projection method \citep{Wegert_2023b}. To ensure that the optimiser does not favour isolated volumes that disrupt the fluid flow by, a term is added to the objective to penalise isolated volumes in the solid phases \citep{WEGERT2025118203}. Here, we take this penalisation term to be $I(\phi_1,\phi_2)=\int_{(\overline{D}_{\phi_1}^\complement\cap D_{\phi_2})\cup(D_{\phi_1}\cap D_{\phi_2})}h^{-2}\chi_{s}~\mathrm{d}x$, which vanishes when no isolated regions are present. Finally,  automatic shape differentiation is used to compute all derivatives.

Figure~\ref{fig:fig10} shows the material phases at the initial and final optimisation iterations. As in the case of the cantilever problem, the material with a larger Young's modulus is concentrated at high-stress regions. These results demonstrate that multi-phase unfitted discretisations in conjunction with automatic shape differentiation can be used to effectively solve multi-phase and multi-physics topology optimisation problems.

\begin{figure}[!t]
    \centering
    \begin{subfigure}{0.7\textwidth}
        \centering
        \includegraphics[width=\linewidth]{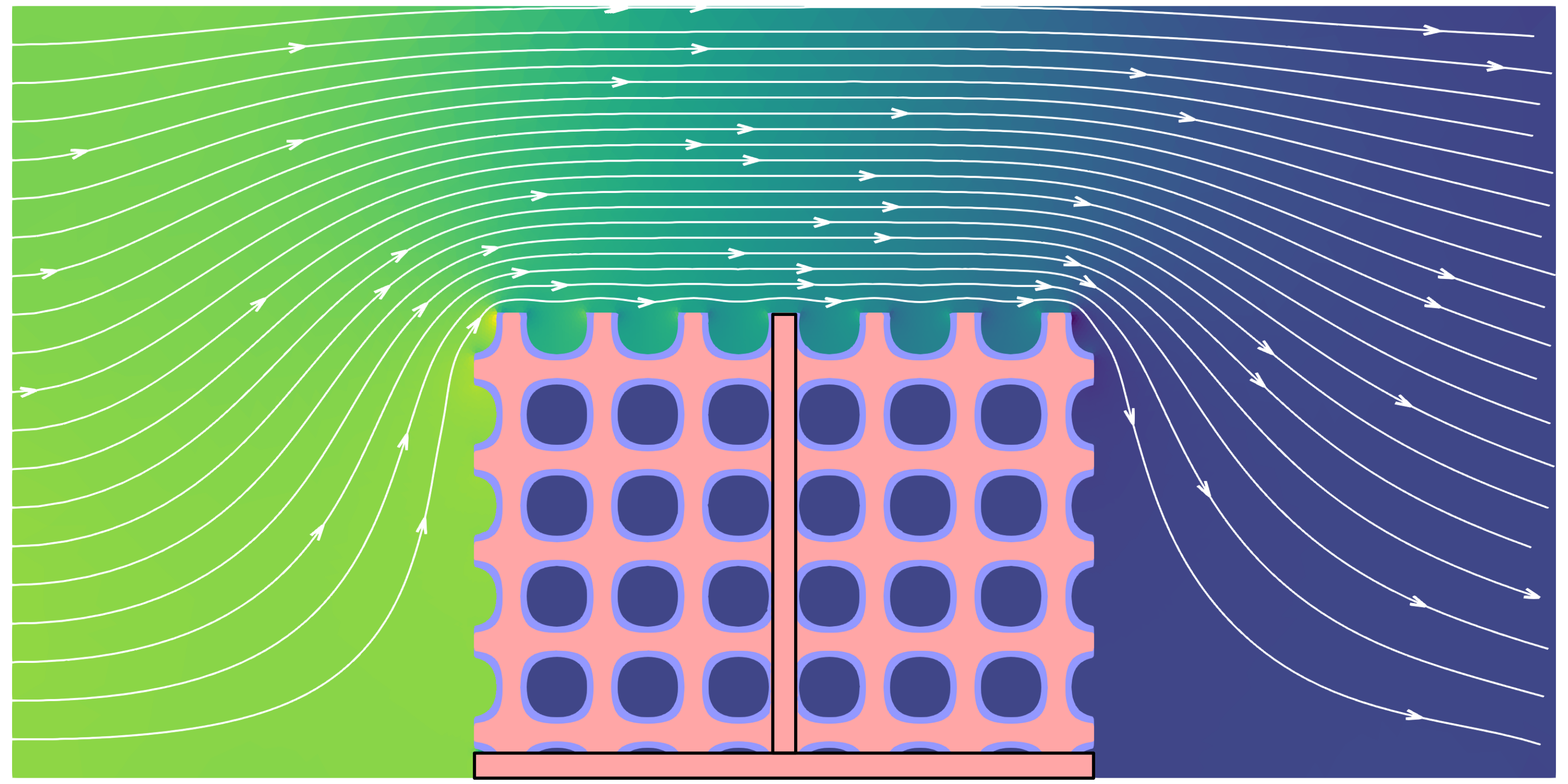}
        \caption{Initial iteration}
        \label{fig: fig10a}
    \end{subfigure}
    \begin{subfigure}{0.7\textwidth}
        \centering
        \includegraphics[width=\linewidth]{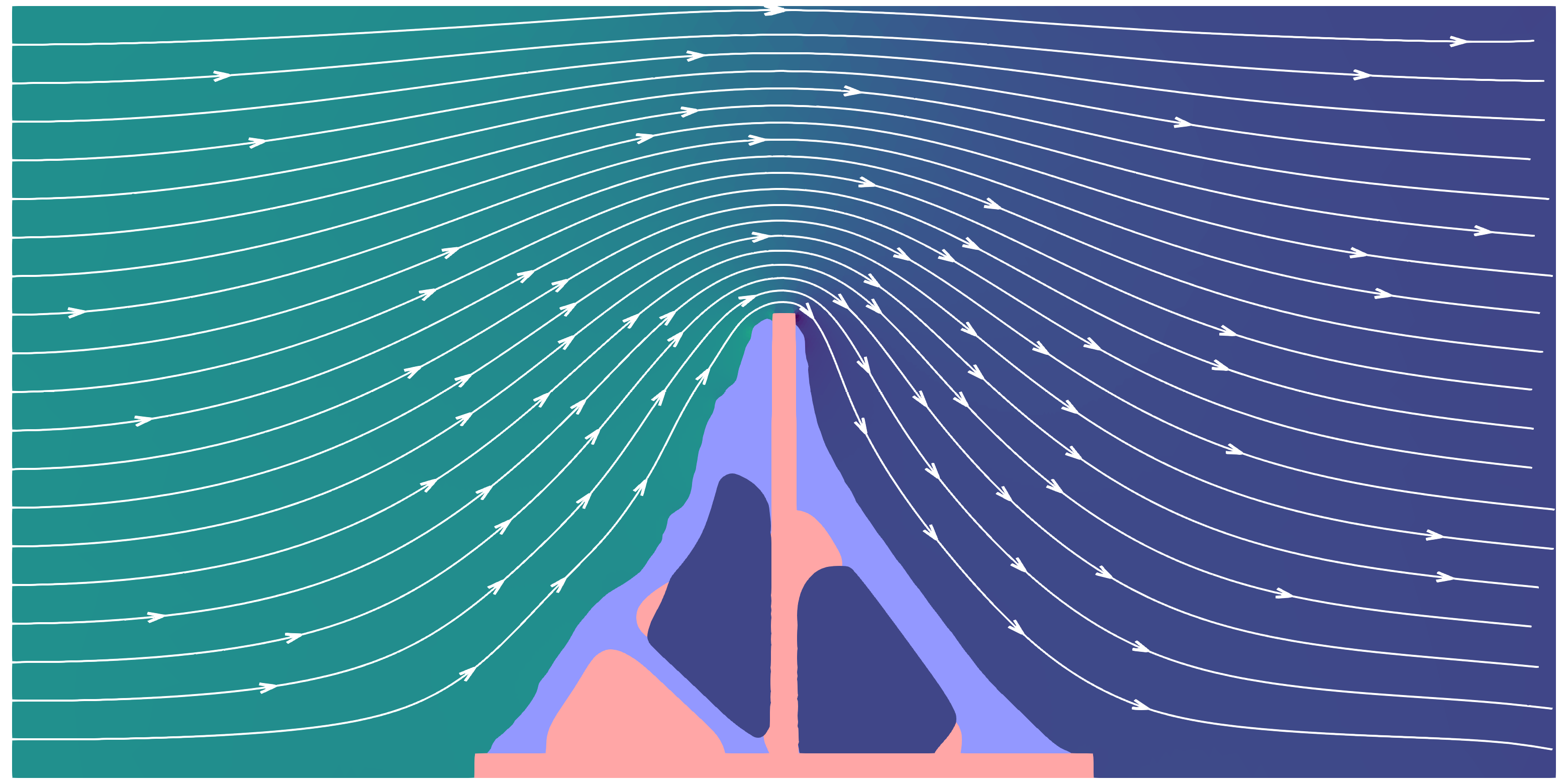}
        \includegraphics[width=1.0\linewidth]{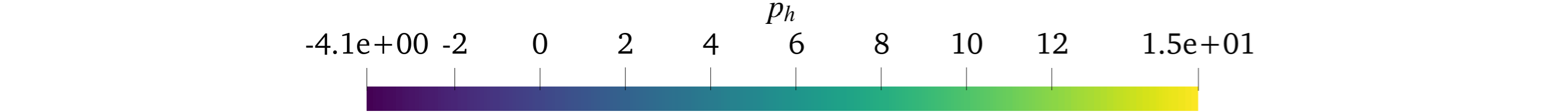}
        \caption{Final iteration}
    \end{subfigure}
    \caption{Visualisation of material phases for the minimum compliance fluid--structure interaction problem.  The blue and red phases corresponds to linear elastic materials with a Young's modulus of 1 and $0.1$, respectively. The pressure field is visualised by the colour map while the direction of flow is indicated by the arrows on the streamlines of the velocity field. The region outlined by black lines in (a) is non-designable.}
    \label{fig:fig10}
\end{figure}

\section{Summary, conclusions, and outlook}\label{sec: Conclusions}\noindent
The focus of this contribution has been on discrete shape calculus techniques and automatic shape differentiation for the case of functionals over domains that are defined by the intersections of multiple level-set functions. For the shape calculus, aimed at discretised cases, we used concepts from convex geometry to generalise the mathematical results of \citet{Berggren_2023} and \citet{WEGERT2025118203} to arbitrary convex polytopes in $\mathbb{R}^d$. We constructed equivalence between level-set and polytopal representations in order to then construct and prove Theorems~\ref{theorem 1}, \ref{theorem 2}, and \ref{theorem 3}. These results establish directional semiderivatives for the three types of functionals: (i) domain integrals (Sec.~\ref{subsec:dom int}); (ii) boundary integrals truncated by a perturbed domain (Sec.~\ref{subsec:boundary int trunc}); and (iii) boundary integrals (Sec.~\ref{subsec:boundary int}). For the case of a single level-set function, our results for (i) and (iii) reduce to those of \citet{Berggren_2023} and \citet{WEGERT2025118203}. Case (ii) is new and appears only in the case of two or more intersecting level-set functions. The new results are general and establish a discrete shape calculus
enabling multi-phase level-set topology optimisation with unfitted discretisations.

The shape calculus results are complemented by our proposed automatic shape differentiation approach for multi-phase problems. In particular, we developed a polytope cutter that provides an extension of the automatic shape differentiation approach of \citet{WEGERT2025118203} to multiple level-set functions. We validated the implementation against finite differences as well as the formulas devised in Theorems~\ref{theorem 1}, \ref{theorem 2}, and \ref{theorem 3}. As expected, the automatic shape differentiation matched the exact derivatives to near machine precision (see Table~\ref{tab:verify}). Our implementation can be leveraged in both serial and distributed CPU computing frameworks, and we demonstrate the scalability of the multi-phase automatic shape differentiation up to 1,658,880,000 cells over 13824 CPUs. The implementation aims for a clear and intuitive application programming interface (API), and is fully compatible with GridapTopOpt's automatic differentiation of arbitrary PDE-constrained maps. Releasing this open-source software will facilitate the solution of complex multi-physics and multi-phase topology optimisation problems.

We demonstrated our framework for multi-phase unfitted level-set topology optimisation by solving topology optimisation problems for multi-phase anisotropic diffusion, multi-phase linear elasticity, and multi-phase fluid--structure interaction. The fluid--structure interaction example especially well demonstrates the power of our approach for solving topology optimisation problems involving multi-phase and multi-physics systems with complex boundary conditions. In particular, interface conditions between structural phases as well as the no-slip boundary condition at the fluid--structure boundaries are applied using unfitted finite elements and Nitsche's method.

In the future, we plan to investigate the use of other unfitted finite element methods such as AgFEM, investigate matrix-free methods to improve the scalability of solvers, and to apply the approach to other complex design problems involving multiple physics and multiple phases.

\section*{CRediT authorship contribution statement}\noindent
\textbf{Zachary J Wegert:} Writing -- original draft, Writing -- review and editing, Conceptualisation, Data curation, Mathematical analysis, Investigation, Methodology, Software, Validation, Visualisation. \textbf{Martin Berggren:} Writing -- review and editing, Mathematical analysis, Investigation, Methodology, Validation. \textbf{Vivien J Challis:} Supervision, Conceptualisation, Writing -- review \& editing, Validation, Project administration, Funding acquisition, Resources.

\section*{Declaration of Competing Interest}\noindent
The authors have no competing interests to declare that are relevant to the content of this article.

\section*{Data availability}\noindent
The source code is available at \url{https://github.com/zjwegert/GridapTopOpt.jl}. To recreate the results, we provide all scripts at \url{https://github.com/zjwegert/Wegert_et_al_2026_MP}.

\section*{Acknowledgement}\noindent
This research was funded (partially or fully) by the Australian Government through the Australian Research Council (ARC), Discovery Grant DP220102759. This research used computational resources provided by the Queensland Cyber Infrastructure Foundation (QCIF) and the National Computational Infrastructure (NCI) Australia. The first author would also like to thank Jordi Manyer for many insightful discussions regarding cutting algorithms.

\appendix

\bibliographystyle{elsarticle-num-names}
\bibliography{main}

\end{document}


\maketitle

\section{Derivation of directional derivatives in Section 3.6}

\noindent In the following, we derive the directional derivatives for the example in Section~3.6.
As noted in Remark~3.22, the derivation is merely formal because it assumes shape differentiability of $u_1$ and $u_2$. 
We define the Lagrangian $\mathcal{L}$ to be
    \begin{equation}
    \mathcal{L}(P,\Lambda;\phi_1,\phi_2)=F(P,P;\phi_1,\phi_2)-A(P,\Lambda;\phi_1,\phi_2)+l(\Lambda;\phi_1,\phi_2),
\end{equation}
where $P=[p_1,p_2]$ and $\Lambda=[\lambda_1,\lambda_2]$.

Since the Lagrangian is linear in $\Lambda$, differentiation of the Lagrangian with respect to $\Lambda$ in the direction $W$ generates the state equation residual, that is,
\begin{equation}
d_\Lambda\mathcal L(P,\Lambda;\phi_1,\phi_2)(W) \equiv
\frac d{dt}\mathcal L(P, \Lambda + tW;\phi_1,\phi_2)\big|_{t=0} =
-A(P, W; \phi_1,\phi_2) + l(W;\phi_1,\phi_2),
\end{equation}
which vanishes for each $W\in \mathcal V_{1,h}\times\mathcal V_{2,h}$ when $P=U$ due to state equation (84). Next, differentiation of the Lagrangian with respect to $P$ in the direction $R$ generates a residual of the adjoint equation, that is,
\begin{equation}\label{eqn: adj residual example}
\begin{aligned}
d_P\mathcal L(P,\Lambda;\phi_1,\phi_2)(R)  &\equiv
\frac d{dt}\mathcal L(P+tR,\Lambda;\phi_1,\phi_2)\big|_{t=0}
\\
&= d_P F(P,P;\phi_1,\phi_2)(R) - d_P A(P, \Lambda;\phi_1,\phi_2)(R)
= 2F(P,R;\phi_1,\phi_2) - A(R, \Lambda;\phi_1,\phi_2).
\end{aligned}
\end{equation}

Finally, under the assumption that $U$ is shape differentiable, the directional derivative of $J$ in the direction $w$ is obtained by differentiating the Lagrangian with respect to $\phi_i$ in the direction $w$ and evaluating the resulting expression for $P=U$ and $\Lambda = \hat\Lambda$, that is,
\begin{equation}\label{eqn: 94}
\begin{aligned}
d_{\phi_i}J(\phi_1,\phi_2)(w) &= d_{\phi_i}\mathcal L(P,\Lambda;\phi_1,\phi_2)\big|_{P=U, \Lambda=\hat\Lambda}(w)
\\
&= d_{\phi_i} F(U,U;\phi_1,\phi_2)(w) - d_{\phi_i} A(U,\hat\Lambda;\phi_1,\phi_2)(w) + d_{\phi_i}l(\hat\Lambda;\phi_1,\phi_2)(w).
\end{aligned}
\end{equation}
where $\hat{\Lambda}$ satisfies the adjoint problem: find $\hat{\Lambda}\in\mathcal{V}_{1, h}\times\mathcal{V}_{2, h}$ such that
\begin{equation}
   A(R, \hat\Lambda;\phi_1,\phi_2) = 2F(U,R;\phi_1,\phi_2)\qquad\forall R\in\mathcal{V}_{1, h}\times\mathcal{V}_{2, h}.
\end{equation}

We now consider derivatives of the expressions $d_{\phi_i} F(U,U;\phi_1,\phi_2)(w)$, $- d_{\phi_i} A(U,\hat\Lambda;\phi_1,\phi_2)(w)$, and $d_{\phi_i}l(\hat\Lambda;\phi_1,\phi_2)(w)$ occurring in \eqref{eqn: 94} term by term.
(Note that these, by the scheme laid out in Section~3.6, are partial derivatives, where $U$ and $\hat\Lambda$ are considered constant.)
First, for the terms $-(\Ab_1\bnabla u_1,\bnabla \lambda_1)_{D_{\phi_1}\cap \overline{D}_{\phi_2}^\complement}-(\Ab_2\bnabla u_2,\bnabla \lambda_2)_{D_{\phi_1}\cap D_{\phi_2}}$ that appear in $A(U,\hat\Lambda;\phi_1,\phi_2)$, we obtain, using Theorem~3.10, that
    \begin{equation}\label{eqn: dphi1Aint}
        \begin{aligned}
            &d_{\phi_1}\left[-(\Ab_1\bnabla u_1,\bnabla \hat\lambda_1)_{D_{\phi_1}\cap \overline{D}_{\phi_2}^\complement}-(\Ab_2\bnabla u_2,\bnabla \hat\lambda_2)_{D_{\phi_1}\cap D_{\phi_2}}\right](w)
            \\&\qquad=\int_{\partial D_{\phi_1}\cap \overline{D}_{\phi_2}^\complement}(\Ab_1\bnabla u_1)\cdot\bnabla \hat\lambda_1\frac{w}{\lvert\partial_{\bn_{\phi_1}}\phi_1\rvert}~\mathrm{d}S+\int_{\partial D_{\phi_1}\cap D_{\phi_2}}(\Ab_2\bnabla u_2)\cdot\bnabla \hat\lambda_2\frac{w}{\lvert\partial_{\bn_{\phi_1}}\phi_1\rvert}~\mathrm{d}S,
        \end{aligned}
    \end{equation}
    \begin{equation}
        \begin{aligned}\label{eqn: dphi2Aint}
             & d_{\phi_2}\left[-(\Ab_1\bnabla u_1,\bnabla \hat\lambda_1)_{D_{\phi_1}\cap \overline{D}_{\phi_2}^\complement}-(\Ab_2\bnabla u_2,\bnabla \hat\lambda_2)_{D_{\phi_1}\cap D_{\phi_2}}\right](w)
            \\&\qquad=-\int_{D_{\phi_1}\cap\partial D_{\phi_2}}(\Ab_1\bnabla u_1)\cdot\bnabla \hat\lambda_1\frac{w}{\lvert\partial_{\bn_{\phi_2}}\phi_2\rvert}~\mathrm{d}S+\int_{D_{\phi_1}\cap\partial D_{\phi_2}}(\Ab_2\bnabla u_2)\cdot\bnabla \hat\lambda_2\frac{w}{\lvert\partial_{\bn_{\phi_2}}\phi_2\rvert}~\mathrm{d}S,
        \end{aligned}
    \end{equation}
    and, similarly,
    \begin{equation}\label{eqn: dphi1F}
        \begin{aligned}
            & d_{\phi_1}F(U,U;\phi_1,\phi_2)(w)
            \\&\qquad=d_{\phi_1}\left[(\Ab_1\bnabla u_1,\bnabla u_1)_{D_{\phi_1}\cap \overline{D}_{\phi_2}^\complement}+(\Ab_2\bnabla u_2,\bnabla u_2)_{D_{\phi_1}\cap D_{\phi_2}}\right](w)
            \\&\qquad=-\int_{\partial D_{\phi_1}\cap \overline{D}_{\phi_2}^\complement}(\Ab_1\bnabla u_1)\cdot\bnabla u_1\frac{w}{\lvert\partial_{\bn_{\phi_1}}\phi_1\rvert}~\mathrm{d}S-\int_{\partial D_{\phi_1}\cap D_{\phi_2}}(\Ab_2\bnabla u_2)\cdot\bnabla u_2\frac{w}{\lvert\partial_{\bn_{\phi_1}}\phi_1\rvert}\,\mathrm{d}S,
        \end{aligned}
    \end{equation}
    \begin{equation}\label{eqn: dphi2F}
        \begin{aligned}
             & d_{\phi_2}F(U,U;\phi_1,\phi_2)(w)
            \\&\qquad= d_{\phi_2}\left[(\Ab_1\bnabla u_1,\bnabla u_1)_{D_{\phi_1}\cap \overline{D}_{\phi_2}^\complement}+(\Ab_2\bnabla u_2,\bnabla u_2)_{D_{\phi_1}\cap D_{\phi_2}}\right](w)
            \\&\qquad=\int_{D_{\phi_1}\cap\partial D_{\phi_2}}(\Ab_1\bnabla u_1)\cdot\bnabla u_1\frac{w}{\lvert\partial_{\bn_{\phi_2}}\phi_2\rvert}~\mathrm{d}S-\int_{D_{\phi_1}\cap\partial D_{\phi_2}}(\Ab_2\bnabla u_2)\cdot\bnabla u_2\frac{w}{\lvert\partial_{\bn_{\phi_2}}\phi_2\rvert}\,\mathrm{d}S,
        \end{aligned}
    \end{equation}
In this example, we will consider the heat source boundary $\Gamma_N$ fixed, independent of the design.
Therefore, trivially,
\begin{equation}\label{eqn: dphiil}
d_{\phi_i} l(\hat\Lambda;\phi_1,\phi_2)(w) = 0.
\end{equation}
Also, since perturbations of $\phi_1$ and $\phi_2$ satisfy Assumption~3.5,
\begin{equation}\label{eqn: dphiij}
d_{\phi_i} j(V,U) =0.
\end{equation}


Next, we consider the directional derivatives of the interface terms in  $-A(U,\hat\Lambda;\phi_1,\phi_2)$, that is,
\begin{equation}\label{eqn: interface terms}
     (\mean{\Ab\bnabla u},\jump{\hat\lambda})_{D_{\phi_1}\cap\partial D_{\phi_2}}+(\mean{\Ab\bnabla \hat\lambda},\jump{u})_{D_{\phi_1}\cap\partial D_{\phi_2}}-(\mu\jump{\hat\lambda},\jump{u})_{D_{\phi_1}\cap\partial D_{\phi_2}}.
\end{equation}
Recall that there are two different results for differentiation of boundary integrals, Theorems~3.15 and~3.18.
The former concerns integrals that are truncated by a domain that is perturbed by the level-set function.
(For an illustration, consider differentiating integrals over $\Gamma_{12}$ under a perturbation of $\phi_1$ in Figure~1.)
This is precisely the case when differentiating the interface terms~\eqref{eqn: interface terms} under a perturbation of $\phi_1$.
Using Theorem~3.15, we thus have
\begin{equation}\label{eqn: dphi1interface}
\begin{aligned}
        &d_{\phi_1}\left[(\mean{\Ab\bnabla u},\jump{\hat\lambda})_{D_{\phi_1}\cap\partial D_{\phi_2}}+(\mean{\Ab\bnabla \hat\lambda},\jump{u})_{D_{\phi_1}\cap\partial D_{\phi_2}}-(\mu\jump{\hat\lambda},\jump{u})_{D_{\phi_1}\cap\partial D_{\phi_2}}\right](w)\\&\quad=-\int_{\partial D_{\phi_1}\cap\partial D_{\phi_2}}\left(\mean{\Ab\bnabla u}\cdot\jump{\hat\lambda}+\mean{\Ab\bnabla \hat\lambda}\cdot\jump{u}-\mu\jump{u}\jump{\hat\lambda}\right)\frac{w}{\lvert\partial_{\bN^2}\phi_1\rvert}~\mathrm{d}\gamma,
\end{aligned}
\end{equation}
where $\bN^2$ is the unit normal to $\partial D_{\phi_1}$ inside $\partial D_{\phi_2}$, outward with respect to $D_{\phi_2}$.

In contrast, the differentiation of interface terms~\eqref{eqn: interface terms} with respect to $\phi_2$ concerns a boundary given by a level-set function that itself is perturbed, a case that is covered by Theorem~3.18.
(For an illustration, consider differentiating integrals over $\Gamma_{12}$ under a perturbation of $\phi_2$ in Figure~1.)
Moreover, both the average and jump operators in expression~\eqref{eqn: interface terms} involve the normal field, which in turn depends on $\phi_2$.
We will thus require the directional derivative of the outward-directed normal field $\bn_{\phi_2}$ on $\partial D_{\phi_2}$.
    As shown by~\citet[expression~(B18)]{10.1002/nme.70284}, the derivative of $\bn_{\phi_2}$ with respect to $\phi_2$ in the direction $w$ is given by
    \begin{equation}\label{eqn: dn}
        \mathrm{d}_{\phi_2}\bn_{\phi_2}(w)=\frac{P_T\bnabla w}{\lvert\partial_{\bn_{\phi_2}}\phi_2\rvert},
    \end{equation}
where $P_{T}=\boldsymbol{I}-\bn_{\phi_2}\otimes\bn_{\phi_2}$ is the projector on to the tangent plane of $\partial D_{\phi_2}$.
Using this result, Theorem~3.18, and the product rule as discussed in Remark~3.21, we have, using notation further elaborated on below,
\begin{equation}\label{eqn: dphi2interface}
\begin{aligned}
    &d_{\phi_2}\left[(\mean{\Ab\bnabla u},\jump{\hat\lambda})_{D_{\phi_1}\cap\partial D_{\phi_2}}+(\mean{\Ab\bnabla \hat\lambda},\jump{u})_{D_{\phi_1}\cap\partial D_{\phi_2}}-(\mu\jump{\hat\lambda},\jump{u})_{D_{\phi_1}\cap\partial D_{\phi_2}}\right](w)\\
    &\quad=\int_{D_{\phi_1}\cap\partial D_{\phi_2}}d_{\phi_2}
    \mean{\Ab\bnabla u}(w)\cdot\jump{\hat\lambda}+d_{\phi_2}\mean{\Ab\bnabla\hat\lambda}(w)\cdot\jump{u}\,\mathrm{d}S
    \\&\quad+\int_{D_{\phi_1}\cap\partial D_{\phi_2}}
    \mean{\Ab\bnabla u}\cdot d_{\phi_2}\jump{\hat\lambda}(w)
     +\mean{\Ab\bnabla\hat\lambda}\cdot d_{\phi_2}\jump{u}(w)\,\mathrm{d}S
     \\&\quad
     -\mu\int_{D_{\phi_1}\cap\partial D_{\phi_2}}
     d_{\phi_2}\jump{\hat\lambda}(w)\cdot\jump{u}
     +\jump{\hat\lambda}\cdot d_{\phi_2}\jump{u}(w)\,\mathrm{d}S
     \\&\quad
     -\int_{D_{\phi_1}\cap\partial D_{\phi_2}}
     \Bigl(\mean{\Ab\partial_{\bn_{\phi_2}}\bnabla u}\jump{\hat\lambda}
     + \mean{\Ab\bnabla u}\jump{\partial_{\bn_{\phi_2}}\hat\lambda}\Bigr)\frac{w}{\lvert\partial_{\bn_{\phi_2}}\phi_2\rvert}\,\mathrm{d}S
     \\&\quad
     -\int_{D_{\phi_1}\cap\partial D_{\phi_2}}
     \Bigl(\mean{\Ab\partial_{\bn_{\phi_2}}\bnabla\hat\lambda}\jump{u}
     + \mean{\Ab\bnabla\hat\lambda}\jump{\partial_{\bn_{\phi_2}}u}\Bigr)\frac{w}{\lvert\partial_{\bn_{\phi_2}}\phi_2\rvert}\,\mathrm{d}S
     \\&\quad
     +\mu\int_{D_{\phi_1}\cap\partial D_{\phi_2}}\Bigl(\jump{\partial_{\bn_{\phi_2}}\hat\lambda}\jump{u}
     +\jump{\hat\lambda}\jump{\partial_{\bn_{\phi_2}}u}\Bigr)
     \frac{w}{\lvert\partial_{\bn_{\phi_2}}\phi_2\rvert}\,\mathrm{d}S
     \\&\quad+\sum_{S\in\mathcal{F}_h\setminus\partial D}\int_{S \cap D_{\phi_1} \cap \partial D_{\phi_{2}}}{\Big\llbracket}\bn_{\phi_{2}}{\Big(}\mean{\Ab\bnabla \hat\lambda}\cdot\jump{u}+\mean{\Ab\bnabla u}\cdot\jump{\hat\lambda}-\mu\jump{u}\cdot\jump{\hat\lambda}{\Big)}{\Big\rrbracket}_S\frac{w}{\lvert\partial_{\boldsymbol{N}^S}\phi_{2}\rvert}~\mathrm{d}\gamma
     \\&\quad+\int_{\partial D \cap D_{\phi_1} \cap \partial D_{\phi_{2}}}(\bn_{\partial D}\cdot\bn_{\phi_{2}}) {\Big(}\mean{\Ab\bnabla \hat\lambda}\cdot\jump{u}+\mean{\Ab\bnabla u}\cdot\jump{\hat\lambda}-\mu\jump{u}\cdot\jump{\hat\lambda}{\Big)}\frac{w}{\lvert\partial_{\boldsymbol{N}^D}\phi_{2}\rvert}~\mathrm{d}\gamma
     \\&\quad+\int_{\partial D_{\phi_1} \cap \partial D_{\phi_{2}}}(\bn_{\phi_1}\cdot\bn_{\phi_{2}}){\Big(}\mean{\Ab\bnabla \hat\lambda}\cdot\jump{u}+\mean{\Ab\bnabla u}\cdot\jump{\hat\lambda}-\mu\jump{u}\cdot\jump{\hat\lambda}{\Big)}\frac{w}{\lvert\partial_{\boldsymbol{N}^1}\phi_{1}\rvert}~\mathrm{d}\gamma,
\end{aligned}
\end{equation}
where the three first integrals on the right side correspond to the the $g'$ term in expression~(81), whereas the rest of the terms correspond to the $dJ(\phi_k)(w)$ term.
Regarding the first integral on the right, for $f$ either $u$ or $\hat\lambda$, it holds that
\begin{equation}\label{eqn: dphi2Anablaf}
d_{\phi_2}\mean{\Ab\bnabla f}(w) = d_{\phi_2}(\kappa_1 \Ab_1\nabla f_1 + \kappa_2\Ab_2\nabla f_2)(w)
= \bigl(d_{\phi_2}\kappa_1(w)\bigr)\,\Ab_1\nabla f_1 + \bigl(d_{\phi_2}\kappa_2(w)\bigr)\,\Ab_2\nabla f_1,
\end{equation}
where, by differentiation, utilising definition~(86) and formula~\eqref{eqn: dn} ,
\begin{equation}\label{eqn: dphi2kappa}
\begin{aligned}
\mathrm{d}_{\phi_2}\kappa_1(w) &=
2\bn_{\phi_2}\cdot\frac{c_1 \Ab_2-c_2\Ab_1}{(c_1+c_2)^2}\,\mathrm{d}_{\phi_2}\bn_{\phi_2}(w)
=2\bn_{\phi_2}\cdot\frac{(c_1 \Ab_2-c_2\Ab_1)}{(c_1+c_2)^2}\frac{P_T\bnabla w}{\lvert\partial_{\bn_{\phi_2}}\phi_2\rvert},
\\
\mathrm{d}_{\phi_2}\kappa_2(w) &= 2\bn_{\phi_2}\cdot\frac{c_2\Ab_1-c_1\Ab_2}{(c_1+c_2)^2}\,\mathrm{d}_{\phi_2}\bn_{\phi_2}(w)
= 2\bn_{\phi_2}\cdot\frac{(c_2\Ab_1-c_1\Ab_2)}{(c_1+c_2)^2}\frac{P_T\bnabla w}{\lvert\partial_{\bn_{\phi_2}}\phi_2\rvert}
\end{aligned}
\end{equation}
Moreover, the derivatives of the jumps in the the second and third integrals can be written, where, again, $f$ is either $u$ or $\hat\lambda$,
\begin{equation}\label{eqn: dphi2jumpf}
\begin{aligned}
d_{\phi_2}\jump{f}(w) &= d_{\phi_2}(f_1\bn_1 + f_2\bn_2)(w)
=f_1\,d_{\phi_2}\bn_1(w) + f_2\,d_{\phi_2}\bn_2(w)
\\
&=f_1\bn_1\cdot\bn_1\,d_{\phi_2}\bn_1(w) + f_2\bn_2\cdot\bn_2\,d_{\phi_2}\bn_2(w)
\\
&=f_1\bn_1\cdot\bn_{\phi_2}\,d_{\phi_2}\bn_{\phi_2}(w) + f_2\bn_2\cdot\bn_{\phi_2}\,d_{\phi_2}\bn_{\phi_2}(w)
\\
&= \jump{f\bn_{\phi_2}}\,d_{\phi_2}\bn_{\phi_2}(w)
= \jump{f\bn_{\phi_2}}\frac{P_T\bnabla w}{\lvert\partial_{\bn_{\phi_2}}\phi_2\rvert}
\end{aligned}
\end{equation}
Finally, since we employ piecewise-linear approximations, the first term in the fourth and fifth integrals vanish here, since they involve second derivatives inside the elements.

Bringing everything together, we find, by substituting derivatives~\eqref{eqn: dphi1F}, \eqref{eqn: dphi1Aint}, \eqref{eqn: dphi1interface}, \eqref{eqn: dphiil}, and~\eqref{eqn: dphiij} into expression~\eqref{eqn: 94}, that
\begin{equation}\label{eqn: dphi1J}
\begin{aligned}
    \mathrm{d}_{\phi_1}J(\phi_1,\phi_2)(w)&=\int_{\partial D_{\phi_1}\cap \overline{D}_{\phi_2}^\complement}(\Ab_1\bnabla u_1)\cdot(\bnabla \hat{\lambda}_1-\bnabla u_1)\frac{w}{\lvert\partial_{\bn_{\phi_1}}\phi_1\rvert}~\mathrm{d}S
    \\&\quad+\int_{\partial D_{\phi_1}\cap D_{\phi_2}}(\Ab_2\bnabla u_2)\cdot(\bnabla \hat{\lambda}_2-\bnabla u_2)\frac{w}{\lvert\partial_{\bn_{\phi_1}}\phi_1\rvert}~\mathrm{d}S
    \\&\quad-\int_{\partial D_{\phi_1}\cap\partial D_{\phi_2}}\left(\mean{\Ab\bnabla u}\cdot\jump{\hat{\lambda}}+\mean{\Ab\bnabla \hat{\lambda}}\cdot\jump{u}-\mu\jump{u}\cdot\jump{\hat{\lambda}}\right)\frac{w}{\lvert\partial_{\bN^2}\phi_1\rvert}~\mathrm{d}\gamma,
\end{aligned}
\end{equation}
and, by combining derivatives~\eqref{eqn: dphi2F}, \eqref{eqn: dphi2Aint}, \eqref{eqn: dphi2interface}, \eqref{eqn: dphiil}, and~\eqref{eqn: dphiij} into expression~\eqref{eqn: 94}, that
\begin{equation}\label{eqn: dphi2J}
    \begin{aligned}
    &\mathrm{d}_{\phi_2}J(\phi_1,\phi_2)(w)=-\int_{D_{\phi_1}\cap\partial D_{\phi_2}}(\Ab_1\bnabla u_1)\cdot(\bnabla \hat{\lambda}_1-\bnabla u_1)\frac{w}{\lvert\partial_{\bn_{\phi_2}}\phi_2\rvert}~\mathrm{d}S
    \\&\quad+\int_{D_{\phi_1}\cap\partial D_{\phi_2}}(\Ab_2\bnabla u_2)\cdot(\bnabla \hat{\lambda}_2-\bnabla u_2)\frac{w}{\lvert\partial_{\bn_{\phi_2}}\phi_2\rvert}~\mathrm{d}S
    \\&\quad+\int_{D_{\phi_1}\cap\partial D_{\phi_2}}d_{\phi_2}
    \mean{\Ab\bnabla u}(w)\cdot\jump{\hat\lambda}+d_{\phi_2}\mean{\Ab\bnabla\hat\lambda}(w)\cdot\jump{u}\,\mathrm{d}S
    \\&\quad+\int_{D_{\phi_1}\cap\partial D_{\phi_2}}
    \mean{\Ab\bnabla u}\cdot d_{\phi_2}\jump{\hat\lambda}(w)
     +\mean{\Ab\bnabla\hat\lambda}\cdot d_{\phi_2}\jump{u}(w)\,\mathrm{d}S
     \\&\quad
     -\mu\int_{D_{\phi_1}\cap\partial D_{\phi_2}}
     d_{\phi_2}\jump{\hat\lambda}(w)\cdot\jump{u}
     +\jump{\hat\lambda}\cdot d_{\phi_2}\jump{u}(w)\,\mathrm{d}S
     \\&\quad
     -\int_{D_{\phi_1}\cap\partial D_{\phi_2}}
     \Bigl(\mean{\Ab\partial_{\bn_{\phi_2}}\bnabla u}\jump{\hat\lambda}
     + \mean{\Ab\bnabla u}\jump{\partial_{\bn_{\phi_2}}\hat\lambda}\Bigr)\frac{w}{\lvert\partial_{\bn_{\phi_2}}\phi_2\rvert}\,\mathrm{d}S
     \\&\quad
     -\int_{D_{\phi_1}\cap\partial D_{\phi_2}}
     \Bigl(\mean{\Ab\partial_{\bn_{\phi_2}}\bnabla\hat\lambda}\jump{u}
     + \mean{\Ab\bnabla\hat\lambda}\jump{\partial_{\bn_{\phi_2}}u}\Bigr)\frac{w}{\lvert\partial_{\bn_{\phi_2}}\phi_2\rvert}\,\mathrm{d}S
     \\&\quad
     +\mu\int_{D_{\phi_1}\cap\partial D_{\phi_2}}\Bigl(\jump{\partial_{\bn_{\phi_2}}\hat\lambda}\jump{u}
     +\jump{\hat\lambda}\jump{\partial_{\bn_{\phi_2}}u}\Bigr)
     \frac{w}{\lvert\partial_{\bn_{\phi_2}}\phi_2\rvert}\,\mathrm{d}S
     \\&\quad+\sum_{S\in\mathcal{F}_h\setminus\partial D}\int_{S \cap D_{\phi_1} \cap \partial D_{\phi_{2}}}{\Big\llbracket}\bn_{\phi_{2}}{\Big(}\mean{\Ab\bnabla \hat\lambda}\cdot\jump{u}+\mean{\Ab\bnabla u}\cdot\jump{\hat\lambda}-\mu\jump{u}\cdot\jump{\hat\lambda}{\Big)}{\Big\rrbracket}_S\frac{w}{\lvert\partial_{\boldsymbol{N}^S}\phi_{2}\rvert}~\mathrm{d}\gamma
     \\&\quad+\int_{\partial D \cap D_{\phi_1} \cap \partial D_{\phi_{2}}}(\bn_{\partial D}\cdot\bn_{\phi_{2}}) {\Big(}\mean{\Ab\bnabla \hat\lambda}\cdot\jump{u}+\mean{\Ab\bnabla u}\cdot\jump{\hat\lambda}-\mu\jump{u}\cdot\jump{\hat\lambda}{\Big)}\frac{w}{\lvert\partial_{\boldsymbol{N}^D}\phi_{2}\rvert}~\mathrm{d}\gamma
     \\&\quad+\int_{\partial D_{\phi_1} \cap \partial D_{\phi_{2}}}(\bn_{\phi_1}\cdot\bn_{\phi_{2}}){\Big(}\mean{\Ab\bnabla \hat\lambda}\cdot\jump{u}+\mean{\Ab\bnabla u}\cdot\jump{\hat\lambda}-\mu\jump{u}\cdot\jump{\hat\lambda}{\Big)}\frac{w}{\lvert\partial_{\boldsymbol{N}^1}\phi_{1}\rvert}~\mathrm{d}\gamma,
\end{aligned}
\end{equation}
 where $\hat{\Lambda}=[\hat{\lambda}_1,\hat{\lambda}_2]$ satisfies the adjoint problem: Find $\hat{\Lambda}\in\mathcal{V}_{1,h}\times\mathcal{V}_{2,h}$ such that
 \begin{equation}
     \begin{aligned}
         &(\Ab_1\bnabla w_1,\bnabla \hat{\lambda}_1)_{D_{\phi_1}\cap \overline{D}_{\phi_2}^\complement}+(\Ab_2\bnabla w_2,\bnabla \hat{\lambda}_2)_{D_{\phi_1}\cap D_{\phi_2}}+j(W,\hat{\lambda})+(\mu\jump{\hat{\lambda}},\jump{W})_{D_{\phi_1}\cap\partial D_{\phi_2}}\\
         &\quad
         -(\mean{\Ab\bnabla \hat{\lambda}},\jump{ W})_{D_{\phi_1}\cap\partial D_{\phi_2}}-(\mean{\Ab\bnabla W},\jump{ \hat{\lambda}})_{D_{\phi_1}\cap\partial D_{\phi_2}}=2F(W,\phi_1,\phi_2),
     \end{aligned}
 \end{equation}
 for all $W\in\mathcal{V}_{1,h}\times\mathcal{V}_{2,h}$.

Details on the evaluations of the various directional derivatives inside reduced gradient~\eqref{eqn: dphi2J} are given in expressions~\eqref{eqn: dphi2Anablaf}, \eqref{eqn: dphi2kappa}, and \eqref{eqn: dphi2jumpf}.
Note also that the first term in integrals six and seven in expression~\eqref{eqn: dphi2J} vanishes in this case, since we use piecewise-linear approximation spaces.



\bibliographystyle{elsarticle-num-names}
\bibliography{main}

%% file: Figure2_new_tex.tex
\begingroup%
  \makeatletter%
  \providecommand\color[2][]{%
    \errmessage{(Inkscape) Color is used for the text in Inkscape, but the package 'color.sty' is not loaded}%
    \renewcommand\color[2][]{}%
  }%
  \providecommand\transparent[1]{%
    \errmessage{(Inkscape) Transparency is used (non-zero) for the text in Inkscape, but the package 'transparent.sty' is not loaded}%
    \renewcommand\transparent[1]{}%
  }%
  \providecommand\rotatebox[2]{#2}%
  \newcommand*\fsize{\dimexpr\f@size pt\relax}%
  \newcommand*\lineheight[1]{\fontsize{\fsize}{#1\fsize}\selectfont}%
  \ifx\svgwidth\undefined%
    \setlength{\unitlength}{5483.00266056bp}%
    \ifx\svgscale\undefined%
      \relax%
    \else%
      \setlength{\unitlength}{\unitlength * \real{\svgscale}}%
    \fi%
  \else%
    \setlength{\unitlength}{\svgwidth}%
  \fi%
  \global\let\svgwidth\undefined%
  \global\let\svgscale\undefined%
  \makeatother%
  \begin{picture}(1,0.54061314)%
    \lineheight{1}%
    \setlength\tabcolsep{0pt}%
    \put(0,0){\includegraphics[width=\unitlength,page=1]{Figure2_new.pdf}}%
    \put(0.7920302,0.24521605){\makebox(0,0)[lt]{\lineheight{1.25}\smash{\begin{tabular}[t]{l}
    $\Omega_4=\overline{D}_{\phi_1}^\complement\cap \overline{D}_{\phi_2}^\complement$
    \end{tabular}}}}%
    \put(0.7920302,0.52097744){\makebox(0,0)[lt]{\lineheight{1.25}\smash{\begin{tabular}[t]{l}
    $\Omega_3=\overline{D}_{\phi_1}^\complement\cap D_{\phi_2}$
    \end{tabular}}}}%
    \put(0.5325895,0.24521605){\makebox(0,0)[lt]{\lineheight{1.25}\smash{\begin{tabular}[t]{l}
    $\Omega_2=D_{\phi_1}\cap D_{\phi_2}$
    \end{tabular}}}}%
    \put(0.5325895,0.52097744){\makebox(0,0)[lt]{\lineheight{1.25}\smash{\begin{tabular}[t]{l}
    $\Omega_1=D_{\phi_1}\cap \overline{D}_{\phi_2}^\complement$
    \end{tabular}}}}%
    \put(0.67303449,0.39981741){\color[rgb]{1,0.78039216,0.18823529}\makebox(0,0)[lt]{\lineheight{1.25}\smash{\begin{tabular}[t]{l}$\Gamma_{12}$\end{tabular}}}}%
    \put(0.67372228,0.109385){\color[rgb]{0.41176471,0.00784314,0.49019608}\makebox(0,0)[lt]{\lineheight{1.25}\smash{\begin{tabular}[t]{l}$\Gamma_{23}$\end{tabular}}}}%
    \put(0.78389913,0.39981741){\color[rgb]{0.41176471,0.00784314,0.49019608}\makebox(0,0)[lt]{\lineheight{1.25}\smash{\begin{tabular}[t]{l}$\Gamma_{23}$\end{tabular}}}}%
    \put(0.97045986,0.39981741){\color[rgb]{0.07843137,0.42352941,0}\makebox(0,0)[lt]{\lineheight{1.25}\smash{\begin{tabular}[t]{l}$\Gamma_{34}$\end{tabular}}}}%
    \put(0,0){\includegraphics[width=\unitlength,page=2]{Figure2_new.pdf}}%
    \put(0.15814745,0.12186497){\makebox(0,0)[lt]{\lineheight{1.25}\smash{\begin{tabular}[t]{l}$D_{\phi_2}=\{\bx\in D:\phi_2(\bx)<0\}$\end{tabular}}}}%
    \put(0,0){\includegraphics[width=\unitlength,page=3]{Figure2_new.pdf}}%
    \put(0.00553813,0.39883207){\makebox(0,0)[lt]{\lineheight{1.25}\smash{\begin{tabular}[t]{l}$D_{\phi_1}=\{\bx\in D:\phi_1(\bx)<0\}$\end{tabular}}}}%
    \put(0.48762067,0.39944766){\color[rgb]{0,0,0.6}\makebox(0,0)[lt]{\lineheight{1.25}\smash{\begin{tabular}[t]{l}$\Gamma_{14}$\end{tabular}}}}%
    \put(0.79174962,0.10600129){\color[rgb]{0,0,0.6}\makebox(0,0)[lt]{\lineheight{1.25}\smash{\begin{tabular}[t]{l}$\Gamma_{14}$\end{tabular}}}}%
    \put(0.52730176,0.10975432){\color[rgb]{1,0.78039216,0.18823529}\makebox(0,0)[lt]{\lineheight{1.25}\smash{\begin{tabular}[t]{l}$\Gamma_{12}$\end{tabular}}}}%
    \put(0.91500111,0.11130348){\color[rgb]{0.07843137,0.42352941,0}\makebox(0,0)[lt]{\lineheight{1.25}\smash{\begin{tabular}[t]{l}$\Gamma_{34}$\end{tabular}}}}%
  \end{picture}%
\endgroup%

%% file: Et_two_ls_tex.tex
\begingroup%
  \makeatletter%
  \providecommand\color[2][]{%
    \errmessage{(Inkscape) Color is used for the text in Inkscape, but the package 'color.sty' is not loaded}%
    \renewcommand\color[2][]{}%
  }%
  \providecommand\transparent[1]{%
    \errmessage{(Inkscape) Transparency is used (non-zero) for the text in Inkscape, but the package 'transparent.sty' is not loaded}%
    \renewcommand\transparent[1]{}%
  }%
  \providecommand\rotatebox[2]{#2}%
  \newcommand*\fsize{\dimexpr\f@size pt\relax}%
  \newcommand*\lineheight[1]{\fontsize{\fsize}{#1\fsize}\selectfont}%
  \ifx\svgwidth\undefined%
    \setlength{\unitlength}{227.69478607bp}%
    \ifx\svgscale\undefined%
      \relax%
    \else%
      \setlength{\unitlength}{\unitlength * \real{\svgscale}}%
    \fi%
  \else%
    \setlength{\unitlength}{\svgwidth}%
  \fi%
  \global\let\svgwidth\undefined%
  \global\let\svgscale\undefined%
  \makeatother%
  \begin{picture}(1,0.86602728)%
    \lineheight{1}%
    \setlength\tabcolsep{0pt}%
    \put(0,0){\includegraphics[width=\unitlength,page=1]{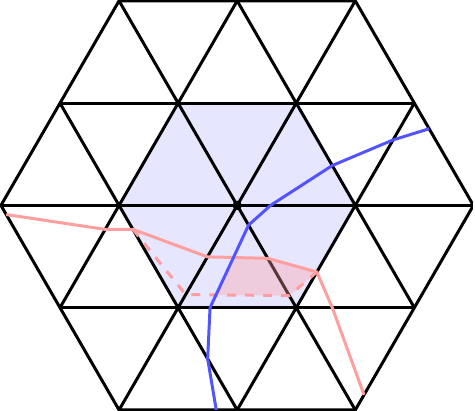}}%
    \put(0.40975576,0.45582536){\makebox(0,0)[lt]{\lineheight{1.25}\smash{\begin{tabular}[t]{l}$\boldsymbol{x}_w$\end{tabular}}}}%
    \put(0.40,0.605){\makebox(0,0)[lt]{\lineheight{1.25}\smash{\begin{tabular}[t]{l}${\rm supp}~w$\end{tabular}}}}%
    \put(0.49648009,0.26539625){\makebox(0,0)[lt]{\lineheight{1.25}\smash{\begin{tabular}[t]{l}$E_t$\end{tabular}}}}%
    \put(0.70053276,0.24373315){\color[rgb]{1,0.61960784,0.61960784}\makebox(0,0)[lt]{\lineheight{1.25}\smash{\begin{tabular}[t]{l}$\partial D_{\phi_1}$\end{tabular}}}}%
    \put(0.19465067,0.13823071){\color[rgb]{1,0.61960784,0.61960784}\makebox(0,0)[lt]{\lineheight{1.25}\smash{\begin{tabular}[t]{l}$D_{\phi_1}$\end{tabular}}}}%
    \put(0.68021143,0.5877673){\color[rgb]{0.32156863,0.33333333,1}\makebox(0,0)[lt]{\lineheight{1.25}\smash{\begin{tabular}[t]{l}$\partial D_{\phi_2}$\end{tabular}}}}%
    \put(0.8340944,0.35237241){\color[rgb]{0.32156863,0.33333333,1}\makebox(0,0)[lt]{\lineheight{1.25}\smash{\begin{tabular}[t]{l}$D_{\phi_2}$\end{tabular}}}}%
  \end{picture}%
\endgroup%

%% file: Unfitted_Discretisation_mapping_only_tex.tex
\begingroup%
  \makeatletter%
  \providecommand\color[2][]{%
    \errmessage{(Inkscape) Color is used for the text in Inkscape, but the package 'color.sty' is not loaded}%
    \renewcommand\color[2][]{}%
  }%
  \providecommand\transparent[1]{%
    \errmessage{(Inkscape) Transparency is used (non-zero) for the text in Inkscape, but the package 'transparent.sty' is not loaded}%
    \renewcommand\transparent[1]{}%
  }%
  \providecommand\rotatebox[2]{#2}%
  \newcommand*\fsize{\dimexpr\f@size pt\relax}%
  \newcommand*\lineheight[1]{\fontsize{\fsize}{#1\fsize}\selectfont}%
  \ifx\svgwidth\undefined%
    \setlength{\unitlength}{478.68660033bp}%
    \ifx\svgscale\undefined%
      \relax%
    \else%
      \setlength{\unitlength}{\unitlength * \real{\svgscale}}%
    \fi%
  \else%
    \setlength{\unitlength}{\svgwidth}%
  \fi%
  \global\let\svgwidth\undefined%
  \global\let\svgscale\undefined%
  \makeatother%
  \begin{picture}(1,0.36130061)%
    \lineheight{1}%
    \setlength\tabcolsep{0pt}%
    \put(0,0){\includegraphics[width=\unitlength,page=1]{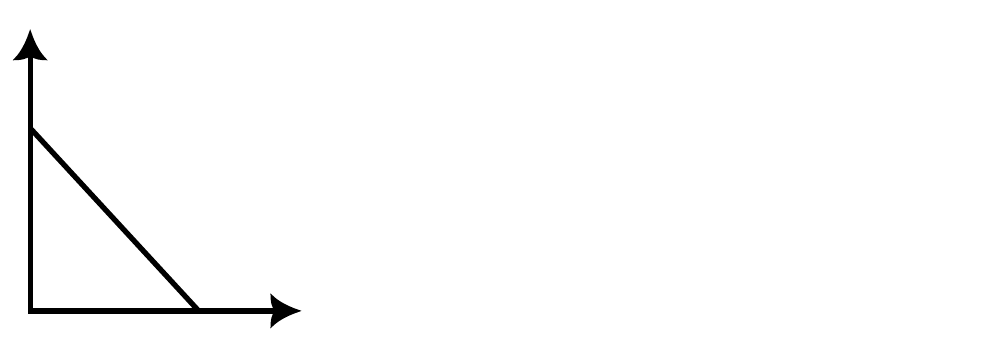}}%
    \put(0.31562297,0.04166542){\color[rgb]{0,0,0}\makebox(0,0)[lt]{\lineheight{1.25}\smash{\begin{tabular}[t]{l}$\xi_1$\end{tabular}}}}%
    \put(0,0){\includegraphics[width=\unitlength,page=2]{Unfitted_Discretisation_mapping_only.pdf}}%
    \put(0.38618491,0.22248104){\color[rgb]{0,0,0}\makebox(0,0)[lt]{\lineheight{1.25}\smash{\begin{tabular}[t]{l}$\boldsymbol{x}=T_\phi(\boldsymbol{\xi})$\end{tabular}}}}%
    \put(0.58496161,0.04169045){\color[rgb]{0,0,0}\makebox(0,0)[lt]{\lineheight{1.25}\smash{\begin{tabular}[t]{l}$\boldsymbol{q}_1$\end{tabular}}}}%
    \put(1.00807157,0.1756941){\color[rgb]{0,0,0}\makebox(0,0)[lt]{\lineheight{1.25}\smash{\begin{tabular}[t]{l}$\boldsymbol{q}_2$\end{tabular}}}}%
    \put(0.77164645,0.31991225){\color[rgb]{0,0,0}\makebox(0,0)[lt]{\lineheight{1.25}\smash{\begin{tabular}[t]{l}$\boldsymbol{q}_3$\end{tabular}}}}%
    \put(0.7686208,0.23756756){\color[rgb]{0,0,0}\makebox(0,0)[lt]{\lineheight{1.25}\smash{\begin{tabular}[t]{l}$\boldsymbol{v}_2(\phi)$\end{tabular}}}}%
    \put(0.84544814,0.1103513){\color[rgb]{0,0,0}\makebox(0,0)[lt]{\lineheight{1.25}\smash{\begin{tabular}[t]{l}$\boldsymbol{v}_1(\phi)$\end{tabular}}}}%
    \put(0.01571492,0.34249918){\color[rgb]{0,0,0}\makebox(0,0)[lt]{\lineheight{1.25}\smash{\begin{tabular}[t]{l}$\xi_2$\end{tabular}}}}%
    \put(0,0){\includegraphics[width=\unitlength,page=3]{Unfitted_Discretisation_mapping_only.pdf}}%
    \put(0.04391166,0.24010396){\color[rgb]{0,0,0}\makebox(0,0)[lt]{\lineheight{1.25}\smash{\begin{tabular}[t]{l}$(0,1)$\end{tabular}}}}%
    \put(0.19078914,0.01197272){\color[rgb]{0,0,0}\makebox(0,0)[lt]{\lineheight{1.25}\smash{\begin{tabular}[t]{l}$(1,0)$\end{tabular}}}}%
    \put(0.02673994,0.01197272){\color[rgb]{0,0,0}\makebox(0,0)[lt]{\lineheight{1.25}\smash{\begin{tabular}[t]{l}$(0,0)$\end{tabular}}}}%
  \end{picture}%
\endgroup%

%% file: split2d_tex.tex
\begingroup%
  \makeatletter%
  \providecommand\color[2][]{%
    \errmessage{(Inkscape) Color is used for the text in Inkscape, but the package 'color.sty' is not loaded}%
    \renewcommand\color[2][]{}%
  }%
  \providecommand\transparent[1]{%
    \errmessage{(Inkscape) Transparency is used (non-zero) for the text in Inkscape, but the package 'transparent.sty' is not loaded}%
    \renewcommand\transparent[1]{}%
  }%
  \providecommand\rotatebox[2]{#2}%
  \newcommand*\fsize{\dimexpr\f@size pt\relax}%
  \newcommand*\lineheight[1]{\fontsize{\fsize}{#1\fsize}\selectfont}%
  \ifx\svgwidth\undefined%
    \setlength{\unitlength}{366.25039889bp}%
    \ifx\svgscale\undefined%
      \relax%
    \else%
      \setlength{\unitlength}{\unitlength * \real{\svgscale}}%
    \fi%
  \else%
    \setlength{\unitlength}{\svgwidth}%
  \fi%
  \global\let\svgwidth\undefined%
  \global\let\svgscale\undefined%
  \makeatother%
  \begin{picture}(1,0.42592779)%
    \lineheight{1}%
    \setlength\tabcolsep{0pt}%
    \put(0,0){\includegraphics[width=\unitlength,page=1]{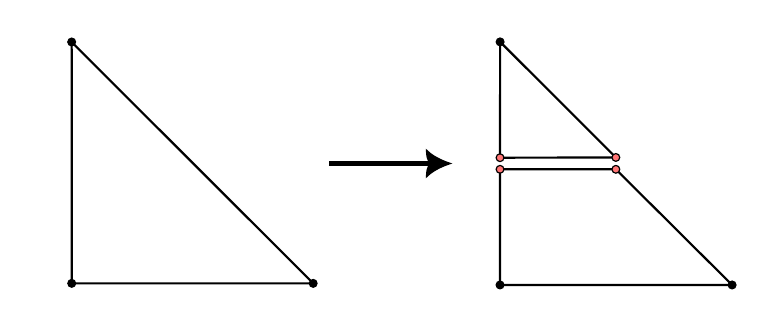}}%
    \put(0,0){\includegraphics[width=\unitlength,page=2]{split2d.pdf}}%
  \end{picture}%
\endgroup%

%% file: split3d_tex.tex
\begingroup%
  \makeatletter%
  \providecommand\color[2][]{%
    \errmessage{(Inkscape) Color is used for the text in Inkscape, but the package 'color.sty' is not loaded}%
    \renewcommand\color[2][]{}%
  }%
  \providecommand\transparent[1]{%
    \errmessage{(Inkscape) Transparency is used (non-zero) for the text in Inkscape, but the package 'transparent.sty' is not loaded}%
    \renewcommand\transparent[1]{}%
  }%
  \providecommand\rotatebox[2]{#2}%
  \newcommand*\fsize{\dimexpr\f@size pt\relax}%
  \newcommand*\lineheight[1]{\fontsize{\fsize}{#1\fsize}\selectfont}%
  \ifx\svgwidth\undefined%
    \setlength{\unitlength}{366.25169649bp}%
    \ifx\svgscale\undefined%
      \relax%
    \else%
      \setlength{\unitlength}{\unitlength * \real{\svgscale}}%
    \fi%
  \else%
    \setlength{\unitlength}{\svgwidth}%
  \fi%
  \global\let\svgwidth\undefined%
  \global\let\svgscale\undefined%
  \makeatother%
  \begin{picture}(1,0.42592936)%
    \lineheight{1}%
    \setlength\tabcolsep{0pt}%
    \put(0,0){\includegraphics[width=\unitlength,page=1]{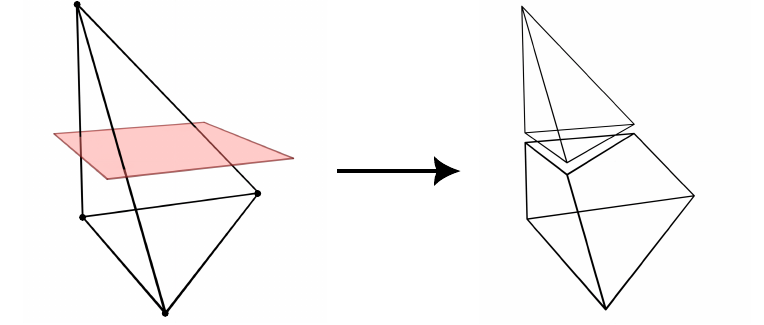}}%
    \put(0,0){\includegraphics[width=\unitlength,page=2]{split3d.pdf}}%
  \end{picture}%
\endgroup%

%% file: simplexify2d_tex.tex
\begingroup%
  \makeatletter%
  \providecommand\color[2][]{%
    \errmessage{(Inkscape) Color is used for the text in Inkscape, but the package 'color.sty' is not loaded}%
    \renewcommand\color[2][]{}%
  }%
  \providecommand\transparent[1]{%
    \errmessage{(Inkscape) Transparency is used (non-zero) for the text in Inkscape, but the package 'transparent.sty' is not loaded}%
    \renewcommand\transparent[1]{}%
  }%
  \providecommand\rotatebox[2]{#2}%
  \newcommand*\fsize{\dimexpr\f@size pt\relax}%
  \newcommand*\lineheight[1]{\fontsize{\fsize}{#1\fsize}\selectfont}%
  \ifx\svgwidth\undefined%
    \setlength{\unitlength}{366.25039889bp}%
    \ifx\svgscale\undefined%
      \relax%
    \else%
      \setlength{\unitlength}{\unitlength * \real{\svgscale}}%
    \fi%
  \else%
    \setlength{\unitlength}{\svgwidth}%
  \fi%
  \global\let\svgwidth\undefined%
  \global\let\svgscale\undefined%
  \makeatother%
  \begin{picture}(1,0.42592779)%
    \lineheight{1}%
    \setlength\tabcolsep{0pt}%
    \put(0,0){\includegraphics[width=\unitlength,page=1]{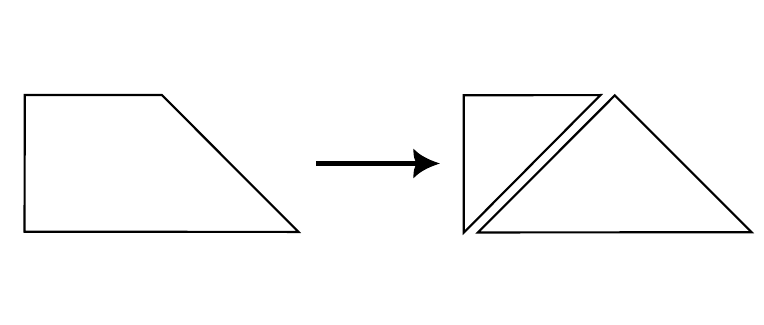}}%
  \end{picture}%
\endgroup%

%% file: simplexify3d_tex.tex
\begingroup%
  \makeatletter%
  \providecommand\color[2][]{%
    \errmessage{(Inkscape) Color is used for the text in Inkscape, but the package 'color.sty' is not loaded}%
    \renewcommand\color[2][]{}%
  }%
  \providecommand\transparent[1]{%
    \errmessage{(Inkscape) Transparency is used (non-zero) for the text in Inkscape, but the package 'transparent.sty' is not loaded}%
    \renewcommand\transparent[1]{}%
  }%
  \providecommand\rotatebox[2]{#2}%
  \newcommand*\fsize{\dimexpr\f@size pt\relax}%
  \newcommand*\lineheight[1]{\fontsize{\fsize}{#1\fsize}\selectfont}%
  \ifx\svgwidth\undefined%
    \setlength{\unitlength}{366.25169649bp}%
    \ifx\svgscale\undefined%
      \relax%
    \else%
      \setlength{\unitlength}{\unitlength * \real{\svgscale}}%
    \fi%
  \else%
    \setlength{\unitlength}{\svgwidth}%
  \fi%
  \global\let\svgwidth\undefined%
  \global\let\svgscale\undefined%
  \makeatother%
  \begin{picture}(1,0.42592936)%
    \lineheight{1}%
    \setlength\tabcolsep{0pt}%
    \put(0,0){\includegraphics[width=\unitlength,page=1]{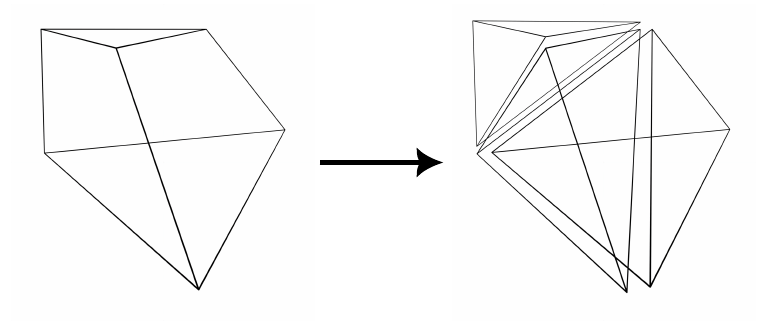}}%
  \end{picture}%
\endgroup%

%% file: rec_cutter_tex.tex
\begingroup%
  \makeatletter%
  \providecommand\color[2][]{%
    \errmessage{(Inkscape) Color is used for the text in Inkscape, but the package 'color.sty' is not loaded}%
    \renewcommand\color[2][]{}%
  }%
  \providecommand\transparent[1]{%
    \errmessage{(Inkscape) Transparency is used (non-zero) for the text in Inkscape, but the package 'transparent.sty' is not loaded}%
    \renewcommand\transparent[1]{}%
  }%
  \providecommand\rotatebox[2]{#2}%
  \newcommand*\fsize{\dimexpr\f@size pt\relax}%
  \newcommand*\lineheight[1]{\fontsize{\fsize}{#1\fsize}\selectfont}%
  \ifx\svgwidth\undefined%
    \setlength{\unitlength}{616.92588025bp}%
    \ifx\svgscale\undefined%
      \relax%
    \else%
      \setlength{\unitlength}{\unitlength * \real{\svgscale}}%
    \fi%
  \else%
    \setlength{\unitlength}{\svgwidth}%
  \fi%
  \global\let\svgwidth\undefined%
  \global\let\svgscale\undefined%
  \makeatother%
  \begin{picture}(1,1.25042461)%
    \lineheight{1}%
    \setlength\tabcolsep{0pt}%
    \put(0,0){\includegraphics[width=\unitlength,page=1]{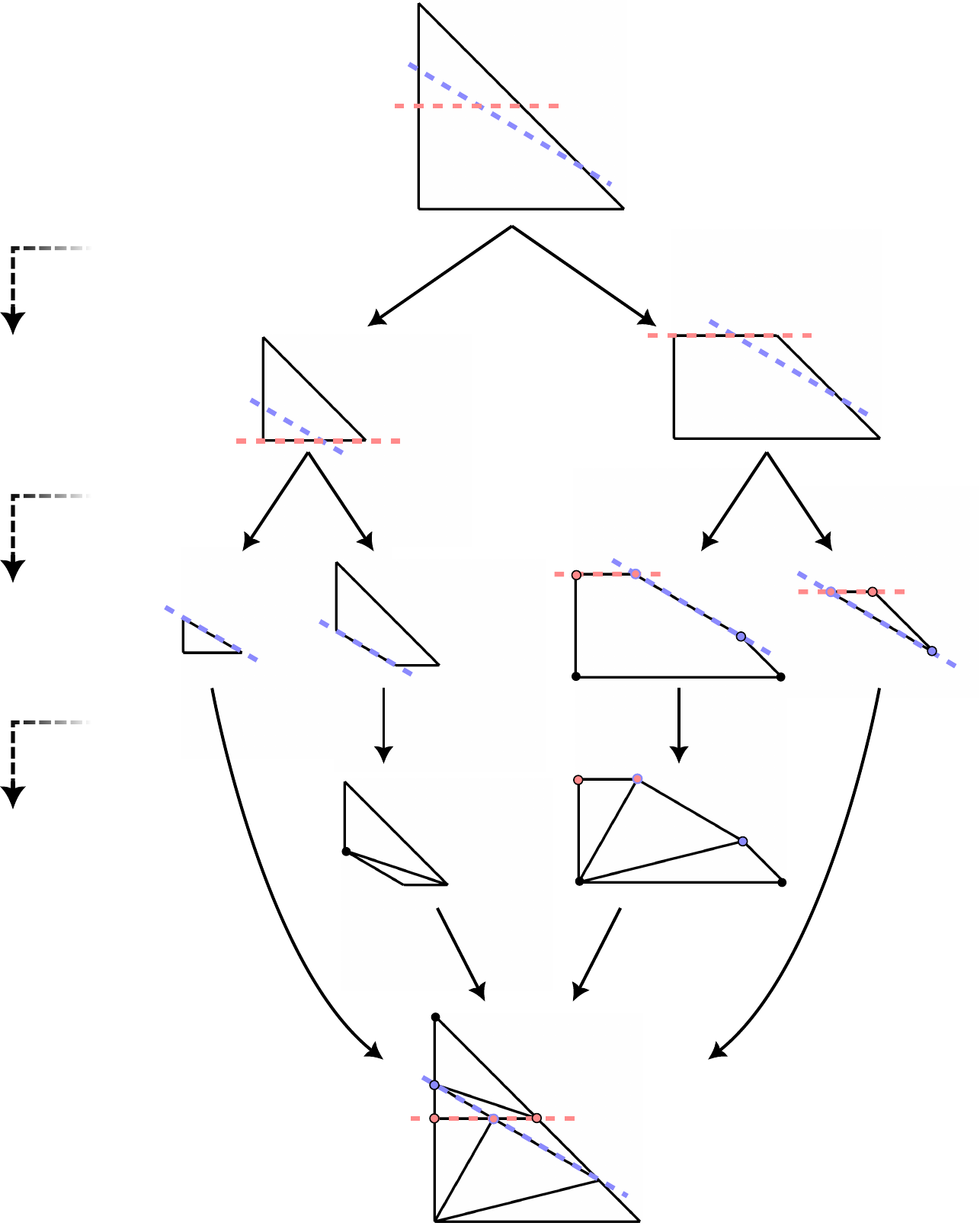}}%
    \put(0.02764335,0.95606986){\color[rgb]{1,0.54117647,0.54509804}\makebox(0,0)[lt]{\lineheight{1.25}\smash{\begin{tabular}[t]{l}\texttt{Split} with $\phi_1$\end{tabular}}}}%
    \put(0.03033534,0.70381167){\color[rgb]{0.54509804,0.54117647,0.99215686}\makebox(0,0)[lt]{\lineheight{1.25}\smash{\begin{tabular}[t]{l}\texttt{Split} with $\phi_2$\end{tabular}}}}%
    \put(0.02764335,0.47583326){\color[rgb]{0,0,0}\makebox(0,0)[lt]{\lineheight{1.25}\smash{\begin{tabular}[t]{l}\texttt{Simplexify}\end{tabular}}}}%
    \put(0,0){\includegraphics[width=\unitlength,page=2]{rec_cutter.pdf}}%
    \put(0.02764335,0.20594929){\color[rgb]{0,0,0}\makebox(0,0)[lt]{\lineheight{1.25}\smash{\begin{tabular}[t]{l}Resulting cut cell\end{tabular}}}}%
    \put(0,0){\includegraphics[width=\unitlength,page=3]{rec_cutter.pdf}}%
  \end{picture}%
\endgroup%

%% file: rec_cutter_facet_tex.tex
\begingroup%
  \makeatletter%
  \providecommand\color[2][]{%
    \errmessage{(Inkscape) Color is used for the text in Inkscape, but the package 'color.sty' is not loaded}%
    \renewcommand\color[2][]{}%
  }%
  \providecommand\transparent[1]{%
    \errmessage{(Inkscape) Transparency is used (non-zero) for the text in Inkscape, but the package 'transparent.sty' is not loaded}%
    \renewcommand\transparent[1]{}%
  }%
  \providecommand\rotatebox[2]{#2}%
  \newcommand*\fsize{\dimexpr\f@size pt\relax}%
  \newcommand*\lineheight[1]{\fontsize{\fsize}{#1\fsize}\selectfont}%
  \ifx\svgwidth\undefined%
    \setlength{\unitlength}{616.92588025bp}%
    \ifx\svgscale\undefined%
      \relax%
    \else%
      \setlength{\unitlength}{\unitlength * \real{\svgscale}}%
    \fi%
  \else%
    \setlength{\unitlength}{\svgwidth}%
  \fi%
  \global\let\svgwidth\undefined%
  \global\let\svgscale\undefined%
  \makeatother%
  \begin{picture}(1,1.26880574)%
    \lineheight{1}%
    \setlength\tabcolsep{0pt}%
    \put(0,0){\includegraphics[width=\unitlength,page=1]{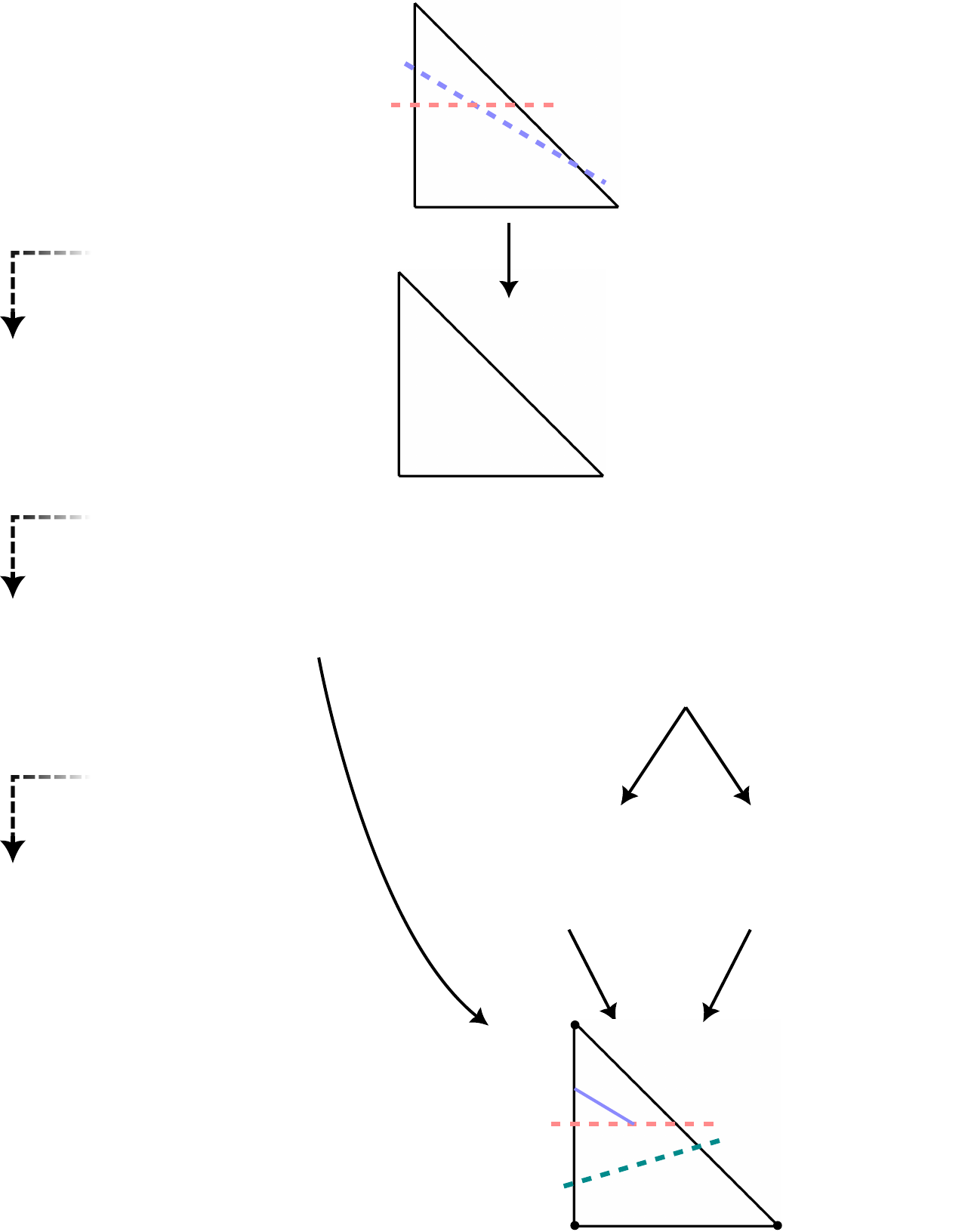}}%
    \put(0.02764335,0.96715675){\color[rgb]{0.54509804,0.54117647,0.99215686}\makebox(0,0)[lt]{\lineheight{1.25}\smash{\begin{tabular}[t]{l}Compute $\mathcal{F}_0$ with $\phi_2$\end{tabular}}}}%
    \put(0.03033534,0.69544747){\color[rgb]{1,0.54117647,0.54509804}\makebox(0,0)[lt]{\lineheight{1.25}\smash{\begin{tabular}[t]{l}\texttt{Split} with $\phi_1$\end{tabular}}}}%
    \put(0.02764335,0.4309985){\color[rgb]{0,0.54117647,0.54509804}\makebox(0,0)[lt]{\lineheight{1.25}\smash{\begin{tabular}[t]{l}\texttt{Split} with $\phi_3$\end{tabular}}}}%
    \put(0,0){\includegraphics[width=\unitlength,page=2]{rec_cutter_facet.pdf}}%
    \put(0.02764335,0.20974215){\color[rgb]{0,0,0}\makebox(0,0)[lt]{\lineheight{1.25}\smash{\begin{tabular}[t]{l}Resulting facet\end{tabular}}}}%
    \put(0,0){\includegraphics[width=\unitlength,page=3]{rec_cutter_facet.pdf}}%
  \end{picture}%
\endgroup%